\documentclass[11pt]{article}

\RequirePackage{amsthm,amsmath,amsfonts,amssymb}
\RequirePackage[colorlinks,citecolor=blue,urlcolor=blue]{hyperref}
\RequirePackage{graphicx}
\RequirePackage{tikz}
	\usetikzlibrary{calc}
\RequirePackage[a4paper, left=3cm,right=3cm,top=2.5cm,bottom=2.5cm]{geometry}
\usepackage{orcidlink}

\numberwithin{equation}{section}
\theoremstyle{plain}
\newtheorem{theorem}{Theorem}
\newtheorem*{theorem*}{Theorem}
\newtheorem*{corollary*}{Corollary}
\newtheorem{lemma}[theorem]{Lemma}
\newtheorem{proposition}[theorem]{Proposition}
\newtheorem{corollary}[theorem]{Corollary}
\newtheorem{conjecture}[theorem]{Conjecture}

\numberwithin{theorem}{section}
\theoremstyle{remark}
\newtheorem*{remark}{Remark}
\newtheorem{definition}{Definition}
\newtheorem*{definition*}{Definition}

\numberwithin{definition}{subsection}
\def\R{\mathbb{R}}
\def\N{\mathbb{N}}
\def\Z{\mathbb{Z}}

\def\C{\mathbb{C}}
\def\P{\mathbb{P}}
\def\E{\mathbb{E}}

\def\1{\mathbbold{1}}

\def\d{\mathrm{d}}
\def\pd{\partial}

\def\rdown#1{\left\lfloor #1 \right\rfloor}
\def\D{\mathbb{D}}

\def\T{\mathbb{T}}

\def\cparam{\mathbf{c}}

\def\cadlag{c\`adl\`ag }
\def\couplingfailure{\tau_{\mathrm{coupling}}}
\def\Phiale{\Phi^{\mathrm{ALE}}}
\def\Phiaux{\Phi^{*}}
\def\thetaale{\theta^{\mathrm{ALE}}} 
\def\thetaaux{\theta^{*}}
\def\preimageale{\phi}
\def\preimageaux{\bar\phi}
\def\preimagemulti{\phi}
\def\fale#1{f^{\mathrm{ALE}}_{#1}} 
\def\faux#1{f^{*}_{#1}} 
\def\Psiale{\Psi}
\def\Psiaux{\overline{\Psi}}
\def\Phimulti{\Phi^{\mathrm{multi}}}

\title{Tip growth in a strongly concentrated aggregation model follows local geodesics}
\author{Frankie Higgs\thanks{\href{mailto://frankiehiggs@gmail.com}{frankiehiggs@gmail.com}}\;  \orcidlink{0000-0002-7300-8412} \\ University of Rome La Sapienza}

\begin{document}

\maketitle

\begin{abstract}
	We analyse the \emph{aggregate Loewner evolution} (ALE),
	introduced in 2018 by Sola, Turner and Viklund
	to generalise versions of diffusion limited aggregation (DLA)
	in the plane using complex analysis.
	They showed convergence of the ALE
	for certain parameters to a single growing slit.
	Started from a non-trivial initial configuration of $k$ needles
	and the same parameters,
	we show that the small-particle scaling limit of ALE
	is the \emph{Laplacian path model},
	introduced by Carleson and Makarov,
	in which the tips grow along geodesics towards $\infty$.
	
	Our proof involves analysis of Loewner's equation
	near its singular points,
	and we extend martingale methods to the backward equation,
	where what we have to control is non-adapted.
	Most conformal growth models introduce
	an extra regularisation factor
	to deal with the singularities in Loewner's equation
	at the sharp tips and right-angle bases of slit particles.
	As an intermediate step we prove a limit result for a model
	with no such regularisation factor,
	developing methods
	which should prove useful in analysing other
	weakly-regularised models with non-trivial limits.
\end{abstract}


\tableofcontents

\section{Introduction}\label{sec1}

We study two models of growth in the plane, identified with $\C$.
The aggregate Loewner evolution (ALE) model
was introduced in \cite{stv-ale}
as a version of diffusion limited aggregation (DLA)
which could take advantage of the natural isotropy of $\C$
and complex-analytic tools.
Via the Riemann mapping theorem,
we have a one-to-one correspondence between a certain family
of conformal maps and the collection of
compact simply-connected subsets of $\C$.
By composing simple conformal maps (corresponding to ``particles'')
we construct a map representing more complicated clusters.
The continuity of the one-to-one correspondence between maps
and clusters allows us to prove limit results about the cluster
by analysing the corresponding map,
so a powerful set of tools from complex analysis
becomes available for analysing growth models.

The Laplacian path model (LPM) was introduced in \cite{lpm}
as a model of tip growth,
in which $k$ simple curves attached to the unit disc
grow towards $\infty$ in the Riemann sphere.
The growth speeds are controlled by
a certain Laplacian field,
which allows the LPM to model ``screening'' effects,
in which the growth rate of each curve is affected by how much
its ``view of $\infty$'' is blocked by other curves.
The strength of the influence of this field is parameterised
by a single $\eta \in \R$.

We will prove that the ALE model of particle aggregation converges,
in the limit as individual particles become small,
to the LPM.
This is an extension of the result of \cite{stv-ale},
that ALE started from a disc with no slits initially attached
converges to an LPM with one arm, i.e.\ to a straight line.
Our methods involve an explicit analysis of the tips of particles,
where usual martingale methods are difficult to apply
due to singularities in \emph{Loewner's equation}
which describes both the ALE and LPM.

\begin{figure}[p]
	\centering
	\includegraphics[width=0.49\textwidth, trim=70 40 60 50, clip]{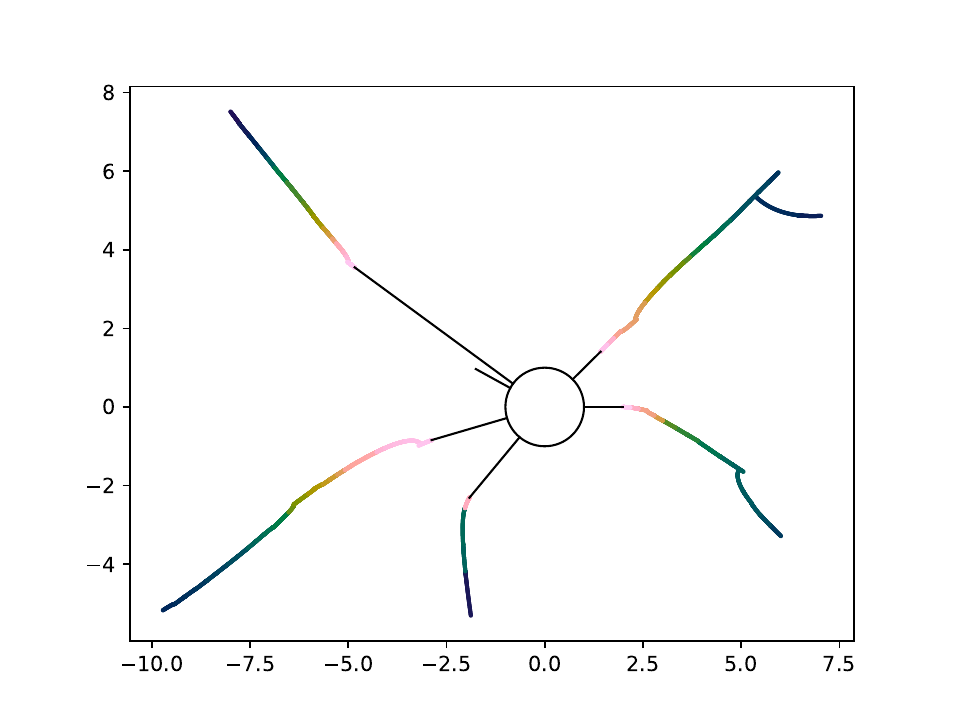}
	\includegraphics[width=0.49\textwidth, trim=105 40 95 50, clip]{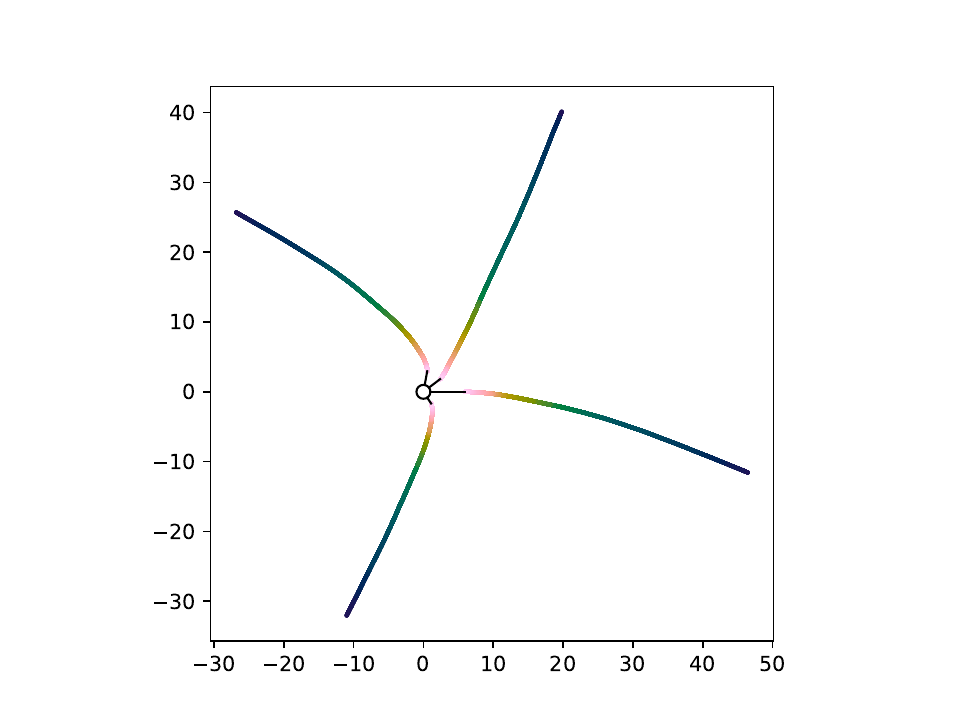}
	\includegraphics[angle=90, width=0.49\textwidth, trim=140 40 130 50, clip]{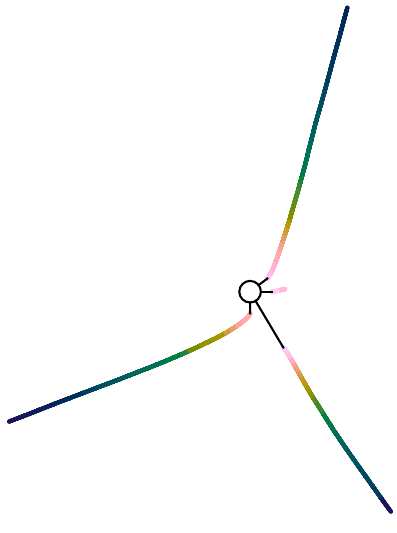}
	\includegraphics[angle=90, width=0.49\textwidth, trim=130 40 120 50, clip]{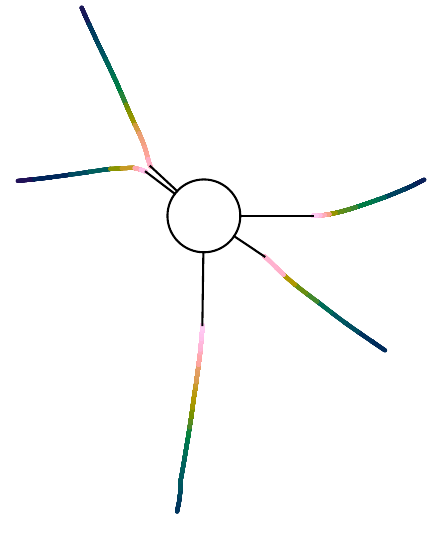}
	\includegraphics[width=0.49\textwidth, trim=65 60 70 70, clip]{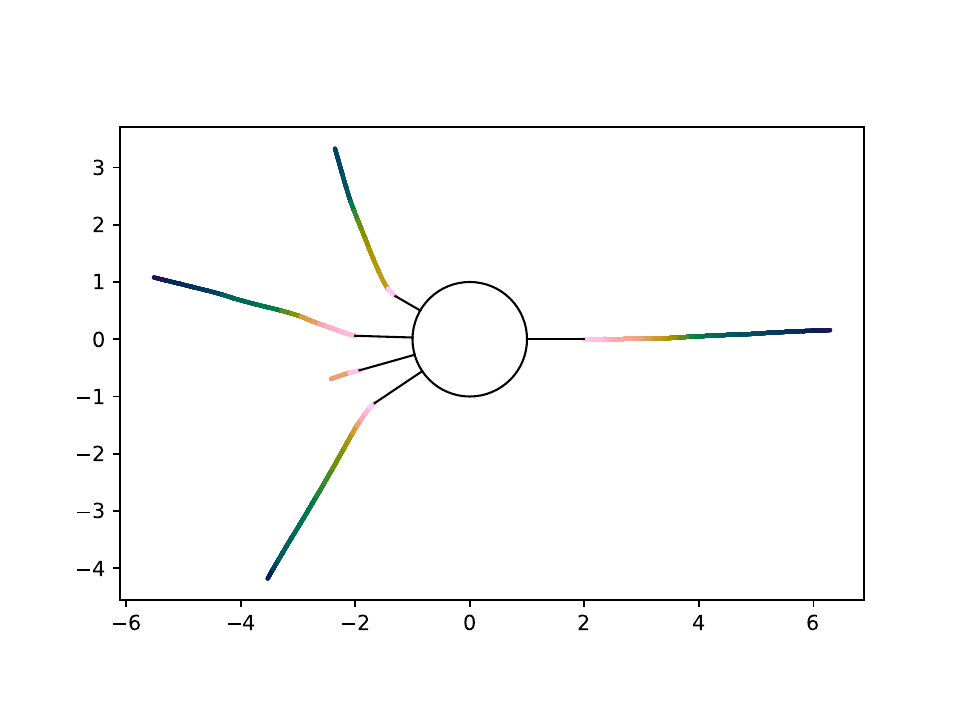}
	\includegraphics[width=0.49\textwidth, trim=70 55 60 65, clip]{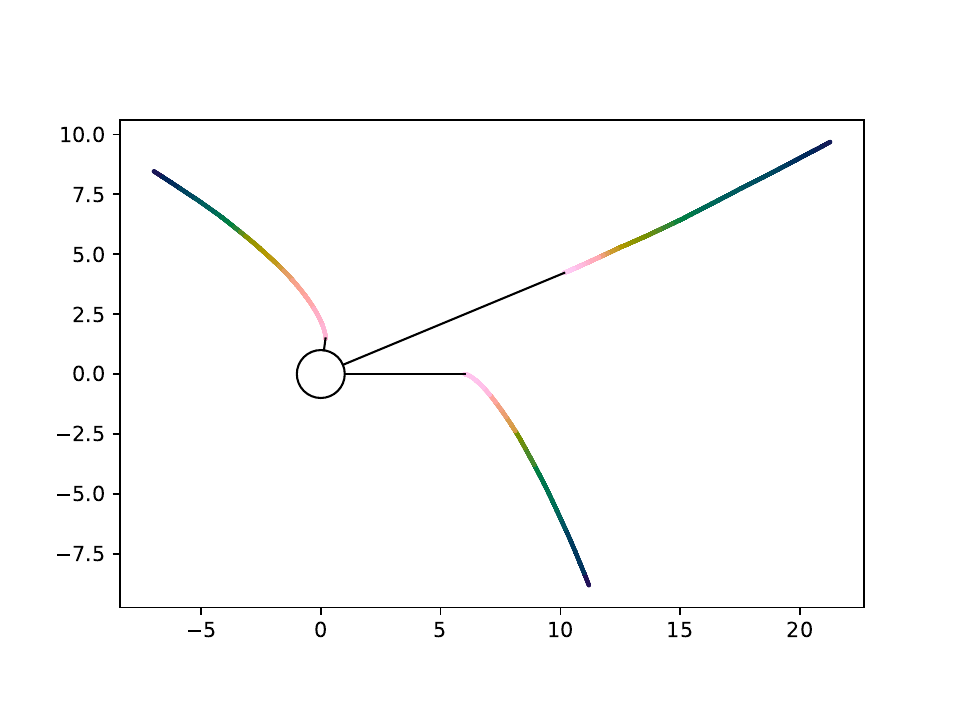}
	\caption{\label{fig:ale-simulation}
		Simulations of the aggregate Loewner evolution
		started from a non-trivial configuration of arms.
		The initial cluster is drawn with thin black lines,
		and the attached particles are coloured by arrival time
		(the lighter particles arrived earlier).
		In the top-left diagram $\sigma$ is large enough that some particles
		are not attached at the tips of the existing slits
		and we see ``branching'' behaviour
		which does not occur in the limiting regime.
		In the three diagrams on the left one slit seems to grow much slower
		than the rest,
		agreeing with the observation in \cite{lpm} that there is competition
		between the arms in the Laplacian path model,
		with only a certain number (depending on $\eta$)
		likely to survive (i.e.\ grow to a length proportional to the diameter
		of the entire cluster) as $T \to \infty$.
		The code used for the simulations is available at
		\href{https://github.com/frankiehiggs/ALE-from-slits}{https://github.com/frankiehiggs/ALE-from-slits}
	}
\end{figure}

\section{Preliminaries}

\subsection{Loewner's equation}

Loewner's equation describes growing sets in the complex plane.
The aggregate Loewner evolution (ALE) model
of particle growth is defined by composing conformal maps,
but we can also view it as a solution to Loewner's equation
with a certain driving function.
Loewner's equation also defines the claimed limit:
the Laplacian path model (LPM).

Loewner's equation encodes families of growing sets
by measures on a cylinder.
Continuity properties of the encoding
allow us to show the sets are close
if the corresponding measures are close.

\begin{definition}
	\label{def:loewner-maps}
	Let $\D = \{ z \in \C : |z| < 1 \}$
	be the open unit disc,
	and $\Delta = \C_{\infty} \setminus \overline{\D}$
	the open \emph{exterior disc},
	the complement of the closed unit disc in the Riemann sphere.
	Let $\T = \partial \D$ denote their common boundary,
	the unit circle.
	
	Given a compact, simply connected $K \subseteq \C$
	with $0 \in K$ (and $K \not= \{0\}$),
	by the Riemann mapping theorem there is a unique conformal map
	$f_K : \Delta \to \Delta \setminus K$
	with $f_K(\infty) = \infty$
	and $f_K(z) = e^c z + O(1)$
	for some $c = c(K) \in \R$ as $z \to \infty$.
	We call $c(K)$ the \emph{(logarithmic) capacity} of $K$.
\end{definition}

If $K$ is a strict superset of $\overline{\D}$, then $c(K) > 0$.%
\footnote{This follows from the fact that $c(K) = \int_0^1 \log |f_K(e^{2\pi i t})|\,\d t$.}
Another useful property of logarithmic capacity
is \emph{additivity}:
if $f_K$ can be written $f_{K_1} \circ f_{K_2}$,
then $c(K) = c(K_1) + c(K_2)$.

We can encode certain families of growing compact sets $(K_t)_{t \geq 0}$
using \emph{Loewner's equation}.

\begin{definition}
	\label{def:loewners-equation}
	Let $\xi : [0, T] \to \R$ be a \cadlag function%
	\footnote{A function $f : [0,T] \to \R$ is said to be \cadlag
		(right-continuous with left limits) if,
		for every $x \in [0,T]$ and every decreasing sequence $x_n \downarrow x$
		and increasing sequence $y_n \uparrow x$,
		$f(x_n) \to f(x)$ as $n \to \infty$
		and $\lim_{n\to\infty}f(y_n)$ exists
		(but is not necessarily equal to $f(x)$).}.
	Then there is a unique solution $\varphi : [0,T]\times \Delta \to \Delta$ to {Loewner's equation},
	\begin{equation}
		\label{eq:loewner}
		\varphi_0(z) = z,
		\quad
		\frac{\partial}{\partial t} \varphi_t(z) 
		=
		\varphi_t'(z) z \frac{z + e^{i\xi_t}}{z - e^{i\xi_t}},
		\quad
		z \in \Delta,
	\end{equation}
	corresponding to a growing cluster
	$K_t := \C_\infty \setminus \varphi_t(\Delta)$
	which is parameterised by capacity,
	i.e. $c(K_t) = t$.
	The \emph{driving function} $\xi_t$
	encodes where the growth at time $t$ is located on the cluster boundary.
\end{definition}

For the existence and uniqueness of the solution to \eqref{eq:loewner},
see \cite[Theorem 3.4]{duren-univalent}
(note that every \cadlag function is piecewise continuous,
meeting the conditions in \cite{duren-univalent}).

We can generalise \eqref{eq:loewner} to encode clusters growing
at arbitrary parts of their boundary, specified by a measure:
\begin{definition}
	\label{def:loewner-with-measure}
	Given a family of probability measures on $\T$, $(\mu_t)_{t \geq 0}$,
	subject to measurability conditions on $t \mapsto \mu_t$
	there is a unique solution to Loewner's equation
	\begin{align}
		\label{eq:loewner-with-measure}
		\varphi_0(z) = z,
		\quad
		\frac{\partial}{\partial t} \varphi_t(z) 
		=
		\varphi_t'(z) \int_\T z \frac{z + e^{i\theta}}{z - e^{i\theta}}\,\d \mu_t(\theta),
		\quad
		z \in \Delta.
	\end{align}
\end{definition}

It is a simple generalisation to start Loewner's equation from
a non-trivial initial condition $\varphi_0$,
provided $\varphi_0$ is a conformal map of the form $f_K$ for some $K$ as above.

\subsection{Particle aggregation models}
\label{sec:particle-aggregation}

We construct the aggregate Loewner evolution (ALE) cluster
by composing conformal maps.

\begin{definition}
	\label{def:particle-map}
	Let $\cparam > 0$,
	and define $d = d(\cparam) > 0$
	to be the unique length
	such that $\overline{\D}\cup (1,1+d]$
	has capacity $\cparam$.
	We can determine explicitly that $d$ satisfies $(d+2)^2/(d+1) = 4e^{\cparam}$,
	and asymptotically $d \sim 2\cparam^{1/2}$ as $\cparam \to 0$.
	
	Let $f^{\cparam}$ be the unique conformal map
	$\Delta \to \Delta \setminus (1,1+d]$
	with $f^{\cparam}(z) \sim e^{\cparam}z$ as $z \to \infty$.
	Given an angle $\theta \in \R$,
	we can define the rotated map
	\begin{align*}
		&f^{\theta,\cparam} : \Delta \to \Delta \setminus e^{i\theta}(1,1+d],\\
		&f^{\theta,\cparam}(z)
		=
		e^{i\theta} f^{\cparam}(e^{-i\theta}z).
	\end{align*}
	
	Given a sequence of angles
	$(\theta_n)_{n \geq 1}$
	and of capacities $(c_n)_{n \geq 1}$,
	let $f_j = f^{\theta_j,c_j}$
	and
	\begin{align*}
		\Phi_n := f_1 \circ f_2 \circ \dots \circ f_n.
	\end{align*}
	Define the $n$th cluster
	$K_n$ as the complement of $\Phi_n(\Delta)$,
	so
	\[
	\Phi_n : \Delta \to \C_{\infty} \setminus K_n.
	\]
	Note that the total capacity is
	$c(K_n) = \sum_{k=1}^{n} c_k$.
\end{definition}

Any cluster constructed this way can be encoded by Loewner's equation.
For $t \geq 0$ let $n_t = \inf\{ n \geq 1 : \sum_{k=1}^{n} c_k > t \}$ (if $c_n = \cparam$ for all $n$, then $n_t = \rdown{t/\cparam}+1$).\footnote{If $\sum_{k=1}^\infty c_k < \infty$, then we can solve \eqref{eq:loewner} up to a finite time.}
Let $\xi_t := \theta_{n_t}$,
and let $(\varphi_t)_{t \geq 0}$ be the solution to \eqref{eq:loewner}
with driving measure $(\xi_t)_{t \geq 0}$.
Then for all $n$, $\varphi_{c(K_n)} = \Phi_n$.

\subsection{Aggregate Loewner evolution}

If $(\theta_n)_{n \geq 1}$
and $(c_n)_{n \geq 1}$
are stochastic processes,
then $(K_n)_{n \geq 1}$
as constructed in
Section~\ref{sec:particle-aggregation}
is a stochastic process
in the space of compact subsets of $\C$.
The aggregate Loewner evolution ALE($\alpha,\eta,\sigma$),
introduced in \cite{stv-ale},
is one such process.

\begin{definition}
	\label{def:ALE}
	We define the ALE inductively:
	for $n \geq 0$
	we choose $\theta_{n+1}$,
	conditional on $\theta_1, \dots, \theta_n$,
	according to the probability density function
	\begin{align*}
		h_{n+1}(\theta)
		=
		\frac{1}{Z_n}
		\left|
		\Phi_n'\left(
		e^{\sigma + i\theta}
		\right)
		\right|^{-\eta},\,
		\theta \in (-\pi, \pi],
	\end{align*}
	where $Z_n := \int_{\T} | \Phi_n'(e^{\sigma + i\theta})|^{-\eta}\,\d \theta$
	is a normalising factor.
	The particle capacities can also be modified.
	For $n \geq 0$ let
	\[
	c_{n+1} = \cparam | \Phi_n'(e^{\sigma + i\theta_{n+1}})|^{-\alpha}.
	\]
\end{definition}

Several simulations of the ALE in the setting of Theorem~\ref{thm:ale-lpm}
are shown in Figure~\ref{fig:ale-simulation}.

\begin{remark}
	We have introduced three parameters: $\alpha \in \R$, $\eta \in \R$,
	and $\sigma > 0$. In this paper we will work with $\alpha = 0$,
	which means that all particles have the same capacity.
	After the statement of Theorem~\ref{thm:ale-lpm}
	we will discuss the effect of changing $\alpha$.
	
	If $\sigma = 0$
	then the poles and zeroes of $\Phi_n'$
	on $\partial \Delta = \T$
	mean that the measure $h_{n+1}$
	would not necessarily be well-defined.
	We will take $\sigma > 0$ but with $\sigma \to 0$
	as $\cparam \to 0$.
	
	The other parameter, $\eta$,
	controls the influence of \emph{harmonic measure}.
	If $\eta = 0$
	then $(\theta_n)_{n\geq 1}$
	is a sequence of independent  random variables uniformly distributed on $(-\pi,\pi]$,
	corresponding to the Hastings--Levitov model \cite{hl-model}.
	
	We can view $|(\Phi_{n}^{-1})'(z)|$
	as the ``density of harmonic measure''
	at $z \in \partial K_n$,
	and so when $\eta > 0$,
	$h_{n+1}$ is concentrated on areas of high harmonic measure,
	and when $\eta < 0$ it is concentrated
	on areas of low harmonic measure.
\end{remark}

\begin{remark}
	In \cite{small-eta},
	a phase transition was proved for negative $\eta$: for $\eta < -2$
	and sufficiently small $\sigma$
	the ALE cluster converges to a \emph{Schramm--Loewner evolution}.
	
	The authors of \cite{stv-ale} showed for $\eta > 1$
	that the scaling limit of an ALE with initial cluster $\overline{\D}$
	is a single straight slit.
	In fact they showed a stronger degeneracy result:
	with probability tending to 1 as $\cparam \to 0$,
	each particle attaches near the tip of its immediate predecessor.
	Theorem~\ref{thm:ale-lpm} extends this result to the case where $K_0$
	is the union of $\overline{\D}$ with more than one slit.
\end{remark}

A series of videos illustrating the effects of $\eta$ and $\alpha$
is available at \href{https://www.youtube.com/playlist?list=PLiaV5rk6Gk7pCjwGQaVOZ1edfbU3xK6td}{this link}.\footnote{\href{https://www.youtube.com/playlist?list=PLiaV5rk6Gk7pCjwGQaVOZ1edfbU3xK6td}{\texttt{https://www.youtube.com/playlist?list=PLiaV5rk6Gk7pCjwGQaVOZ1edfbU3xK6td}}}

\subsection{Laplacian path model}
\label{sec:lpm-definition}

The Laplacian path model (LPM)
was defined in 2002 by Carleson and Makarov \cite{lpm}
to generalise several models of needle-like growth in mathematical physics.
We define the Laplacian path model as a growing set
in the complex plane using Loewner's equation \eqref{eq:loewner-with-measure}:
the LPM is encoded by an initial configuration of $k$ arms
$K_0$ (which uniquely determines $\varphi_0$),
and a driving measure $(\mu_t)_{t \geq 0}$.

\begin{definition}
	\label{def:LPM}
	Let $K_0$ be the union of rays of the form
	$e^{i\theta_j} (1, 1+d_j]$,
	where the $\theta_j$ are distinct angles in $(-\pi,\pi]$
	and $d_j > 0$
	for $j = 1, \dots, k$.
	Let $\Delta_0 = \Delta \setminus K_0$,
	and $\Phi_0^{\mathrm{LPM}}$ the unique conformal map
	$\Phi_0^{\mathrm{LPM}} : \Delta \to \Delta_0$
	satisfying $\Phi_0^{\mathrm{LPM}}(z) = e^{c_0} z + O(1)$
	as $z \to \infty$ for some positive $c_0$.
	
	We define the LPM cluster with parameter $\eta$,
	which has $k$ growing slits whose
	tips at time $t > 0$ are at $a^j_t \in \Delta \setminus K_0$,
	via Loewner's equation
	with initial condition $\varphi_0(z) = \Phi_0^{\mathrm{LPM}}(z)$.
	We have the driving measure
	\begin{align}
		\label{lpm-definition}
		\mu^{\mathrm{LPM}}_t = \sum_{j=1}^{k} p^j_t \delta_{\phi^j_t},
	\end{align}
	where $e^{i\phi^j_t}$ is
	the preimage of $a^j_t$ under $\Phi_t^{\mathrm{LPM}}$,
	and
	\begin{align}
		\label{lpm-weights}
		p^j_t := \frac{|(\Phi_t^{\mathrm{LPM}})''(e^{i\phi^j_t})|^{-\eta}}{Z_t},
	\end{align}
	and
	\begin{align}
		Z_t := \sum_{j=1}^{k} |(\Phi_t^{\mathrm{LPM}})''(e^{i\phi^j_t})|^{-\eta}.
	\end{align}
\end{definition}

\begin{remark}
	We can think of the growth of each slit in the geodesic LPM at time $t$
	as being in the direction of the hyperbolic geodesic
	from $a^j_t$ to $\infty$ in $\Delta \setminus K_t$,
	at a speed proportional to
	$|(\Phi_t^{\mathrm{LPM}})''(e^{i\phi^j_t})|^{-(\eta-1)}$.
\end{remark}
\begin{remark}
	The weight given to the $j$th slit by \eqref{lpm-weights}
	depends on the second derivative of $\Phi_t^{\mathrm{LPM}}$,
	while the probability of attaching near a given point
	in the ALE model depends on the first derivative of $\Phi_t^{\mathrm{ALE}}$.
\end{remark}
\begin{remark}
	In \cite{lpm},
	Carleson and Makarov
	obtained the definition involving the second derivative
	by considering growth speeds
	proportional to the Laplacian field $\nabla G_t$,
	where $G_t$ is the Green's function of $\Delta \setminus K_t$,
	given by $G_t(z) = \log |(\Phi_t^{\mathrm{LPM}})'(z) |$.
	
	We could also derive $\mu_t^{\mathrm{LPM}}$
	by removing the regularisation from the ALE attachment measure.
	For $\sigma > 0$, let $\mu_t^\sigma$
	be the regularised ALE measure on $\T$
	corresponding to the LPM cluster at time $t$,
	i.e.\ the measure with density $h^\sigma(\theta) \propto (\Phi_t^{\mathrm{LPM}})'(e^{\sigma + i\theta})|^{-\eta}$.
	If $\eta > 1$ then as $\sigma \to 0$
	the density becomes concentrated around the zeroes of $(\Phi_t^{\mathrm{LPM}})'$
	at $\phi_t^j$ for $j = 1, \dots, k$,
	and $|(\Phi_t^{\mathrm{LPM}})'(e^{\sigma + i\phi^j_t})|
	\sim \sigma |(\Phi_t^{\mathrm{LPM}})''(e^{i\phi_t^j})|$.
	Hence
	$\mu_t^\sigma \to \mu_t^{\mathrm{LPM}}$ as $\sigma \to 0$
	for a fixed time $t$.
\end{remark}
\begin{remark}
	The $\eta$ in Definition~\ref{def:LPM} corresponds to $\eta-1$
	in the definition given in \cite{lpm};
	we have shifted it to match the $\eta$
	used in the definition of the ALE in \cite{stv-ale}.
\end{remark}
\begin{remark}
	For any configuration of $k$ slits,
	$K_0$ is a ``gearlike domain'' as defined in \cite{gearlike},
	and so Schwarz--Christoffel methods can be used to compute
	an expression for $\Phi_0^{\mathrm{LPM}}$.
	To compute the expression,
	the preimages of all the tips and bases of the slits
	must be determined,
	which can be done numerically using methods from \cite{gearlike}.
	We used this to compute the maps for arbitrary slit configurations
	in our simulation of the ALE shown in Figure~\ref{fig:ale-simulation}.
\end{remark}

Almost all known results about the Laplacian path model appear in
\cite{lpm} and \cite{selander-thesis}.
In \cite{lpm}, Carleson and Makarov defined both the geodesic LPM above
and the \emph{needle LPM},
in which the cluster is a collection of straight slits
which grow at speed proportional to $|(\Phi^{\mathrm{LPM}}_t)''(e^{i\phi^j_t})|^{-(\eta-1)}$.\\

Selander examined in his PhD thesis \cite{selander-thesis}
a version of the geodesic LPM in which the weights
$p^j_t$ are fixed constants,
and the needle LPM with growth speeds
proportional to harmonic measure at the tips,
which corresponds to $\eta = 3/2$.
Each of these can be viewed as a simplification of a ``non-branching DLA''.

For the needle LPM with $\eta = 3/2$,
Selander proved results about the stability of stationary solutions:
slit configurations in which the ratio of lengths remains constant
for all $t \geq 0$.
For his modified geodesic model,
he proved convergence as $t \to \infty$
to a stationary configuration determined by the weights,
when started from any initial configuration.\\

For a simplified ``chordal'' geodesic LPM,
in which finitely many curves grow from the tip of an infinite half-line,
Carleson and Makarov
proved a stability result:
starting the process from a two-arm configuration,
if $\eta < 2$ then both arms always survive,
while there are configurations in which only one arm survives
if $\eta > 2$.
We will use some of the techniques they developed
for the chordal case
when we analyse the ALE and non-chordal LPM.

For the unsimplified geodesic LPM
(i.e.\ the model defined by \eqref{lpm-definition}),
Carleson and Makarov also proved a stability result.
Theorem 5 of \cite{lpm} states that if $1 < \eta < \eta_c$
for a critical value $\eta_c = \frac{18}{3 + 4\log 2} \approx 3.11815$
then the symmetric three-armed cluster
is a \emph{stable} stationary solution,
i.e.\ the driving measure for an LPM started from an initial condition
which is \emph{close} to three symmetric arms in a certain sense
converges, as $t \to \infty$, to a symmetric driving measure
$\frac{1}{3}(\delta_{\theta} + \delta_{\theta+2\pi/3} + \delta_{\theta+4\pi/3})$.
Moreover, it is unstable for $\eta > \eta_c$,
in that small perturbations may lead to convergence to a different stationary solution.

They also proved limiting results for the simpler ``needle'' LPM,
and made a number of conjectures about the geodesic LPM.\\

Although little progress has been made since 2002
on analysing the geodesic LPM,
various modifications have been studied and applied.
In \cite{half-plane-finger-growth}, Gubiec and Szymczak
apply a similar construction to model finger growth in the half-plane.
We hope that relating the LPM to the more active area
of conformal growth models
will revive interest in the model and many of the
open problems stated in \cite{lpm}.\\

The LPM and similar models have also been used to model the development of cracks
in materials and formation of systems of rivers and streams.
In particular, the chordal geodesic LPM with two needles at the tip of a half-line
can be used to model the bifurcation of a stream.
Carleson and Makarov proved that the angle between the two resulting streams
in the chordal LPM
must be $2\pi / 5$ (also predicted by other authors using conformal mappings \cite{3.99, laplacian-networks}),
which agrees with the average angle of $72^\circ$
measured in a Florida stream system
\cite{streams, streams2}.

\subsection{Main result}

Our main theorem concerns convergence of the ALE as $\cparam \to 0$
to the Laplacian path model cluster started from the same initial conditions.

\begin{theorem}
	\label{thm:ale-lpm}
	For a fixed $T > 0$,
	let $(\Phi^{\mathrm{ALE}}_t)_{t\in[0,T]}$ be the ALE($\alpha,\eta, \sigma$) map
	started from the initial cluster
	$K_0 \cup \overline{\D}$, where
	$K_0 = \bigcup_{j=1}^{k} e^{i\phi^j_0}(1, 1+d_j]$
	for $d_j > 0$ and distinct $\phi^j_0 \in [0,2\pi)$.
	Let $\mu^{\mathrm{ALE}}_t$ be the driving measure for $\Phi^{\mathrm{ALE}}$,
	i.e. $\mu^{\mathrm{ALE}}_t = \delta_{\theta_{\rdown{t/\cparam}+1}}$.
	Let the parameters be $\alpha = 0$, $\eta > 1$ and $\sigma = \cparam^{\gamma}$
	for $\gamma = \max\left(\frac{2(\eta+2)}{\eta-1}, 8\right)$ (see the remark after Theorem~\ref{thm:ale-aux-global} for an explanation of each term).
	
	Let $(\Phi^{\mathrm{LPM}}_t)_{t \in [0,T]}$ be the map for the LPM
	started from the same initial conditions
	with the same parameter $\eta$,
	and let $\mu^{\mathrm{LPM}}_t$ be the driving measure
	$\mu^{\mathrm{LPM}}_t = \sum_j \bar p^j_t \delta_{\bar \phi^j_t}$.
	
	Then
	as $\cparam \to 0$, $\mu^{\mathrm{ALE}}_t \otimes m_{[0,T]}$
	converges in distribution to $\mu^{\mathrm{LPM}}_t \otimes m_{[0,T]}$
	as random elements of the space of measures $S = \T \times [0,T]$,
	where $m_{[0,T]}$ is normalised Lebesgue measure on $[0,T]$.
	
	In particular this means if $K^{\mathrm{ALE}}_t$ and $K^{\mathrm{LPM}}_t$
	are the respective clusters at time $t$,
	we have $K^{\mathrm{ALE}}_T \to K^{\mathrm{LPM}}_T$
	weakly as a random element of the space of compact subsets of $\C$
	(equipped with the Carath\'{e}odory topology).
\end{theorem}

\begin{remark}
	In the slit convergence result of \cite{stv-ale}, a different value of $\gamma$ is used,
	which is $\Theta((\eta-1)^{-2})$
	as $\eta \downarrow 1$,
	while the $\gamma$ given above is $\Theta((\eta-1)^{-1})$.
	The slit convergence result of \cite{stv-ale} starts with
	a single slit of capacity $\cparam$,
	while ours begin at a macroscopic size,
	so it is not surprising that they require a much smaller $\sigma$
	for their attachment measure to detect the tip of that small initial slit.
\end{remark}

\begin{remark}
	Convergence of $K_n \to K$ in the Carath\'{e}odory topology
	is equivalent to convergence of the corresponding maps
	$f_{K_n} \to f_K$ on compact subsets of $\Delta$
	\cite[Chapter 3.6]{lawler-conformal-book}.
\end{remark}

\begin{remark}
	\label{remark:alpha}
	In \cite[Corollary 10]{stv-ale} it is shown that if $\alpha > 0$,
	the ALE	started from the trivial initial cluster $\overline{\D}$
	still converges to the union of $\overline{\D}$ with a single slit.
	We conjecture that a similar modification of our Theorem~\ref{thm:ale-lpm}
	should also be true:
\end{remark}

\begin{conjecture}
	\label{thm:alpha-conjecture}
	Let $T$, $K_0$, $\eta$ and $\sigma$ all be as in Theorem~\ref{thm:ale-lpm}
	and suppose $\alpha \geq 0$.
	Run the ALE$(\alpha, \eta, \sigma)$ process until the capacity reaches $T$,
	then the ALE cluster converges
	as $\cparam \to 0$
	to the Laplacian path model with parameter $\alpha + \eta$.
\end{conjecture}

We expect that proving Conjecture~\ref{thm:alpha-conjecture} would
be similar to the $\alpha = 0$ case.
The computations in Section~\ref{sec:removing-regularisation}
would become messier and
the martingale arguments in Section~\ref{sec:Gronwall}
become a bit more difficult but fundamentally the same arguments should apply.

To show the convergence of the ALE model to the LPM,
we introduce two intermediate models which we will call the
auxiliary ALE, and the multinomial model.

The auxiliary ALE attaches particles exactly at the tips of its $k$ slits,
essentially replacing the ALE attachment measure,
which is supported on $\T$ at each step,
with a finite-dimensional process supported on $k$ atoms like the LPM.

We will then introduce the multinomial model
to replace the
regularised derivative in the ALE with the second derivative \emph{on} the boundary.
This is quite simple.
The most difficult and involved step is showing that
the multinomial model,
which is a discrete-time random growth model,
converges to the continuous-time deterministic LPM.
We essentially prove a strong law of large numbers
at each tip:
if we have attached a large number of slits by time $t$
and the probability of attaching at slit $j$
is $p$, then this is similar to a solution to Loewner's equation at time $t$
with a driving function that always gives weight $p$ to the tip of slit $j$.
It is complicated by the fact that the probabilities and weights
both change over time, as do the preimages of the tips.
Analysing the fluctuation of the weights at each tip
from the Laplacian path model's \eqref{lpm-weights}
is the most technical part of Section~\ref{sec:Gronwall},
involving controlling strong feedback in the dynamics.
Another important technical detail is that Loewner's equation is singular at the tips.

\section{Analytic tools}
\label{analytic-tools}

\begin{remark}
	Frequently throughout the paper,
	we will use a constant $A$
	which changes from line to line,
	having the properties that $0 < A < \infty$
	and that $A$ depends only on the parameters which we are holding fixed,
	such as $T$, $K_0$ and $\eta$ in Theorem~\ref{thm:ale-lpm}.
	
	We will occasionally note explicitly on which parameters $A$ depends,
	but it is always independent of the parameter $\cparam$
	which we allow to tend to $0$,
	and of the number of particles.
\end{remark}

The choice of a slit as our particle shape
is very convenient because the map and its derivative have explicit expressions:

\begin{proposition}
	\label{thm:f-expressions}
	Let $\cparam > 0$ and $f = f^{\cparam} : \Delta \to \Delta \setminus (1,1+d]$ as above.
	Let $\beta = \beta_\cparam$
	be the unique angle in $(0,\pi)$ such that $f(e^{i\beta}) = 1$.
	Then $\beta \sim 2\cparam^{1/2}$ for small $\cparam$,
	and for $w \in \Delta$,
	\begin{align}
		\label{f-explicit}
		f(z) = \frac{e^{\cparam}}{2w}(w+1) \left(w+1 + \sqrt{w^2 + 2(1-2e^{-\cparam})w + 1}\right) - 1
	\end{align}
	and
	\begin{align}
		\label{f-derivative}
		f'(w) = \frac{f(w)}{w} \frac{w-1}{(w-e^{i\beta})^{1/2}(w-e^{-i\beta})^{1/2}}.
	\end{align}
\end{proposition}
\begin{proof}
	See \cite[Equation (8) and Lemma 4]{stv-ale}.
\end{proof}

We can find useful bounds on $|f'(w)|$ using \eqref{f-derivative}:
\begin{corollary}
	\label{thm:f-prime-estimate} 
	There are universal constants $A_1, A_2 > 0$ such that
	for all $\cparam < 1$,
	for $w \in \Delta$,
	if $|w - e^{ i\beta}| \leq \frac{3}{4}\beta$,
	then
	\begin{align*}
		A_1 \frac{\beta^{1/2}}{|w - e^{i\beta}|^{1/2}}
		\leq
		|f'(w)|
		\leq
		A_2 \frac{\beta^{1/2}}{|w - e^{i\beta}|^{1/2}},
	\end{align*}
	and similarly if
	$|w - e^{-i\beta}| \leq \frac{3}{4}\beta$.
	
	Moreover, there is a third constant $A_3$
	such that if
	$\min \{ | w - e^{i\beta} |, | w - e^{-i\beta} | \}
	> \frac{3}{4}\beta$,
	then
	\[
	|f'(w)| \leq A_3.
	\]
\end{corollary}

Near $e^{\pm i \beta}$, $f$ has a non-linear effect on distances,
quantified in \cite[Lemma 12]{small-eta}:
\begin{lemma}
	\label{thm:distance-estimate}
	For $w \in \Delta$,
	for all $\cparam < 1$,
	if $| w - e^{i\beta} | \leq \beta / 2$,
	then
	\begin{align*}
		\begin{aligned}
			|f(w) - 1|
			=\ 
			&2(e^{\cparam} - 1)^{1/4}
			|w - e^{i\beta}|^{1/2}\\
			&\times
			\left(
			1 +
			O\left[
			\frac{|w - e^{i\beta}|}{\cparam^{1/2}}
			\vee
			\cparam^{1/4} |w - e^{i\beta}|^{1/2}
			\right]
			\right).
		\end{aligned}
	\end{align*}
\end{lemma}

One useful tool for analysing Loewner's equation is the \emph{backward equation}.
See \cite[Chapter 4]{lawler-conformal-book}\footnote{The backward equation appears in \cite{lawler-conformal-book} in the proof of Theorem 4.13.}.
\begin{definition}
	Fix $T > 0$. Let $\xi : [0,T] \to \R$ be a \cadlag function,
	and $(\varphi_t)_{t \in [0,T]}$
	the solution to Loewner's equation with driving function $\xi$.
	Then the \emph{backward equation}
	is the system of ordinary differential equations
	\begin{align}
		\label{eq:backwards-equation}
		u_0(z) = z, \quad \frac{\pd}{\pd t} u_t(z) = u_t(z) \frac{u_t(z) + e^{i\xi_{T-t}}}{u_t(z) - e^{i\xi_{T-t}}},\quad z \in \Delta,
	\end{align}
	for $t \in [0,T]$.
	Then $u_T = \varphi_T$,
	but it is not usually true that $u_t = \varphi_t$
	for $t < T$.
\end{definition}

\begin{remark}
	It is often easier to analyse $u_t$
	than $\varphi_t$
	since the former is governed by an ordinary rather than partial differential equation.
\end{remark}

\begin{remark}
	If the driving function $\xi$ corresponds to an aggregation process,
	so $\xi_t = \theta_{\rdown{t/\cparam}+1}$,
	then solving Loewner's (forward) equation gives
	$\varphi_{N\cparam} = f_1 \circ \dots \circ f_N$.
	Understanding $\varphi_{N\cparam}(z)$, $\varphi_{N\cparam}'(z)$, etc.\
	requires analysing $f_N(z)$,
	$f_{N-1}(f_N(z))$, etc.\
	For example, $\varphi_{N\cparam}'(z) = \prod_{n=1}^{N} f_n'((f_{n+1}\circ\dots\circ f_N)(z))$,
	but the location of $(f_{n+1}\circ\dots\circ f_N)(z)$
	is difficult to understand using the forward equation.
	However, if we solve the backward equation
	\eqref{eq:backwards-equation},
	we have $u_{n\cparam} = f_{N-n+1} \circ \dots \circ f_N$
	for all $n \leq N$,
	so we could write $\varphi_{N\cparam}'(z) = \prod_{n=1}^{N} f_n'(u_{(N-n)\cparam}(z))$.
\end{remark}

When keeping track of points on the circle,
it is also useful to have estimates on \emph{angular distortion}
away from the tip and bases of the slit:
\begin{lemma}
	\label{thm:angular-distortion}
	For any positive constant $L > 0$,
	there is a constant $A_{L} > 0$ such that
	if $e^{i\alpha}$ and $e^{i\alpha'}$ are each
	at a distance at least
	$L$ from 1
	then
	\begin{equation}
		\label{far-from-tip-derivative}
		|f'(e^{i\alpha})| \leq e^{A_L \cparam}
	\end{equation}
	and
	\begin{equation}
		\label{linear-angle-changes}
		e^{-A_L \cparam} | \alpha - \alpha' |
		\leq
		|
		\arg f(e^{i\alpha})
		-
		\arg f(e^{i\alpha'})
		|
		\leq
		e^{A_L \cparam} | \alpha - \alpha' |
	\end{equation}
	for all sufficiently small $\cparam$.
\end{lemma}

\begin{proof}
	The estimate \eqref{far-from-tip-derivative}
	follows by expanding \eqref{f-derivative}.
	When $|\theta| > \beta$,
	\cite[equation (37)]{stv-ale} says that
	\[
	1 + \cos(\mathrm{arg} f(e^{i\theta})) = (1 + \cos\theta)e^\cparam
	\]
	and this gives us \eqref{linear-angle-changes}.
\end{proof}

In \cite{stv-ale}, the results of which we extend here,
the authors developed a useful set of estimates on the solutions $u^0$ and $u^1$
to \eqref{eq:backwards-equation}
with respective driving functions $\xi^0$ and $\xi^1$.

\begin{lemma}[Lemma 11 of \cite{stv-ale}]
	\label{ucomp} 
	Suppose $z_0 \in {\Delta}$, $T>0$ and $\xi^0:(0,T] \to \R$ are given and let
	\[
	\Lambda_t = \int_0^t \frac{2 |u_s^0(z_0)|^2 \d s}{|(u_s^0)'(z_0)||u_s^0(z_0)e^{-i \xi^0_{T-s}}-1|^2}.
	\]
	There exists some absolute constant $A$ such that, for all $|z| > 1$  satisfying 
	\begin{align}
		\label{zassump}
		|z-z_0| &\leq A^{-1} \inf_{0 \leq t \leq T} \left ( \frac{|u_t^0(z_0)e^{-i \xi^0_{T-t}} -1|}{|(u_t^0)'(z_0)|} \wedge \left ( \int_0^t \frac{|(u_s^0)'(z_0)| }{ |u_s^0(z_0)e^{-i \xi^0_{T-s}} -1|^3} \d s \right )^{-1} \right ), 
	\end{align}
	we have, for all $0 \leq t \leq T$,
	\begin{align}
		\label{loewner-taylor}
		\left | \log \frac{u^0_t(z) - u^0_t(z_0)}{(z-z_0)(u^0_t)'(z_0)}\right | \leq A |z-z_0| \int_0^t \frac{|(u^0_s)'(z_0)|\d s}{|u_s^0(z_0)e^{-i \xi^0_{T-s}} -1|^3}  
	\end{align}
	(where we interpret the left hand side as being equal to $0$ if $z=z_0$) and
	\begin{align}
		\label{logubound}
		\left | \log \frac{(u^0_t)'(z)}{(u^0_t)'(z_0)} \right | \leq A |z-z_0| \int_0^t \frac{|(u^0_s)'(z_0)|\d s}{|u_s^0(z_0)e^{-i \xi^0_{T-s}} -1|^3} . 
	\end{align}
	Furthermore, $A$ can be chosen so that if, in addition, $\xi^1:(0,T] \to \R$ satisfies
	\begin{equation}
		\label{xiassump}
		\|\xi^1-\xi^0\|_T \leq \frac{1}{A} \inf_{0 \leq t \leq T} \left ( \frac{|u_t^0(z_0)e^{-i \xi^0_{T-t}} -1|}{|(u_t^0)'(z_0)|\Lambda_t+|u^0_t(z_0)|} \wedge \left ( \int_0^t \frac{\Lambda_s|(u_s^0)'(z_0)| + |u^0_s(z_0)|}{ |u_s^0(z_0)e^{-i \xi^0_{T-s}} -1|^3} \d s \right )^{-1} \right ),
	\end{equation}
	then, for all $0 \leq t \leq T$,
	\begin{equation}
		\left | u^1_t(z)-u^0_t(z) \right | \leq A|(u^0_t)'(z_0)| \|\xi^1 - \xi^0\|_T \Lambda_t  
		\label{uLambdabound}
	\end{equation}
	and
	\begin{align}
		\label{derivatives-bound}
		\left | \log \frac{(u^1_t)'(z)}{(u^0_t)'(z)} \right | \leq A \| \xi^1-\xi^0\|_T \int_0^t \frac{\Lambda_s|(u_s^0)'(z_0)| + |u^0_s(z_0)|}{ |u_s^0(z_0)e^{-i \xi^0_{T-s}} -1|^3} \d s. 
	\end{align}
\end{lemma}

\section{Removing the regularisation}
\label{sec:removing-regularisation}

In this section we reduce the ALE,
which has an attachment density supported on $\T$,
to the ``multinomial model,''
which, like the LPM, grows only at the tips,
and the probability of attaching to a tip
at time $t$ is determined by the second derivative
of the map corresponding to the cluster.

The proofs in this section mostly involve
careful calculations
to estimate $|\Phi_n'|$ around relevant points.

\subsection{ALE to auxiliary model}
\label{sec:ale-to-aux}

First we define the auxiliary model,
which is a version of the ALE in which growth occurs exactly at the tips.
This first step is a type of dimension reduction,
in which we restrict possible attachment points from all of $\T$
to only $k$ points.

\begin{definition}
	\label{def:aux-model}
	Let $K_0$ be of the same form as above, and $\Phi_0^{*} = \Phi_0^{\mathrm{LPM}}$.
	We will choose an attachment point $\theta$ as we do for the ALE,
	but before attaching a particle at $\theta$,
	we rotate the entire cluster so that $\theta$ is the preimage
	of the tip of a slit.
	We define the map $\Phi_n^*$ of the auxiliary model inductively:
	At step $n$ we still have a configuration of $k$ curves
	(which are not necessarily line segments).
	Let the preimage under $\Phi_n^*$ of the tip of the $j$th curve
	be $\exp(i\preimageaux^j_n)$.
	Choose $\theta_{n+1}^*$ according to the conditional density
	\begin{align*}
		h(\theta \,|\, \theta_1, \dots, \theta_n) 
		=
		\frac{1}{Z_n^*} | (\Phi_t^*)'(e^{\sigma + i\theta}) |^{-\eta},
	\end{align*}
	and if $j_n = \mathop{\mathrm{argmin}}_j 
	|e^{i\theta_{n+1}^*} - e^{i\preimageaux^j_n}|$
	(which almost surely uniquely exists),
	let $\hat\theta_{n+1}^* = \preimageaux^{j_n}_n$,
	and $\delta_{n+1} = \theta_{n+1}^* - \hat\theta_{n+1}^*$.
	
	Then if $R_{\theta}(z) = e^{i\theta}z$,
	set
	\begin{align*}
		\Phi_{n+1}^* = R_{\delta_{n+1}} \circ \Phi_n^* \circ R_{-\delta_{n+1}} \circ f_{\theta_{n+1}^*}.
	\end{align*}
\end{definition}

\begin{remark}
	This procedure of choosing the $(n+1)$th angle
	and attaching a slit to a rotated version of the $n$th cluster
	is illustrated in \autoref{fig:auxiliary-diagram}.
\end{remark}

\begin{figure}[h!]
	\centering
	\begin{tikzpicture}
		\draw node (empty-circle) at (0,0) {\fbox{\includegraphics[width=0.3\textwidth,trim=227 162 195 147, clip]{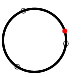}}};
		\draw node at ($(empty-circle)+(1.7,0.5)$) {$\theta_{6}^*$};
		\draw node at ($(empty-circle)+(1.8,-0.3)$) {$\bar \phi_5^1$};
		\draw node at ($(empty-circle)+(-1.3,1.8)$) {$\bar \phi_5^2$};
		\draw node at ($(empty-circle)+(-1.7,-1.8)$) {$\bar \phi_5^3$};
		\draw node (Kn) at (7.5,0) {\fbox{\includegraphics[width=0.3\textwidth,trim=210 132 174 115, clip]{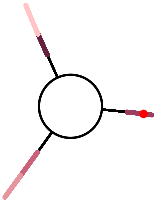}}};
		\draw node (new-slit) at (0,-5.5) {\fbox{\includegraphics[width=0.3\textwidth,trim=227 162 187 147, clip]{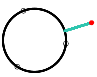}}};
		\draw node at ($(new-slit)+(1.2,-0.3)$) {$\bar \phi_5^1$};
		\draw node at ($(new-slit)+(-1.5,1.5)$) {$\bar \phi_5^2$};
		\draw node at ($(new-slit)+(-1.8,-1.5)$) {$\bar \phi_5^3$};
		\draw node (Knplusone) at (7.5,-5.5) {\fbox{\includegraphics[width=0.3\textwidth,trim=175 125 182 120, clip]{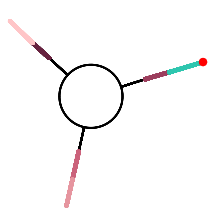}}};
		\draw node (rotated-slit) at (0,-11) {\fbox{\includegraphics[width=0.3\textwidth,trim=227 162 187 147, clip]{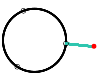}}};
		\draw node at ($(rotated-slit)+(1.1,0.0)$) {$\bar \phi_5^1$};
		\draw node at ($(rotated-slit)+(-1.5,1.5)$) {$\bar \phi_5^2$};
		\draw node at ($(rotated-slit)+(-1.8,-1.5)$) {$\bar \phi_5^3$};
		\draw node (rotated-Knplusone) at (7.5,-11) {\fbox{\includegraphics[width=0.3\textwidth,trim=164 132 198 115, clip]{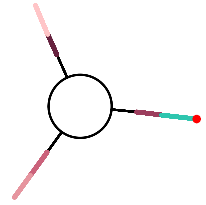}}};
		\draw[->] (empty-circle) -- node[anchor=south] {$\Phi_5^*$} (Kn);
		\draw[->] (empty-circle) -- node[anchor=west] {$f_{\theta_{6}^*}$} (new-slit);
		\draw[->] (empty-circle) -- node[above=0.1cm] {$\Phi_{6}^*$} (Knplusone);
		\draw[->] (new-slit) -- node[anchor=west] {$R_{-\delta_{6}}$} (rotated-slit);
		\draw[->] (rotated-slit) -- node[anchor=south] {$\Phi_5^*$} (rotated-Knplusone);
		\draw[->] (rotated-Knplusone) -- node[anchor=west] {$R_{\delta_{6}}$} (Knplusone);
		\draw[->] (new-slit) -- node[anchor=south] {$R_{\delta_{6}} \circ \Phi_5^* \circ R_{-\delta_{6}}$} (Knplusone);
	\end{tikzpicture}
	\caption{%
		\label{fig:auxiliary-diagram}One step
		in the construction of the auxiliary model.
		The large dot at one tip represents the image of
		$\exp(i \theta_{6}^*)$ under each map.
		Comparing $K_5$ and $K_6$,
		we can see that both have three curves
		but the attachment points have been rotated
		from one cluster to the next.
	}
\end{figure}

\begin{remark}
	Since $f_\theta \circ R_{-\delta} = R_{-\delta} \circ f_{\theta+\delta}$,
	we can write $\Phiaux_n$ in two ways:
	\begin{align}
		\nonumber
		\Phiaux_n &= R_{\delta_1 + \dots + \delta_n} \circ \Phi_0 \circ R_{-\delta_1} \circ f_{\theta_1^*} \circ R_{-\delta_2} \circ \dots \circ R_{-\delta_n} \circ f_{\theta_n^*}\\
		\label{eq:aux-rotation}
		&=
		R_{\delta_1 + \dots + \delta_n} \circ \Phi_0 \circ R_{-(\delta_1 + \dots + \delta_n)} \circ
		f_{\theta_1^* + \delta_2 + \delta_3 + \dots + \delta_n}
		\circ f_{\theta_2^* + \delta_3 + \dots + \delta_n}
		\circ \dots \circ
		f_{\theta_n^*},
	\end{align}
	and this latter expression is also equal to
	\begin{align*}
		R_{\delta_1 + \dots + \delta_n} \circ \Phi_0 \circ R_{-(\delta_1 + \dots + \delta_n)}
		\circ
		f_{\hat \theta_1^* + \delta_1 + \dots + \delta_n}
		\circ
		f_{\hat \theta_2^* + \delta_2 + \dots + \delta_n}
		\circ
		f_{\hat \theta_n^* + \delta_n}.
	\end{align*}
\end{remark}

\begin{definition}
	\label{def:coupling-failure}
	We denote the preimages of the $j$th tip by
	$e^{i\phi^j_n}$ in the ALE model,
	and by $e^{i\bar \phi^j_n}$ in the
	auxiliary model.
	Define
	\begin{align*}
		j_n   &:= \mathop{\mathrm{argmin}}_j | e^{i\theta_n} - e^{i\phi^j_n} |,\\
		j_n^* &:= \mathop{\mathrm{argmin}}_j | e^{i\theta_n^*} - e^{i\bar\phi^j_n} |.
	\end{align*}
	If we construct a version of each of the ALE and auxiliary models
	on a common probability space,
	then we can define the stopping time
	\begin{align*}
		\couplingfailure := \min\{ n \geq 1 : j_n \not= j_n^* \}.
	\end{align*}
\end{definition}

To show that the ALE and auxiliary models are close,
we will couple the models so that at each step
the probability $j_n = j_n^*$ is close to 1,
and show that with high probability the ALE lands very near a tip.

First we will examine the regions near the tips in each model,
and find that the probability of landing near each tip is similar.
Then we will prove in each case that the probability of not landing
near any tip is $o(1)$,
generalising the main result of \cite{stv-ale},
with a similar argument.

\begin{definition}
	\label{def:tau-D}
	The coupling relies on both models attaching near tips,
	so we define another stopping time which triggers
	when a particle in either model is too far from its nearest tip.
	Let
	\begin{align*}
		\tau_D = \inf\{ n \geq 1 : |\theta_n^* - \hat\theta_n^*| > D \text{ or } \min_{j} |\theta_n - \phi^j_n| > D \},
	\end{align*}
	where $D = \sqrt{\sigma}$.
\end{definition}

We will establish a bound on $\P[ \tau_D < n ]$
which does not depend on the exact nature of the coupling,
and then use this bound to show that a coupling exists
under which $\P[ \couplingfailure \leq \rdown{T/\cparam} ]$
is small.

\begin{remark}
	The ALE and auxiliary models will have a common weak limit if
	with high probability $\tau_D \wedge \couplingfailure > T/\cparam$.
	This event means that every particle in the ALE model
	has been attached within $D$ of a ``main'' tip,
	and every particle in the auxiliary model has chosen the same slit
	to attach to as the ALE model.
	An example of the coupling of these two models
	is shown in \autoref{fig:ale-aux-coupling}.
\end{remark}

\begin{theorem}
	\label{thm:ale-aux-global}
	There exists a constant $A = A(K_0, k, \eta, T)$
	such that
	the ALE and auxiliary models can be constructed on a common probability
	space and on this space
	\begin{align*}
		\P(\tau_D \wedge \couplingfailure \leq \rdown{T/\cparam}) \leq A \cparam^{1/2}
	\end{align*}
	provided $\sigma < \cparam^{\gamma}$
	for $\gamma = \frac{2(\eta+2)}{\eta-1} \vee \frac{5\eta + 10}{2\eta} \vee 8$.
\end{theorem}

\begin{figure}
	\centering
	\includegraphics[width=0.49\linewidth,trim = 140 50 120 80, clip]{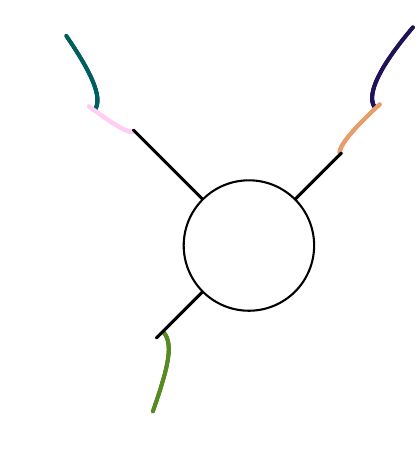}
	\includegraphics[width=0.49\linewidth,trim = 140 50 120 80, clip]{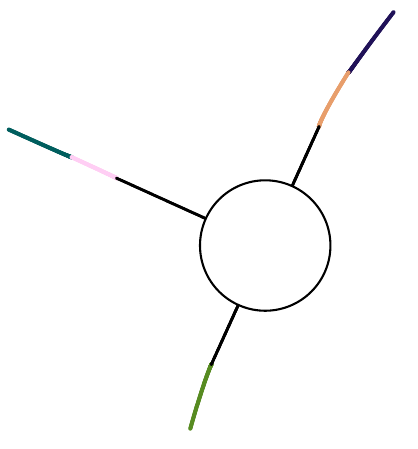}
	\caption{\label{fig:ale-aux-coupling}
		Coupled ALE (left) and auxiliary (right) models
		at a time $n < \tau_D \wedge \tau_{\mathrm{coupling}}$
		started from a three-arm configuration.
		Note the slight rotation of three initial arms
		in the auxiliary model,
		and that the curves are repelled from each other.
	}
\end{figure}

\begin{remark}
	The three terms determining $\gamma$ come from three separate requirements
	in the proof of this section's result:
	we require $\sigma < \cparam^8$
	so that the derivatives of each model
	look similar near the ``main'' tips,
	as in Proposition~\ref{thm:coupling-induction}.
	We require that $\sigma < \cparam^{\frac{5\eta+10}{2\eta}}$
	so that ``old'' particles do not contribute in the ALE model,
	and we require that
	$\sigma < \cparam^{\frac{2(\eta+2)}{\eta-1}}$
	so that the measures are concentrated very tightly around the tip
	of each particle,
	with each particle attached within distance $D$
	of a main tip.
	
	All three terms and their maximum are plotted in \autoref{fig:gamma}.
	
	In fact,
	the $\frac{5\eta + 10}{2\eta}$ term is redundant
	since it is always less than 8,
	but we retain it because the assumption
	$\gamma \geq \frac{5\eta + 10}{2\eta}$
	makes some of the arguments clearer,
	as it corresponds to one of the exponents
	in Proposition~\ref{thm:ale-to-aux-inductive}.
\end{remark}

\begin{remark}
	It would be possible to use a slightly smaller $\gamma$ if we
	replaced the $A\cparam^{1/2}$ with a weaker $o(1)$ upper bound
	in Theorem~\ref{thm:ale-aux-global}
	(which is all we require to prove Theorem~\ref{thm:ale-lpm})
	but we have left it in its current form to avoid complicating
	the proofs further.
\end{remark}

\begin{figure}
	\centering
	\includegraphics[width=\linewidth, trim = 10 10 10 10, clip]{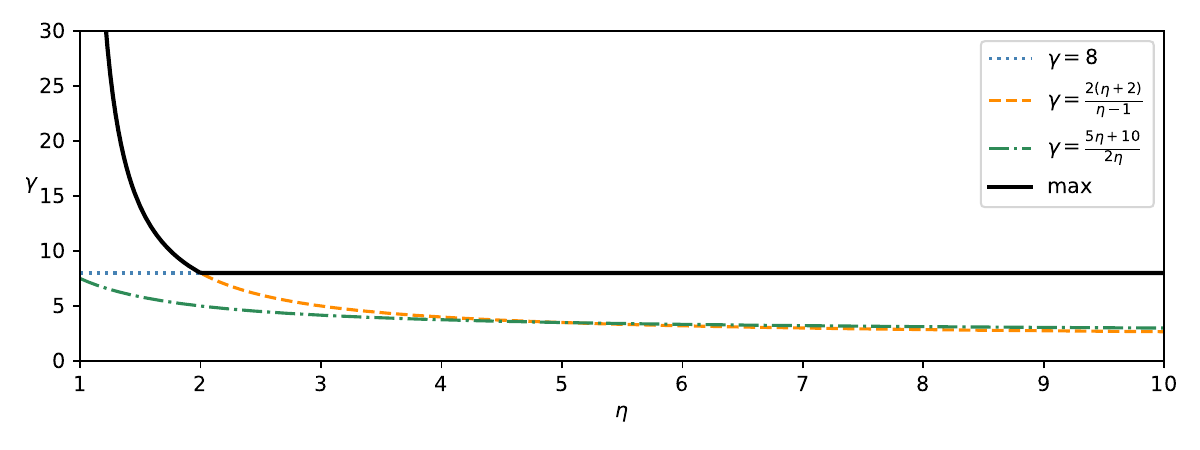}
	\caption{\label{fig:gamma}A plot of the exponent $\gamma$ used to control
		the regularisation in Theorem~\ref{thm:ale-aux-global}
		as a function of the parameter $\eta$.
		Each of the three terms determining $\gamma$ is plotted separately,
		with their maximum shown as a solid black curve.
	}
\end{figure}

The proof of Theorem~\ref{thm:ale-aux-global} will follow
from a proposition controlling $\tau_D$ and another
controlling $\tau_D \wedge \couplingfailure$,
both of which we state now.

\begin{proposition}
	\label{thm:ale-to-aux-inductive}
	For any $n < N$,
	\begin{align*}
		\P[ n+1 = \tau_D \,|\, n < \tau_D ]
		\leq
		A \cparam^{-\eta} \sigma^{\frac{\eta-1}{2}}
		+
		A \cparam^{-2\eta}\sigma^{\eta-1}
		+
		A \cparam^{-\frac{5\eta}{4} - \frac{1}{2}} \sigma^{\eta/2}.
	\end{align*}
	If $\sigma < \cparam^{\gamma}$
	where $\gamma
	=
	\frac{2(\eta+2)}{\eta-1} \vee \frac{5\eta + 10}{2\eta} \vee 8$,
	then
	this implies 
	$\P[ n+1 = \tau_D \,|\, n < \tau_D ] \leq A \cparam^2$
	and so $\P[ \tau_D \leq \rdown{T/\cparam} ] \leq AT\cparam$.
\end{proposition}

\begin{proposition}
	\label{thm:coupling-induction}
	Assume that $\sigma < \cparam^8$.
	We can construct the coupling of the ALE and auxiliary models
	in such a way that for all $n \leq T/\cparam$,
	on the event $\{n < \tau_D \wedge \couplingfailure\}$,
	the conditional probability that $j_{n+1} \not= \bar j_{n+1}$
	or $n+1 = \tau_D$ is almost surely bounded by
	\begin{align*}
		ATk\cparam^{3/2} + \P\left[ \delta_{n+1} > D \,|\, n < \tau_D \right] + \P\left[ \min_j |\theta_{n+1} - \phi^j_n | > D \,\bigg|\, n < \tau_D \right].
	\end{align*}
\end{proposition}

\begin{proof}[Proof of Theorem~\ref{thm:ale-aux-global}]
	Note
	\[
	\P(\tau_D \wedge \couplingfailure \leq \rdown{T/\cparam})
	\leq
	\P(\tau_D \leq \rdown{T/\cparam})
	+ \P(\couplingfailure \leq \rdown{T/\cparam} \,|\, \tau_D > \rdown{T/\cparam}).
	\]
	By Proposition~\ref{thm:ale-to-aux-inductive},
	the first term in the sum above is less than $A \cparam$
	under any coupling.
	Then using Proposition~\ref{thm:coupling-induction}
	we can construct a coupling such that
	for any $n \leq \rdown{T/\cparam}$,
	\begin{align*}
		\P( \couplingfailure = n+1 \,|\, n < \couplingfailure \text{ and } \tau_D > \rdown{T/\cparam} )
		&\leq
		A\cparam^{3/2} + 2 \P( n+1 = \tau_D \,|\, n < \tau_D ),
	\end{align*}
	and the second term is bounded by $A\cparam^2$.
	Therefore, summing over $n < \rdown{T/\cparam}$,
	\[
	\P( \couplingfailure \leq \rdown{T/\cparam} \,|\, \tau_D > \rdown{T/\cparam} )
	\leq A\cparam^{1/2}
	\]
	as required.
\end{proof}

Now we have to prove Propositions~\ref{thm:ale-to-aux-inductive}
and \ref{thm:coupling-induction}.
We begin with Proposition~\ref{thm:coupling-induction}
as its proof is somewhat simpler.

\begin{proof}[Proof of Proposition~\ref{thm:coupling-induction}]
	The result follows simply if we can prove the claim:
	for $n \leq T/\cparam$, on the event $\{n < \tau_D \wedge \couplingfailure\}$,
	\begin{align*}
		\sup_{|\theta| < \cparam^2}
		\left| \log\frac{
			(\Phi_n^{\mathrm{ALE}})'(e^{\sigma + i\theta}e^{i\phi^j_n})
		}{
			(\Phi_n^*)'(e^{\sigma + i\theta}e^{i\bar\phi^j_n})
		}		
		\right|
		\leq
		AT\cparam^{3/2}
	\end{align*}
	almost surely.
	
	Denote the preimages under $\Phiale_n$ of the $k$ tips 
	in the ALE model by $\preimageale^1_n, \dots, \preimageale^k_n$.
	For each $j \in \{1,2,\dots,k\}$ denote the subsequence of
	$(\thetaale_\ell)_{\ell = 1}^n$
	consisting of the times at which a particle is attached to
	the $j$th slit by
	\[
	\thetaale_{n_j(1)}, \thetaale_{n_j(2)}, \dots, \thetaale_{n_j(N_j)}.
	\]
	Near $\preimageale_n^j$ we can decompose $\Phiale_n$ as
	\begin{equation}
		\label{ale-decomposition}
		\Phiale_n
		=
		\Phiale_0
		\circ
		\Psi_{0}^j
		\circ
		\fale{n_j(1)} \circ \Psi^j_1
		\circ
		\fale{n_j(2)} \circ \Psi^j_2
		\circ
		\dots
		\circ
		\fale{n_j(N_j)} \circ \Psi^j_{N_j},
	\end{equation}
	where for $0 < \ell < N_j$,
	$\Psi^j_\ell
	= \fale{n_j(\ell)+1} \circ \fale{n_j(\ell)+2} \circ \dots \circ \fale{n_j(\ell+1)-1}
	= \left(\Phiale_{n_j(\ell)}\right)^{-1} \circ \Phiale_{n_j(\ell+1)-1}$,
	the map which attaches every particle landing at slits \emph{other than}
	the $j$th between the $\ell$th and $(\ell+1)$th
	time a particle lands on the $j$th slit.
	The two other maps are
	$\Psi^j_{0} = \left(\Phiale_{0}\right)^{-1} \circ \Phiale_{n_j(1)-1}$
	and
	$\Psi^j_{N_j} = \left(\Phiale_{n_j(N_j)}\right)^{-1} \circ \Phiale_{n}$,
	which correspond to the particles before the first attachment
	and after the last attachment to the $j$th slit, respectively.
	Note that any of these $\Psi^j_\ell$ maps may be the identity map on $\Delta$ if $n_j(\ell+1) = n_j(\ell) + 1$.
	The important common feature of all these $\Psi^j_\ell$ maps
	is that their derivatives have no poles or zeroes near the $j$th slit.
	
	Since the event we have conditioned on implies $n < \couplingfailure$,
	the subsequence of $(\thetaaux_\ell)_{\ell = 1}^{n}$
	consisting of times when a particle was attached to the $j$th slit is
	\[
	\thetaaux_{n_j(1)}, \thetaaux_{n_j(2)}, \dots, \thetaaux_{n_j(N_j)}
	\]
	with the same $n_j$ as above.
	There is a similar decomposition to \eqref{ale-decomposition},
	complicated only slightly by the cluster rotation.
	Let $\Delta_{l}^{n} := \delta_l+\delta_{k+1}+\dots+\delta_n$.
	Recall that $\thetaaux_l = \hat\theta^*_l + \delta_l$. Then
	\begin{align*}
		\Phiaux_n
		&=
		R_{\Delta_1^n}
		\circ
		\Phiaux_0
		\circ
		R_{-\Delta_{1}^{n}}\circ\\
		&\phantom{==}\circ
		\Psiaux^j_0
		\circ
		\faux{\hat\theta^*_{n_j(1)} + \Delta_{n_j(1)}^{n}} \circ \Psiaux^j_1
		\circ \dots \circ
		\faux{\hat\theta^*_{n_j(N_j)} + \Delta_{n_j(N_j)}^{n}} \circ \Psiaux^j_{N_j},
	\end{align*}
	where $\Psiaux^j_\ell
	=
	\faux{\hat\theta^*_{n_j(\ell)+1}+\Delta_{n_j(\ell)+1}^n}
	\circ \dots \circ
	\faux{\hat\theta^*_{n_j(\ell+1)-1}+\Delta_{n_j(\ell+1)-1}^{n}}$
	for $0 < \ell < N_j$ and the two end-cases
	are also defined similarly to $\Psi^j_0$ and $\Psi^j_{N_j}$.\\
	
	Now fix $\theta \in (-\cparam^2,\cparam^2)$
	and we will compare
	the densities
	$|(\Phiale_n)'(e^{\sigma + i \theta}e^{i\preimageale^j_n})|^{-\eta}$
	and
	$|(\Phiaux_n)'(e^{\sigma + i \theta}e^{i\preimageaux^j_n})|^{-\eta}$
	when $N_j > 0$.
	We omit the simpler case $N_j = 0$.
	
	First we will bound
	\[
	\left|
	\log \frac{
		(\Psi^j_{N_j})'(e^{\sigma + i \theta}e^{i\preimageale^j_n})
	}{
		(\Psiaux^j_{N_j})'(e^{\sigma + i \theta}e^{i\preimageaux^j_n})
	}
	\right|.
	\]
	The two maps $\Psi^j_{N_j}$ and $\Psiaux^j_{N_j}$
	are solutions to Loewner's equation with driving functions
	\[
	\xi_t = \thetaale_{\rdown{t/\cparam}+n_j(N_j)+1},
	\quad
	\overline{\xi}_t = \hat\theta^*_{\rdown{t/\cparam}+n_j(N_j)+1}
	+ \Delta_{\rdown{t/\cparam}+n_j(N_j)+1}^n
	\]
	respectively, for $t \in [0, (n-n_j(N_j))\cparam)$.
	Remark 3.7 of \cite{selander-thesis} notes that the points
	$e^{i\bar\phi^j_n}$ are repelled from each other on the circle.
	We therefore know that there is an $L > 0$
	depending only on the initial conditions and $T$
	such that 
	$e^{\sigma+i\theta}e^{i\preimageale^j_n}$
	and $e^{\sigma + i \theta}e^{i\preimageaux^j_n}$ keep at least a distance
	$L$ away from $\xi_t$ and $\overline{\xi}_t$ respectively.
	We will apply Lemma~\ref{ucomp},
	so note $\Lambda_t \asymp t$,
	condition \eqref{xiassump} is met for sufficiently small $\cparam$,
	and the bound on the right of \eqref{zassump} is $\asymp 1$.
	
	We know
	$e^{i \preimageale^j_n}
	=
	\left( \Psi^j_{N_j} \right)^{-1}(
	\thetaale_{n_j(N_j)}
	)$
	and
	$e^{i \preimageaux^j_n}
	=
	\left( \Psiaux^j_{N_j} \right)^{-1}(
	\hat\theta^*_{n_j(N_j)} + \Delta_{n_j(N_j)}^n
	)$,
	and so applying Lemma \ref{ucomp} gives us, since $(n-n_j(N_j))\cparam \leq T$,
	\begin{align*}
		\left|
		\log
		\frac{
			(\Psi^j_{N_j})'(e^{\sigma + i\theta} e^{i\preimageale_n^j})
		}{
			(\Psiaux^j_{N_j})'(e^{\sigma + i\theta} e^{i\preimageaux_n^j})
		}
		\right|
		&\leq
		A_L T \left(
		\left| \preimageale_n^j - \preimageaux_n^j \right|
		+
		|\delta_1| + \dots + |\delta_{n-1}|
		\right),
	\end{align*}
	and a further calculation using explicit estimates for $f^{-1}$
	similar to \eqref{linear-angle-changes}
	tells us
	\[
	|\preimageale_n^j - \preimageaux_n^j| \leq A_L (|\delta_1| + \dots + |\delta_n|).
	\]
	Hence, with a constant $A$ depending on $A_L$ and $T$,
	\begin{align}
		\label{non-singular-ratio}
		\left|
		\log
		\frac{
			(\Psi^j_{N_j})'(e^{\sigma + i\theta} e^{i\preimageale_n^j})
		}{
			(\Psiaux^j_{N_j})'(e^{\sigma + i\theta} e^{i\preimageaux_n^j})
		}
		\right|
		&\leq
		A \left(
		|\delta_1| + \dots + |\delta_{n}|
		\right).
	\end{align}
	The next two terms in the expansions of each derivative
	are the two largest,
	\begin{align}
		\label{biggest-terms}
		(\fale{n_j(N_j)})' \left( \Psi^j_{N_j}(e^{\sigma + i\theta} e^{i\preimageale_n^j})\right)
		\quad\text{and}\quad
		(\faux{\hat\theta^*_{n_j(N_j)} + \Delta_{n_j(N_j)}^{n}})' \left( \Psiaux^j_{N_j}(e^{\sigma + i\theta} e^{i\preimageaux_n^j})\right).
	\end{align}
	It is clear from \eqref{f-derivative}
	that, to first order, the size of $|f'(z)|$ depends on $|z-1|$.
	This means that to find the size of the two derivatives above,
	we are most interested in
	\[
	\left| \Psi^j_{N_j}(e^{\sigma + i\theta} e^{i\preimageale_n^j}) - e^{i\thetaale_{n_j(N_j)}}\right|
	\quad\text{and}\quad
	\left|
	\Psiaux^j_{N_j}(e^{\sigma + i\theta} e^{i\preimageaux_n^j})
	-
	e^{i(\hat\theta^*_{n_j(N_j)} + \Delta_{n_j(N_j)}^n) }
	\right|.
	\]
	If we apply \eqref{loewner-taylor}, the first of these is
	\begin{align}
		\nonumber
		\left|
		\Psi^j_{N_j}(e^{\sigma + i\theta} e^{i\preimageale_n^j}) - e^{i\thetaale_{n_j(N_j)}}
		\right|
		&=
		\left|
		\Psi^j_{N_j}(e^{\sigma + i\theta} e^{i\preimageale_n^j})
		-
		\Psi^j_{N_j}(e^{i\preimageale_n^j})
		\right|\\
		\nonumber
		&=
		\left|
		(e^{\sigma+i\theta}-1)e^{i\preimageale_n^j} (\Psi^j_{N_j})'(e^{i\preimageale_n^j})
		\right|
		\left[
		1 + O\left( |e^{\sigma+i\theta}-1| \right)
		\right]\\
		&=
		\label{distancefromtip-ale}
		|e^{\sigma + i\theta} - 1| \left|(\Psi^j_{N_j})'(e^{i\preimageale_n^j})\right|
		\left[
		1 + O\left( |e^{\sigma+i\theta}-1| \right)
		\right]
	\end{align}
	and similarly the second term is
	\begin{align}
		\label{distancefromtip-aux}
		|e^{\sigma + i\theta} - 1|
		\left|(\Psiaux^j_{N_j})'(e^{i\preimageaux_n^j})\right|
		\left[
		1 + O\left( |e^{\sigma+i\theta}-1| \right)
		\right].
	\end{align}
	Our previous bounds on $\left| \log \frac{
		(\Psi^j_{N_j})'(e^{\sigma + i\theta} e^{i\preimageale_n^j})
	}{
		(\Psiaux^j_{N_j})'(e^{\sigma + i\theta} e^{i\preimageaux_n^j})
	}\right|$
	apply equally to $\left| \log \frac{
		(\Psi^j_{N_j})'(e^{i\preimageale_n^j})
	}{
		(\Psiaux^j_{N_j})'(e^{i\preimageaux_n^j})
	}\right|$,
	so \eqref{f-derivative} allows us to directly compare
	the derivatives in \eqref{biggest-terms}.
	
	Firstly, a calculation using the explicit form of $f(z)$
	in \eqref{f-explicit}
	shows that if $|z-1| < \cparam^{1/2}$ then
	$f(z) = f(1) + O(|z-1|) = 1 + d + O(|z-1|)$.
	The next estimate $z = 1 + O(|z-1|)$ is obvious.
	Putting this estimate into the denominator
	of the second fraction in \eqref{f-derivative},
	we get $|z - e^{\pm i \beta}|^{1/2}
	= |1 - e^{i\beta}|^{1/2} (1 + O\left(\frac{|z-1|}{\cparam^{1/2}}\right) )$.
	Thus for $z$ close to 1 the behaviour of $|f'(z)|$ mainly depends on $|z-1|$,
	and we can write
	\begin{align*}
		|f'(z)| = \frac{1+d(\cparam)}{|1-e^{i\beta}|}\, |z-1|
		\left(
		1 + O(\cparam^{-1/2} |z-1|)
		\right).
	\end{align*}
	
	Therefore
	\begin{align}
		\nonumber
		\left|
		\log
		\frac{
			(\fale{n_j(N_j)})'\left(
			\Psi^j_{N_j}(e^{\sigma + i\theta} e^{i\preimageale_n^j})
			\right)
		}{
			(\faux{\hat\theta^*_{n_j(N_j)} + \Delta_{n_j(N_j)}^{n}})'\left(
			\Psiaux^j_{N_j}(e^{\sigma + i\theta} e^{i\preimageaux_n^j})
			\right)
		}
		\right|
		&\leq
		\left|
		\log
		\frac{
			\left|(\Psi^j_{N_j})'(e^{i\preimageale_n^j})\right|
		}{
			\left|(\Psiaux^j_{N_j})'(e^{i\preimageaux_n^j})\right|
		}
		\right|
		+
		A \frac{|e^{\sigma + i\theta} - 1|}{\cparam^{1/2}} \\
		\label{tip-comparison}
		&\leq
		A \left(
		|\delta_1| + \dots + |\delta_{n}|
		+
		\frac{|e^{\sigma + i\theta} - 1|}{\cparam^{1/2}}
		\right).
	\end{align}
	
	Now as $|\fale{n_j(N_j)}(\Psi_{N_j}(e^{\sigma + i\theta}e^{i\preimageale_{n}^{j}}))|
	\geq 1 + \cparam^{1/2}$
	and
	$|\faux{n_j(N_j)}(\Psiaux_{N_j}(e^{\sigma + i\theta}e^{i\preimageaux_{n}^{j}}))|
	\geq 1 + \cparam^{1/2}$,
	we can apply Lemma \ref{ucomp}
	to compare the two remaining derivatives
	\[
	\left|
	\left(
	\Phiale_0 \circ	\Psi_{0}^j \circ \fale{n_j(1)} \circ \dots \circ \Psi_{N_j-1}^j
	\right)'\left(
	\fale{n_j(N_j)}(\Psi_{N_j}(e^{\sigma + i\theta}e^{i\preimageale_{n}^{j}}))
	\right)
	\right|
	\]
	and
	\begin{align*}
		\Bigg|
		\left(
		R_{\Delta_1^n}
		\circ
		\Phiaux_0
		\circ
		R_{-\Delta_1^n}
		\circ
		\Psiaux^j_0
		\circ
		\faux{\hat\theta^*_{n_j(1)} + \Delta_{n_j(1)}^{n}}
		\circ \dots \circ
		\Psiaux^j_{N_j-1}
		\right)'
		\Big(\faux{n_j(N_j)}(\Psiaux_{N_j}(e^{\sigma + i\theta}e^{i\preimageaux_{n}^{j}}))
		\Big)
		\Bigg|.
	\end{align*}
	
	The two maps
	above whose derivatives we consider 
	can be generated by the backward equation using driving functions
	which we will call $\xi^0$ (generating
	$\Phiale_0 \circ \dots \circ \Psi_{N_j-1}^j$)
	and $\xi^1$ (generating 
	$R_{\Delta_1^n} \circ \dots \circ \Psiaux^j_{N_j-1}$).
	We can now use Lemma~\ref{ucomp} to compare the above derivatives.
	We have the bound
	$\| \xi^0 - \xi^1 \|_{\infty} \leq |\delta_1| + \dots + |\delta_n|$.
	Taking $z_0
	= \fale{n_j(N_j)}(\Psi_{N_j}(e^{\sigma + i\theta}e^{i\preimageale_{n}^{j}}))$,
	we have $|z_0| - 1 \geq \cparam^{1/2}$,
	so using standard estimates for conformal maps
	\cite[Chapter 4.4, Corollary 6]{duren-univalent} we have
	$\frac{\cparam^{1/2}}{A} \leq |(u_t^0)'(z_0)| \leq \frac{A}{\cparam^{1/2}}$
	for all $0 < t < T$,
	where the constant $A$ depends only on $T$ and $K_0$.
	This gives us, modifying the constant $A = A(T, K_0)$ appropriately
	(as we will throughout),
	$\frac{\cparam^{1/2}}{A} t \leq |\Lambda_t| \leq \frac{A}{\cparam^{3/2}}t$.
	
	The right-hand side of \eqref{zassump} is bounded below by
	\begin{align*}
		\inf_{0 < t \leq T}
		\left(
		\frac{\cparam^{1/2}}{A \cparam^{-1/2}}
		\wedge
		\left(
		t \frac{A\cparam^{-1/2}}{\cparam^{3/2}}
		\right)^{-1}
		\right)
		\geq
		A^{-1} \cparam^{2},
	\end{align*}
	so the condition \eqref{zassump} is satisfied by
	$z = \faux{\hat\theta^*_{n_j(N_j)} + \Delta_{n_j(N_j)}^n}(
	\Psiaux^j_{N_j}(e^{\sigma + i\theta} e^{i\preimageaux_n^j})
	)$.
	
	The resulting bounds on
	$\left|
	\log \frac{(u_t^0)'(z)}{(u_t^0)'(z_0)}
	\right|$
	and
	$\left|
	\log \frac{(u_t^1)'(z)}{(u_t^0)'(z)}
	\right|$
	are then, by Lemma~\ref{ucomp},
	\[
	\left|
	\log
	\frac{(u_t^0)'(z)}
	{(u_t^0)'(z_0)}
	\right|
	\leq
	A |z-z_0|
	t \cparam^{-2}
	\]
	and
	\[
	\left|
	\log
	\frac{(u_t^1)'(z)}
	{(u_t^0)'(z)}
	\right|
	\leq
	A
	(|\delta_1| + \dots + |\delta_n|)
	\cparam^{-7/2}.
	\]
	
	Expanding \eqref{f-explicit} for $w$ close to $1$,
	we can estimate
	\begin{align}
		\label{near-tip-distortion}
		f(w) = 1 + d(\cparam) + O\left( \frac{|w - 1|^2}{\cparam^{1/2}} + |w-1| \right),
	\end{align}
	so using \eqref{distancefromtip-ale} and \eqref{distancefromtip-aux},
	\begin{align*}
		|z - z_0|
		&\leq
		(1+d(\cparam))
		\left|
		\thetaale_{n_j(N_j)} - (\hat\theta^*_{n_j(N_j)} + \Delta_{n_j(N_j)}^n)
		\right|\\
		&\phantom{\leq}
		+
		A \frac{|e^{\sigma + i\theta} - 1|^2}{\cparam^{1/2}}
		\left(
		\left|(\Psi^j_{N_j})'(e^{i\preimageale_n^j})\right|
		+
		\left|(\Psiaux^j_{N_j})'(e^{i\preimageaux_n^j})\right|
		\right)\\
		&\leq
		A( |\delta_1| + \dots + |\delta_n|
		+ \cparam^{-1/2} |e^{\sigma + i\theta} - 1|^2
		).
	\end{align*}
	
	Hence
	\begin{align*}
		\left|
		\log\frac{
			\left(
			R_{\Delta_1^n}
			\circ
			\Phiaux_0
			\circ
			R_{-\Delta_1^n}
			\circ
			\dots \circ
			\Psiaux^j_{N_j-1}
			\right)'\left(
			\faux{n_j(N_j)}(\Psiaux_{N_j}(e^{\sigma + i\theta}e^{i\preimageaux_{n}^{j}}))
			\right)
		}{
			\left(
			\Phiale_0 \circ	\Psi_{0}^j \circ \fale{n_j(1)} \circ \dots \circ \Psi_{N_j-1}^j
			\right)'\left(
			\fale{n_j(N_j)}(\Psi_{N_j}(e^{\sigma + i\theta}e^{i\preimageale_{n}^{j}}))
			\right)
		}
		\right|
	\end{align*}
	is
	\begin{align}
		\nonumber
		\left|
		\log\frac{
			(u_T^1)'(z  )
		}{
			(u_T^0)'(z_0)
		}
		\right|
		&\leq
		\left|
		\log\frac{
			(u_T^1)'(z)
		}{
			(u_T^0)'(z)
		}
		\right|
		+		
		\left|
		\log\frac{
			(u_T^0)'(z  )
		}{
			(u_T^0)'(z_0)
		}
		\right|\\
		\label{underneath-tip-comparison}
		&\leq
		A
		\left(
		(|\delta_1| + \dots + |\delta_n|)\cparam^{-7/2}
		+
		|e^{\sigma + i\theta} - 1|^2 \cparam^{-5/2}
		\right).
	\end{align}
	
	Putting together \eqref{non-singular-ratio},
	\eqref{tip-comparison} and \eqref{underneath-tip-comparison}
	we have
	\begin{align*}
		\left|
		\log\frac{
			(\Phiale_n)'(e^{\sigma + i\theta}e^{i\preimageale_{n}^{j}})
		}{
			(\Phiaux_n)'(e^{\sigma + i\theta}e^{i\preimageaux_{n}^{j}})
		}
		\right|
		&\leq
		A
		\left(
		\frac{|\delta_1| + \dots + |\delta_n|}{\cparam^{7/2}}
		+
		\frac{|e^{\sigma + i\theta} - 1|}{\cparam^{1/2}}
		+
		\frac{|e^{\sigma + i\theta} - 1|^2}{\cparam^{5/2}}
		\right)
	\end{align*}
	and we can check that this is bounded by $AT \cparam^{3/2}$,
	so our claim and the result follow.
\end{proof}

The following three lemmas will be used to prove
Proposition~\ref{thm:ale-to-aux-inductive}.

First we show that
the probability of attaching more than distance $D$
from \emph{any} tip is $o(\cparam)$ in each model.
We will demonstrate it only for the ALE,
because identical proofs work for the auxiliary model,
setting $D = 0$ instead of $D = \sqrt{\sigma}$.

\begin{lemma}
	\label{partition-function-lower-bound}
	There exists a constant $A$ depending only on $\eta$, $T$ and the initial conditions
	such that
	for any $1 \leq n < \rdown{T/\cparam}$,
	on the event $\{ n < \tau_D \}$
	we have
	\[
	Z_n \geq \cparam^{\eta} \sigma^{-(\eta - 1)}
	\]
	almost surely.
\end{lemma}

The above bound is not sharp:
the $\cparam^{\eta}$ can be removed,
but the proof is more complicated,
and the bound in the lemma is all that we require.

\begin{proof}
	For $|\theta| < \cparam$
	and a fixed $j$,
	let $z = e^{\sigma + i (\preimageale_n^j + \theta)}$,
	and using \eqref{ale-decomposition} we can write
	\[
	|(\Phiale_n)'(z)|
	=
	\left|(\Psiale_{N_j}^j)'(z)\right|
	\left|(\fale{n_j(N_j)})'(\Psiale_{N_j}^j(z))\right|
	\left|
	(\Phiale_0 \circ \dots \circ \Psiale_{N_j - 1}^j)'(
	\fale{n_j(N_j)}(\Psiale_{N_j}^j(z))
	)
	\right|.
	\]
	Applying Lemma \ref{ucomp} to the first term in this decomposition
	with $z_0 = e^{i\preimageale_n^j}$,
	because $n \leq T/\cparam$
	we have a constant lower bound $L > 0$
	on the distance between $u_t(z_0)$
	and the driving function.
	Therefore we have from \eqref{logubound} that
	\begin{align}
		\label{push-derivative}
		|(\Psiale_{N_j}^j)'(z)|
		=
		|(\Psiale_{N_j}^j)'(e^{i\preimageale_n^j})|
		(1 + O( \cparam )).
	\end{align}
	From \eqref{f-derivative} and \eqref{distancefromtip-ale} 
	we obtain
	\begin{align}
		\nonumber
		|(\fale{n_j(N_j)})'(\Psiale_{N_j}^j(z))|
		&=
		\frac{1 + d(\cparam)}{|1 - e^{i\beta_\cparam}|}(1 + O(\cparam^{1/2}))
		\left|
		\Psiale_{N_j}^j(z) - e^{i\thetaale_{n_j(N_j)}}
		\right|\\
		\label{tip-derivative}
		&=
		\frac{1 + d(\cparam)}{|1 - e^{i\beta_\cparam}|}
		| z - e^{i\preimageale_n^j} |
		\left|
		(\Psiale_{N_j}^j)'(e^{i\preimageale_n^j})
		\right|
		(1 + O(\cparam^{1/2})).
	\end{align}
	Analysis of Loewner's reverse equation
	when $z$ is far from the driving measure shows that
	there is a constant $A > 0$ with
	\[
	A^{-1} \leq \left| (\Psiale_{N_j}^j)'(e^{i\preimageale_n^j}) \right| \leq A,
	\]
	(see \cite[equation (26)]{stv-ale})
	and using \eqref{distancefromtip-ale} and \eqref{near-tip-distortion}
	we know
	\[
	|\fale{n_j(N_j)}(\Psiale_{N_j}^j(z))| - 1 \geq \cparam^{1/2},
	\]
	and so by
	standard conformal map estimates
	(see \cite[Chapter 4.4]{duren-univalent}),
	\begin{align}
		\label{off-circle-derivative}
		\frac{\cparam^{1/2}}{A}
		\leq
		\left|
		(\Phiale_0 \circ \dots \circ \Psiale_{N_j - 1}^j)'(
		\fale{n_j(N_j)}(\Psiale_{N_j}^j(z))
		)
		\right|
		\leq
		A \cparam^{-1/2}.
	\end{align}
	Then combining \eqref{push-derivative}, \eqref{tip-derivative}
	and \eqref{off-circle-derivative},
	we have
	\begin{align*}
		|(\Phiale_n)'(e^{\sigma + i(\preimageale_n^j + \theta)})|^{-\eta}
		&\geq
		A \cparam^{\eta} | e^{\sigma + i\theta} - 1 |^{-\eta}.
	\end{align*}
	Thus,
	using the bound
	$|e^{\sigma + i\theta}-1|^2 = (e^{\sigma}-1)^2 + 2e^{\sigma}(1 - \cos\theta)
	\leq A (\sigma^2 + \theta^2)$ for $|\theta| < \cparam$ and
	substituting $x = \theta/\sigma$, we get
	\begin{align*}
		Z_n
		&\geq
		A \cparam^{\eta}
		\int_{-\cparam}^{\cparam} \frac{\d \theta}{|e^{\sigma + i\theta} - 1|^{\eta}}\\
		&\geq
		A \cparam^{\eta}
		\sigma^{-(\eta - 1)}
		\int_{-\cparam/\sigma}^{\cparam/\sigma}
		\frac{\d x}{(1 + x^2)^{\eta/2}},
	\end{align*}
	and since $\sigma < \cparam$,
	the integral term is bounded below by a constant.
\end{proof}

\begin{lemma}
	\label{no-contribution-from-circle}
	Let $\zeta \in \Delta$.
	If $\cparam$ is sufficiently small,
	then $\log |f(\zeta)| < \cparam^{1/2}$
	implies that $|f'(\zeta)| > 1$.
\end{lemma}

\begin{proof}
	Let $e^{r + i\theta} = f(\zeta)$.
	We can use Proposition~\ref{thm:f-expressions} to obtain
	\begin{align*}
		(f^{-1})'(w) = \frac{f^{-1}(w)}{w} \frac{w-1}{\sqrt{(w+1)^2 - 4e^{\cparam}w}}
	\end{align*}
	for $w \in \Delta \setminus (1, 1+d(\cparam)]$.
	
	Then
	we know that $|f^{-1}(w)| < |w|$ for any $w$,
	and if we write $w = e^{r + i\theta}$,
	then we can calculate that
	$$
	\left| \frac{w-1}{(w+1)^2 - 4e^c w} \right| < 1
	\iff
	\cos \theta < \frac{e^\cparam}{\cosh r}.
	$$
	Hence $|(f^{-1})'(w)| < 1$,
	and so $|f'(\zeta)| > 1$.
	
	Since $\cosh^{-1}(e^{\cparam}) \sim (2\cparam)^{1/2}$
	for small $\cparam$,
	if $\cparam$ is sufficiently small and $r < \cparam^{1/2}$,
	then the condition $\cos\theta < \frac{e^\cparam}{\cosh r}$ is satisfied
	for any $\theta$.
\end{proof}


\begin{lemma}
	\label{on-circle-deriv-bound}
	For any sufficiently large constant $A_T$,
	there exists a constant $B > 0$
	depending on $T$ and $K_0$
	such that
	on the event $\{n < \couplingfailure\}$,
	i.e.\ when the ALE and auxiliary models choose the same particle
	at each of the first $n$ steps,
	the following is almost surely true:
	for any $\theta \in \T$,
	if $k \leq n$ satisfies
	\begin{align*}
		|
		(\fale{l+1} \circ \fale{l+2} \circ \dots \circ \fale{n})(e^{i\theta})
		-
		e^{i\thetaale_{l}}
		|
		\geq
		|
		e^{i (\beta_\cparam + \frac{\cparam^{1/2}}{A_T})} - 1
		|
		\text{ for all } k \leq l \leq n,
	\end{align*}
	and
	$|e^{i\theta} - e^{i\thetaale_n}|
	\geq
	|e^{i (\beta_\cparam + \frac{\cparam^{1/2}}{A_T})} - 1|
	$
	then
	$
	\left|
	(\fale{k} \circ \dots \circ \fale{n})'(e^{i\theta})
	\right|
	\leq
	B \cparam^{-1/2}.
	$
\end{lemma}

\begin{proof}
	On the event $\{n < \couplingfailure\}$,
	each angle $\thetaale_{l}$ is within $D$
	of the $j_l$th tip, $\phi^{j_l}_{l-1}$.
	Define the coupled angle sequence $(\bar\theta_l)_{l\leq n}$
	and corresponding particle maps $(\bar f_l)_{l\leq n}$
	inductively by $\bar\theta_1 = \phi^{j_1}_0$,
	and $\bar\theta_l = \bar\phi^{j_l}_{l-1}$ for $l > 1$,
	where $\bar\phi^{j_l}_{l-1}$ is the preimage of the $j_l$th tip
	under $\bar\Phi_{l-1} = \bar f_1 \circ\dots\circ \bar f_{l-1}$.
	The difference between the two angle sequences is then
	almost surely
	$\sup_{l \leq n} |\thetaale_l - \bar\theta_l|
	< A \cparam^{-1}D$,
	for a constant $A > 0$.
	Then using the bound on $\left|\log \frac{(u^1_t)'(z)}{(u^0_t)'(z)} \right|$
	from Lemma~\ref{ucomp},
	we can establish the bound we want on the ALE process
	by establishing it using the angle sequence 
	$(\bar\theta_l)_{l \leq n}$,
	since $D = \sqrt{\sigma}$ is very small.
	We can also think of this as first considering $D = 0$
	and then perturbing the result.\\
	
	Since we have fixed a $T < \infty$,
	there is a positive constant $L$ such that
	$L < \min_{i \not= j} |\bar\preimageale^i_l - \bar\preimageale^j_l|$
	for all $1 \leq l \leq n$.
	Let $A_{L}$ be the corresponding constant in
	the bounds \eqref{far-from-tip-derivative}
	and \eqref{linear-angle-changes}
	from Lemma~\ref{thm:angular-distortion}.
	For any given $\theta \in \T$ there can be at most one $j$
	such that
	there exists an $l \leq n$
	with $|\bar f_l(e^{i\theta}) - \bar\preimageale^j_{l-1}|
	\leq L e^{- A_{L} T} =: m$.
	If there is no such $j$,
	or the corresponding $l$
	is less than $k$,
	then by \eqref{far-from-tip-derivative},
	$
	\left|
	(\bar f_{k} \circ \dots \circ \bar f_{n})'(e^{i\theta})
	\right|
	\leq
	e^{A_m T}.
	$\\
	
	Suppose that such a $j$ does exist.
	We will split the angle sequence
	$(\bar \theta_l)_{k \leq l \leq n}$
	into times when particles are attached to slit $j$
	and times when particles are attached elsewhere.
	
	Set $n_0 = \max \{ k \leq l \leq n : \bar\theta_l = \bar\preimageale^j_l \}$,
	then for $i \geq 0$ set
	$n_i' = \max \{ k \leq l < n_i : \bar\theta_l \not= \bar\preimageale^j_l \}$
	and $n_{i+1} = \max \{ k \leq l < n_i' : \bar\theta_l = \bar\preimageale^j_l \}$
	until one of the sets is empty,
	and call the last well defined value $n_{p+1}$
	(if $n_{p}'$ is the final well-defined value of the above maxima,
	then set $n_{p+1} = n_{p}'$).
	Define the capacities $t_0 = (n - n_0)\cparam$,
	and for $i \geq 0$,
	$t_i' = (n_i - n_i')\cparam$
	and 
	$t_{i+1} = (n_i' - n_{i+1})\cparam$.
	The map $\bar f_{k} \circ \dots \circ \bar f_{n}$
	is then generated by using the backwards equation
	\eqref{eq:backwards-equation}
	with a driving function which first
	is at a distance at least $L$ from $\bar\preimageale^j_n$
	for a period of duration $t_0 \geq 0$,
	then takes the constant value $\preimageale^j_{n_0}$
	for a period of duration $t_0' \geq \cparam$,
	then is distance at least $L$ from tip $j$ again
	for duration $t_1 \geq \cparam$, and so on,
	terminating after a total time $(n-k+1)\cparam \leq T$.
	
	We can therefore decompose $\bar f_{k} \circ \dots \circ \bar f_{n}$
	as
	\begin{align*}
		\Psi_{p+1}
		\circ
		f_{t_p', \bar\preimageale^j_{n_p'}}
		\circ
		\Psi_p
		\circ 
		\dots
		\circ
		f_{t_0', \bar\preimageale^j_{n_0'}}
		\circ
		\Psi_0
	\end{align*}
	where the $\Psi$ maps are generated by driving functions
	bounded away from the $j$th slit.
	
	Using \eqref{far-from-tip-derivative},
	we have
	$
	|\Psi_i'(
	f_{t_{i-1}',\bar\preimageale^j_{n_{i-1}'}} \circ \dots \circ \Psi_0(e^{i\theta})
	)|
	\leq
	e^{A_{L} t_i}
	$
	for each $i$,
	and so the total contribution to 
	$
	\left|
	(\bar f_{k} \circ \dots \circ \bar f_{n})'(e^{i\theta})
	\right|
	$
	by $\Psi'$ terms is bounded by a constant $e^{A_{L}T}$.
	
	The other terms can give a larger contribution.
	Let $\delta_i$ be the distance in $\R/2\pi \Z$ between
	$\arg [(\Psi_i \circ \dots \circ \Psi_0)(e^{i\theta})]$
	and $\bar\preimageale^j_{n_i}$,
	and similarly let $\delta_i'$ be
	the distance between
	$\arg [(f_{t_{i}',\bar\preimageale^j_{n_{i}'}} \circ \Psi_i \circ \dots \circ \Psi_0)(e^{i\theta})]$
	and $\bar\preimageale^j_{n_i}$.
	Then using \eqref{f-derivative}
	we can compute
	\begin{align*}
		\left|f_{t_{i}',\bar\preimageale^j_{n_{i}'}}'(
		\Psi_i \circ \dots \circ \Psi_0)(e^{i\theta})
		)\right|
		=
		e^{t_i' / 2}
		\frac{
			|e^{i\delta_i} - 1|
		}{
			|e^{i\delta_i'} - 1|
		}.
	\end{align*}
	Using \eqref{linear-angle-changes},
	we can calculate $\frac{\delta_{i+1}}{\delta_i'} \leq e^{A_{L}t_i}$,
	and so the total contribution from all of the
	$f_{t_{i}', \bar\preimageale^j_i}'$ terms is,
	possibly increasing the constant $A_{L}$,
	\begin{align*}
		e^{\frac{1}{2}\sum_{i=0}^{p} t_i'}
		\frac{
			\left| e^{i\delta_0} - 1 \right|
		}{
			\left| e^{i\delta_p'} - 1 \right|
		}
		\prod_{i=0}^{p-1}
		\frac{|e^{i\delta_{i+1}} - 1|}{|e^{i\delta_i'} - 1|}
		&\leq
		e^{(\frac{1}{2} + A_{L}) T}
		\frac{2}{\cparam^{1/2}/A_T}\\
		&=
		B \cparam^{-1/2}
	\end{align*}
	as required.	
\end{proof}

\begin{proof}[Proof of Proposition~\ref{thm:ale-to-aux-inductive}]
	First note that
	on the event $\{n < \tau_D\}$, we have
	$$\P[ n+1 < \tau_D \,|\, (\theta_1, \dots, \theta_n) ] \leq \left(1 - \sum_j \int_{\preimageale^j_n-D}^{\preimageale^j_n+D} h_n(\theta) \,\d \theta\right) + 
	\left(1 - \sum_j \int_{\preimageaux^j_n-D}^{\preimageaux^j_n+D} h_n^*(\theta)\,\d \theta\right)$$
	almost surely.
	We will hence bound $1 - \sum_j \int_{-D}^{D} h_n( \preimageale^j_n +\theta)\,\d \theta$,
	and a similar bound will apply to the auxiliary term.
	Let $z = e^{\sigma + i\theta}$.
	Note that
	\begin{align*}
		1 - \sum_j \int_{-D}^{D} h_n(\preimageale^j_n + \theta) \d \theta
		&=
		\sum_j
		\int_{D < |\theta - \preimageale^j_n| < \cparam^2} h_n(\theta)\d \theta
		+
		\int_{\{ |\theta - \preimageale^j_n| \geq \cparam^2\, \forall j \}}
		h_n(\theta)\d \theta.
	\end{align*}
	If $D < |\theta - \preimageale_n^j| < \cparam^2$
	for some $j$,
	then we can use \eqref{push-derivative}, \eqref{tip-derivative}
	and \eqref{off-circle-derivative}
	to establish the almost-sure upper bound
	\begin{align}
		\label{near-tip-upper-bound}
		|(\Phiale_n)'(e^{\sigma + i\theta})|^{-\eta} \leq A |e^{\sigma + i\theta} - e^{i\preimageale_n^j}|^{-\eta}
		\leq A D^{-\eta}.
	\end{align}
	Then using \eqref{near-tip-upper-bound} and Lemma~\ref{partition-function-lower-bound},
	almost surely
	\begin{align*}
		\int_{\preimageale_n^j + D}^{\preimageale_n^j + \cparam^2} h_n(\theta) \d\theta
		&\leq
		A
		\cparam^{-\eta}
		D^{-(\eta - 1)}
		\sigma^{\eta-1}.
	\end{align*}
	
	Next we suppose $\theta$ is further than $\cparam^2$ from any tip $\preimageale_n^j$.
	
	
	
	We will classify the point $z = e^{\sigma + i\theta}$ based on its projection $\hat{z} = e^{i\theta}$
	and follow them both through their backwards evolutions,
	so let $z_k = (f_k \circ f_{k+1} \circ \dots \circ f_n)(z)$,
	and also let
	$\hat z_k = (f_k \circ f_{k+1} \circ \dots \circ f_n)(\hat z)$.
	Let $L$ be a lower bound on
	$\min_{l \leq n}\min_{i \not= j}|\preimageale_l^i - \preimageale_l^j|$,
	let $A_{L/2}$ be the constant in Lemma~\ref{thm:angular-distortion} corresponding to points further than $L/2$
	from 1,
	and let $A_T = e^{A_{L/2}T}$.
	Observe that there can be at most two values of $k$ such that
	both $\hat z_k \in \T$ and
	$
	|\hat z_k - e^{i\thetaale_{k-1}}|
	<
	|e^{i(\beta_\cparam + \frac{\cparam^{1/2}}{A_T})} - 1|,
	$
	because after the second condition is met once,
	applying $\rdown{T/\cparam}$ maps cannot move $\hat z_k$
	further than $\beta$ away from a tip,
	and it will be taken off the circle the next time
	a slit is attached near to it.
	This gives us three cases
	depending on the number of $k$s.
	
	\textbf{Case I:} First suppose there is no such $k$.
	Then by Lemma~\ref{on-circle-deriv-bound},
	$|(\fale{1} \circ \dots \circ \fale{n})'(\hat z)| \leq B\cparam^{-1/2}$
	almost surely,
	so by \eqref{loewner-taylor} from Lemma~\ref{ucomp},
	\[
	|(\fale{1} \circ \dots \circ \fale{n})(z)
	- (\fale{1} \circ \dots \circ \fale{n})(\hat z)|
	\leq B \cparam^{-1/2}\sigma.
	\]
	This gives us
	$|(\fale{1} \circ \dots \circ \fale{n})(z)| - 1 \leq B \cparam^{-1/2}\sigma$,
	and hence by Lemma~\ref{no-contribution-from-circle},
	$
	|(\fale{1} \circ \dots \circ \fale{n})'(z)| > 1.
	$
	Let $j$ be the index of the closest $\preimageale_0^j$
	to $(\fale{1} \circ \dots \circ \fale{n})(\hat z)$.
	We can decompose $\Phi_0$
	as $f_{t_j,\preimageale_0^j} \circ \Psi^j_0$,
	where $\Psi^j_0$ is generated in the reverse equation
	by a driving function
	run for total time $c_0 - t^j$
	which stays at least distance $L/4$
	from $(u_s(\hat z))_{s \leq c_0 - t^j}$
	at all times.
	Then using equation (26) of \cite{stv-ale},
	\[
	e^{-A c_0}
	\leq
	|(\Psi_0^j)'( (\fale{1} \circ \dots \circ \fale{n})(\hat z) )|
	\leq
	e^{A c_0}
	\]
	for an appropriate constant $A = A(K_0)$.
	If none of the $n$ particles have been attached to the initial $j$th slit,
	then the assumption we made on $\theta$ tells us
	$|(\Psi_0^j \circ \fale{1} \circ \dots \circ \fale{n})(\hat z)
	- e^{i\preimageale_0^j}| \geq \cparam / A$.
	If any particles \emph{were} attached to the $j$th slit,
	then we must have
	$|(\Psi_0^j \circ \fale{1} \circ \dots \circ \fale{n})(\hat z) )
	- e^{i\preimageale_0^j}| \geq \cparam^{1/2}/A > \cparam/A$,
	otherwise some $k$ as above would exist.
	In either case we have
	\[
	|(\Psi_0^j \circ \fale{1} \circ \dots \circ \fale{n})(z) )
	- e^{i\preimageale_0^j}|
	\geq
	\cparam/2A,
	\]
	and so 
	\[
	|f_{t_j,\preimageale_0^j}'(
	(\Psi_0^j \circ \fale{1} \circ \dots \circ \fale{n})(z)
	)|
	\geq
	\cparam/18A.
	\]
	Thus, for a modified constant $A$,
	$|\Psiale_0'(z)| \geq \cparam/A$,
	and so $h_n(\theta) = O( \sigma^{\eta - 1}/\cparam^{2\eta} )$.\\
	
	\textbf{Case II:} If there is only one such $k$,
	we must have
	$
	|\hat z_k - e^{i\thetaale_{k-1}}|
	\geq
	\frac{\cparam^{1/2}}{2 A_T^2},
	$
	since there is an $l \geq k$ with $\thetaale_{l}$
	attached at the same slit as $\thetaale_{k-1}$
	and if $\hat z_k$
	were any closer to $e^{i\thetaale_{k-1}}$
	then
	$\hat z_{l+1}$
	would be within $|e^{i(\beta_\cparam + \frac{\cparam^{1/2}}{A_T})} - 1|$
	of $e^{i\thetaale_l}$.
	
	So
	$
	|(\fale{k-1})'( \hat z_k )|
	\geq
	\left|
	f'\left(
	\exp\left(i\frac{\cparam^{1/2}}{2A_T^2}\right)
	\right)
	\right|
	\geq
	A
	$
	for some constant $A > 0$,
	then
	using Lemma~\ref{ucomp}
	and Lemma~\ref{on-circle-deriv-bound} 
	%
	we can derive
	\begin{align}
		\label{no-contribution-proof:away-from-tip-bound}
		|(\fale{k-1})'( z_k )|
		\geq
		A^{-1},
	\end{align}
	since $\sigma \ll \cparam^{3/2}$.
	
	To compute $|(\Phi_0 \circ \fale{1} \circ \dots \circ \fale{k-2})'
	(\hat z_{k-1})|$,
	consider two cases:
	(a)
	$|\hat z_k -	e^{i\thetaale_{k-1}}|
	\leq
	|e^{i(\beta - \frac{\cparam^{1/2}}{A_T})} - 1|$,
	or (b)
	$|\hat z_k -e^{i\thetaale_{k-1}}|
	>
	|e^{i(\beta - \frac{\cparam^{1/2}}{A_T})} - 1|$.
	
	In case (a),
	$| \hat z_{k-1} | - 1
	\geq
	|f(e^{i(\beta - \frac{\cparam^{1/2}}{A_T})})| - 1
	\geq
	\frac{\cparam^{1/2}}{A}$.
	By an identical argument to that which established
	\eqref{off-circle-derivative},
	\[
	|(\Phi_0 \circ \fale{1} \circ \dots \circ \fale{k-2})'(
	\hat z_{k-1}
	)|
	\geq
	\frac{\cparam^{1/2}}{A},
	\]
	and again this easily extends to
	\begin{align}
		\label{no-contribution-proof:off-circle-bound}
		|(\Phi_0 \circ \fale{1} \circ \dots \circ \fale{k-2})'(
		z_{k-1}
		)|
		\geq
		\frac{\cparam^{1/2}}{A},
	\end{align}
	since $\sigma \ll \cparam^{3/2}$.
	Combining
	\eqref{no-contribution-proof:off-circle-bound} with
	\eqref{no-contribution-proof:away-from-tip-bound}
	and Lemma~\ref{no-contribution-from-circle},
	we have
	\begin{align*}
		|(\Phiale_n)'(z)| \geq \frac{\cparam}{A},
	\end{align*}
	and so $h_n(\theta) = O( \sigma^{\eta-1}/\cparam^{2\eta} )$.
	
	In case (b),
	we have
	\[
	\min_{\pm}
	|
	\hat z_k
	-
	e^{i(\thetaale_{k-1} \pm \beta_\cparam)}
	|
	<
	\frac{2 \cparam^{1/2}}{A_T},		
	\]
	and so
	\[
	\min_{\pm}
	|
	z_k
	-
	e^{i(\thetaale_{k-1} \pm \beta_\cparam)}
	|
	<
	\frac{3 \cparam^{1/2}}{A_T}.
	\]
	Without loss of generality 
	the minimum is achieved in both cases
	by $e^{i(\thetaale_{k-1} + \beta_\cparam)}$.
	Let $\delta = |z_k
	-
	e^{i(\thetaale_{k-1} + \beta_\cparam)}|$,
	then by Corollary~\ref{thm:f-prime-estimate},
	\begin{align}
		\label{near-base-contribution}
		| (\fale{k-1})'(z_k) |
		\geq
		A^{-1}
		\frac{
			\cparam^{1/4}
		}{
			\delta^{1/2}
		}
	\end{align}
	and 
	Lemma~\ref{thm:distance-estimate}
	tells us that
	$|
	z_{k-1}
	-
	e^{i\thetaale_{k-1}}
	|
	\asymp
	\cparam^{1/4} \delta^{1/2}	
	$.
	If the attachment point $\thetaale_{k-1}$
	is on the $j$th slit, then it is
	within $D$ of the preimage $\preimageale^j_{k-2}$
	of $\thetaale_{l}$ under $\fale{l+1} \circ \dots \circ \fale{k-2}$,
	where $l < k-1$ was the previous time
	a particle was attached at the $j$th slit
	(without loss of generality such an $l$ exists,
	as we can decompose the initial condition
	to have a slit of capacity $\cparam$
	at the top of a longer slit in position $j$).

	By \eqref{tip-derivative},
	the size of
	$|(\fale{l})'(z_{l+1})|$
	depends on
	$|z_{l+1} - e^{i\thetaale_l}|$,
	which, similarly to \eqref{linear-angle-changes},
	satisfies
	\[
	e^{-AT}
	\leq
	\frac{
		|z_{l+1} - e^{i\thetaale_l}|
	}{
		|
		z_{k-1}
		-
		e^{i\preimageale_{k-2}^j}
		|
	}
	\leq
	e^{AT}.
	\]
	Then since
	\begin{align*}
		|
		z_{k-1}
		-
		e^{i\thetaale_{k-1}}
		|
		&\leq
		|
		z_{k-1}
		-
		e^{i\preimageale_{k-2}^j}
		|
		+
		|
		e^{i\preimageale_{k-2}^j}
		-
		e^{i\thetaale_{k-1}}
		|\\
		&\leq
		|
		z_{k-1}
		-
		e^{i\preimageale_{k-2}^j}
		|
		+
		2D,
	\end{align*}
	we have, for some constant $A$,
	\[
	A^{-1} \cparam^{1/4} \delta^{1/2}
	\leq
	|
	z_{k-1}
	-
	e^{i\preimageale_{k-2}^j}
	|
	+
	2D
	\leq
	e^{AT} |z_{l+1} - e^{i\thetaale_l}| + 2D,
	\]
	and so
	\begin{align*}
		|z_{l+1}-e^{i\thetaale_l}| \geq \sigma \vee \left(\frac{\cparam^{1/4}\delta^{1/2}}{A} - 2D\right).
	\end{align*}
	Then
	\[
	|(\fale{l})'(z_{l+1})|
	\geq
	\frac{1}{A'} \cparam^{-1/2} |z_{l+1} - e^{i\thetaale_l}|,
	\]
	and so combining this with
	Lemma~\ref{no-contribution-from-circle}
	and \eqref{near-base-contribution},
	\begin{align}
		\label{cancellation-near-base}
		|(\fale{l} \circ \dots \circ \fale{n})'(z)|
		&\geq
		\frac{ \sigma \vee ( A^{-1}\cparam^{1/4}\delta^{1/2} - 2D )}{\delta^{1/2}}.
	\end{align}
	Since
	$|z_{k-1} - e^{i\preimageale_{k-2}^j}|
	\geq \sigma$,
	\eqref{cancellation-near-base} is bounded below by
	$
	A^{-1}
	\cparam^{-1/4}
	\frac{\sigma}{9D^2}
	$.
	Then since $| z_l | - 1 \geq \cparam^{1/2}$,
	we have an overall lower bound on $|\Phi_n'(z)|$ of
	\[
	A^{-1}
	\cparam^{1/2}
	\frac{ \sigma \vee ( A^{-1}\cparam^{1/4}\delta^{1/2} - 2D )}{\delta^{1/2}}.
	\]
	Note that $\delta$ is proportional to $|z - (f_n \circ \dots \circ f_{k+1})(e^{i(\thetaale_{k-1}+\beta_\cparam)})|$,
	so we can integrate our bound on $| \Phi_n'(z) |^{-\eta}$
	from $\sigma$ to $\cparam^2$ to get
	$\int_{\{ \text{one $k$ exists} \}} h_n(\theta)\d \theta
	\leq A \sigma^{-1} D^{\eta + 2} \cparam^{-\frac{5\eta}{4} - \frac{1}{2}}$.
	
	\textbf{Case III:}
	If there are two values of $k$, $k_2 < k_1$,
	we can find a bound of the same order
	using the same argument as above, replacing $k$ by $k_1$
	and $l$ by $k_2$.
\end{proof}

\subsection{Auxiliary to multinomial model}

Next we will define the \emph{multinomial model},
in which the probability of attaching a particle at each tip
depends on the second derivative of the relevant map.
This essentially corresponds to taking $\sigma \to 0$
in the auxiliary model.

We explained the significance of the second derivative of the map
for the Laplacian path model
in Section~\ref{sec:lpm-definition},
so the multinomial model can be viewed as a halfway point
between the ALE and LPM.

\begin{definition}
	\label{def:multinomial-model}
	Begin with the same initial condition
	as the other models,
	$\Phimulti_0 = \Phi_0$.
	Let the preimages of the $k$ tips
	under $\Phimulti_n$ be $\preimagemulti_n^j$
	for $j = 1,2,\dots, k$.
	Choose $\theta_{n+1}$
	from $\{ \preimagemulti_n^1, \dots, \preimagemulti_n^k \}$,
	with probabilities
	\[
	\P( \theta_{n+1} = \preimagemulti_n^j )
	=
	\frac{
		| (\Phimulti_n)''(e^{i\preimagemulti_n^j}) |^{-\eta}
	}{
		Z_n
	},
	\]
	where $Z_n = \sum_j | (\Phimulti_n)''(e^{i\preimagemulti_n^j}) |^{-\eta}$.
	Define inductively the maps
	\begin{align*}
		\Phimulti_{n+1} = \Phimulti_n \circ f_{\theta_{n+1}}.
	\end{align*}
\end{definition}

If they choose the same tips,
the auxiliary model (re-rotated to fix the basepoints)
and the multinomial model coincide exactly.
It is therefore fairly simple to prove a coupling exists
as long as the probabilities assigned to the tips are close:

\begin{proposition}
	\label{thm:aux-multi-coupling}
	Let $(\thetaaux_l)_{l \leq n}$
	be the angle sequence for the auxiliary model,
	without the rotation used in Section~\ref{sec:ale-to-aux},
	and let $(\theta_l)_{l \leq n}$
	be the angle sequence for the multinomial model
	with the same initial conditions.
	Define the stopping time $\tau_{\not=}
	= \inf\{ l : \thetaaux_l \not= \theta_l \}$.
	On the event $\{ l < \tau_{\not=} \wedge n \}$,
	the conditional distributions of
	$\thetaaux_{l+1}$ and $\theta_{l+1}$
	are almost surely supported on the same set
	$\{ \preimagemulti_l^1, \dots, \preimagemulti_l^k \}$,
	and
	when $\sigma$ is as in Theorem~\ref{thm:ale-aux-global} and $D = \sqrt{\sigma}$
	we almost surely have
	\begin{align*}
		\max_{1 \leq j \leq k}
		|
		\P( \thetaaux_{l+1} = \preimagemulti_l^j \,|\, (\theta_k)_{k \leq l} )
		-
		\P( \theta_{l+1} = \preimagemulti_l^j \,|\, (\theta_k)_{k \leq l} )
		|
		\leq
		A
		\cparam^{-1}D
	\end{align*}
	for a deterministic constant $A$ depending on $T$ and $\eta$ and $\Phi_0$.
\end{proposition}

\begin{corollary}
	\label{thm:aux-multi}
	We can construct the coupling of $(\thetaaux_l)_{l \leq n}$
	and $(\theta_l)_{l \leq n}$ in such a way that $\P(\tau_{\not=} < \rdown{T/\cparam}) \leq AT \cparam^{-2}D$.
\end{corollary}

The proof of Proposition~\ref{thm:aux-multi-coupling} is relatively simple: 
since $(\Phiaux_n)'(e^{i\preimagemulti_l^j}) = 0$,
the value of $\P( \thetaaux_{l+1} = \preimagemulti_l^j )$
is asymptotically proportional to
$|(\Phiaux_n)''(e^{i\preimagemulti_l^j})|^{-\eta}$,
and hence approximates
$\P( \theta_{l+1} = \preimagemulti_l^j )$.

\begin{proof}[Proof of Proposition~\ref{thm:aux-multi-coupling}]
	We need to show
	$\int_{-D}^{D}
	|(\Phiaux_l)'(e^{\sigma + i(\preimagemulti_l^j+\theta)})|^{-\eta}
	\,\d \theta$
	is proportional to $|(\Phiaux_l)''(e^{i\preimagemulti_l^j})|^{-\eta}$.
	For $|\theta| < D$,	
	let $\gamma$ be the line segment in $\overline{\Delta}$
	from $e^{i\preimagemulti_l^j}$
	to $e^{\sigma + i(\preimagemulti_l^j+\theta)}$.
	Then by the fundamental theorem of calculus,
	\begin{align*}
		(\Phiaux_l)'(e^{\sigma + i(\preimagemulti_l^j+\theta)})
		&=
		(e^{\sigma + i\theta} - 1) e^{i\preimagemulti_l^j}
		(\Phiaux_l)''(e^{i\preimagemulti_l^j})
		+
		\int_\gamma
		(e^{\sigma + i(\preimagemulti_l^j + \theta)} - z)
		(\Phiaux_l)^{(3)}(z)
		\,\d z,
	\end{align*}
	and so
	\begin{align}
		\label{eq:aux-multi-comparison}
		\left| \log \frac{
			|(\Phiaux_l)'(e^{\sigma + i(\preimagemulti_l^j+\theta)})|
		}{
			|e^{\sigma + i\theta} - 1|
			\big|(\Phiaux_l)''(e^{i\preimagemulti_l^j})\big|
		}\right|
		&\leq
		\frac{
			|e^{\sigma + i\theta} - 1|
			\sup_{z \in \gamma} | (\Phiaux_l)^{(3)}(z) |
		}{
			|(\Phiaux_l)''(e^{i\preimagemulti_l^j})|
		}.
	\end{align}
	We can decompose $\Phiaux_l$
	as $\Phiaux_{k-1} \circ f_k \circ \Psi_k$,
	where $1 \leq k \leq l$ is the last time a particle
	was attached to slit $j$
	(if no particle has been attached we can regard the top
	part of the initial slit as a particle of capacity $\cparam$,
	so without loss of generality we can assume at least one particle
	has been added to each slit).
	For $u$ satisfying the backwards equation \eqref{eq:backwards-equation}
	with driving function $\xi$ on $[0,T]$,
	equation (26) of \cite{stv-ale} gives an expression for the spatial derivative
	\[
	u_t'(z)
	=
	\exp\left(
	t - \int_0^t \frac{2 \,\d s}{(u_s(z)e^{-i\xi_{T-s}}-1)^2}
	\right).
	\]
	Using this expression and its higher spatial derivatives,
	we find that
	for a constant $A$ depending only on $\eta$, $T$ and the initial conditions,
	we have bounds
	$A^{-1} \leq |\Psi_k'(z)| \leq A$,
	$|\Psi_k''(z)| \leq A$,
	$|\Psi_k^{(3)}(z)| \leq A$
	for all $z \in \gamma$.
	Then by the chain rule,
	and as $f_k'(\Psi_k(e^{i\preimageale_l^j})) = 0$,
	\begin{align*}
		|(\Phiaux_l)''(e^{i\preimageale_l^j})|
		&=
		\left|
		\Psi_k''(e^{i\preimageale_l^j})
		(\Phiaux_{k-1} \circ f_k)'(\Psi_k(e^{i\preimageale_l^j}))
		+
		(\Psi_k'(e^{i\preimageale_l^j}))^2
		(\Phiaux_{k-1} \circ f_k)''(\Psi_k(e^{i\preimageale_l^j}))
		\right|\\
		&\geq
		A^{-2}
		\left|
		(\Phiaux_{k-1} \circ f_k)''(\Psi_k(e^{i\preimageale_l^j}))
		\right|\\
		&=
		A^{-2}
		|f_k''(\Psi_k(e^{i\preimageale_l^j}))|
		\times
		|(\Phiaux_{k-1})'(f_k(\Psi_k(e^{i\preimageale_l^j})))|.
	\end{align*}
	Since
	$\Psi_k(e^{i\preimageale_l^j}) = e^{i\thetaaux_k}$,
	\begin{equation}
		\label{eq:f-second-derivative}
		|f_k''(\Psi_k(e^{i\preimageale_l^j}))|
		=
		f''(1)
		=
		\frac{1 + d(\cparam)}{2\sqrt{1-e^{-\cparam}}}
		\geq
		\frac{\cparam^{-1/2}}{4},
	\end{equation}
	and $|f_k(\Psi_k(e^{i\preimageale_l^j}))| - 1 \geq \cparam^{1/2}$,
	so $|(\Phiaux_{k-1})'(f_k(\Psi_k(e^{i\preimageale_l^j})))|
	\geq \frac{\cparam^{1/2}}{A}$.
	Hence we have a constant lower bound on
	$|(\Phiaux_l)''(e^{i\preimageale_l^j})|$,
	and only need an upper bound on
	$\sup_{z \in \gamma} | (\Phiaux_l)^{(3)}(z) |$.
	By repeated application of the chain rule with the same decomposition
	of $\Phiaux_l$,
	similarly to the above we find
	\[
	| (\Phiaux_l)^{(3)}(z) | \leq A \cparam^{-1}
	\]
	for all $z \in \gamma$.
	Hence the right-hand side of \eqref{eq:aux-multi-comparison}
	is almost surely bounded by
	$A \cparam^{-1} D$,
	and so
	$
	\P( \thetaaux_{l+1} = \preimagemulti_l^j )
	=
	e^{O(D/\cparam)}
	\P( \theta_{l+1} = \preimagemulti_l^j ).
	$
\end{proof}

\section{Multinomial to Laplacian path model}
\label{sec:singular-methods}

The step from the multinomial model
to the LPM is the most difficult
and technically involved part of the proof
of Theorem~\ref{thm:ale-lpm},
so deserves its own section.
We must track the movement of the preimages of the tips
in the two models,
which are governed by Loewner's equation
close to the driving function
where it is singular,
and the second derivative of the maps $\Phi_n^{\mathrm{multi}}$
and $\Phi_{t}^{\mathrm{LPM}}$,
which are governed by the second spatial derivative
of the same singular PDE.

As the Laplacian path model is generated by a driving \emph{measure}
rather than a function, we will first specify what we mean by convergence
of driving measures.

\begin{definition}
	\label{def:wasserstein}
	Given a metric space $(X,d)$,
	and two Borel measures $\mu_1, \mu_2$ on $X$,
	the bounded Wasserstein distance
	$d_{\mathrm{BW}}(\mu_1, \mu_2)$
	is defined
	\[
	d_{\mathrm{BW}}(\mu_1, \mu_2)
	=
	\sup_{\varphi \in \mathcal{H}}
	\left|
	\int_{X} \varphi \,\d\mu_1
	-
	\int_{X} \varphi \,\d\mu_2
	\right|,
	\]
	where $\mathcal{H} = \{ \varphi \in C(X) : \| \varphi \|_{\mathrm{Lip}} + \| \varphi \|_{\infty} \leq 1 \}$,
	for
	\[
	\| \varphi \|_{\mathrm{Lip}} = \sup_{x \not= y} \frac{|\varphi(x) - \varphi(y)|}{d(x,y)},
	\quad
	\| \varphi \|_{\infty} = \sup_{x} |\varphi(x)|.
	\]
\end{definition}

\begin{proposition}
	\label{thm:weak-convergence}
	Let $(X,d)$ be a separable metric space,
	then $d_{\mathrm{BW}}$ metrises weak convergence of finite measures on $X$,
	i.e.\ if $\mu$, $\mu_n$ for $n \in \N$ are finite measures on $X$,
	then $\mu_n \Rightarrow \mu$ if and only if $d_{\mathrm{BW}}(\mu_n, \mu) \to 0$ as $n \to \infty$.
\end{proposition}
\begin{proof}
	See Theorem 11.3.3 of \cite{dudley}.
\end{proof}

\begin{remark}
	Weak convergence $\mu_n \Rightarrow \mu$
	of measures on $X$ can also be implied by convergence
	$\int_X \varphi \,\d \mu_n \to \int_X \varphi \,\d \mu$
	for bounded continuous functions $\varphi$.
	We use the smaller space of test functions
	$\varphi \in \mathcal{H}$ in this section because
	it is substantially easier to prove convergence
	of the integrals for our two measures in this case,
	as we make use of the bound on $\| \varphi \|_{\mathrm{Lip}}$
	in the proof of Corollary \ref{thm:weak-conv-multi-lpm},
	and this still suffices to imply weak convergence of the measures.
\end{remark}

\begin{remark}
	We can view each driving measure $(\mu_t)_{t \in [0,T]}$
	as a single probability measure $\mu$ on the cylinder $S = \T \times [0,T]$
	given by $\mu_t \otimes m_{[0,T]}$
	where $m_{[0,T]}$ is normalised Lebesgue measure on $[0,T]$.
	Then by Proposition 1 of \cite{anisotropic-hl0},
	weak convergence of these measures on the cylinder implies convergence
	of the corresponding clusters in the Carath\'{e}odory topology.
\end{remark}

We will use Proposition~\ref{thm:weak-convergence}
to deduce convergence of the cylinder measures corresponding to
the multinomial model and LPM
by showing that the preimages of the tips are close
and the second derivatives are close at these preimages.

\begin{definition}
	\label{def:delta-x-p}
	Let $$x_j(t) := e^{i\phi^j_t}
	\quad\text{ and }
	\quad \bar x_j(t) := e^{i\bar\phi^j_t}$$
	where $\phi^j_t$ and $\bar \phi^j_t$
	are the preimages defined in the statement
	of Theorem~\ref{thm:ale-lpm}.
	
	For $t \in [0,T]$ let
	\begin{align}
		\label{eq:indicator}
		I^j_t := 1\{ \xi_t = x_j(t) \} = 1\{ \theta_{\rdown{t/\cparam}+1} = \phi^j_{\rdown{t/\cparam}\cparam} \} = I^j_{n\cparam},
	\end{align}
	for $t \in [n\cparam,(n+1)\cparam)$,
	and note that $\E( I^l_{n\cparam} | \theta_1, \dots, \theta_n) = p^l_{n\cparam}$.
	
	For all $t \leq T$ define
	\[
	\delta_\mathrm{x}(t) :=
	\sup_{s \leq t} \sum_{j=1}^{k}
	| x_j(s) - \bar x_j(s) |,
	\]
	and
	\[
	\delta_{\mathrm{p}}(t) :=
	\sup_{s \leq t} \sum_{j=1}^{k}
	\left| \Phi_s''(x_j(s)) - \bar\Phi_s''(\bar x_j(s)) \right|.
	\]
	Note that for some constant $A$,
	$|p_t^j - \bar p_t^j| \leq A \delta_{\mathrm{p}}(t)$
	for all $j \in \{1, \dots, k\}$ and all $t \leq T$.
	
	Finally, let
	\[
	\delta_{\mathrm{total}}(t) := \delta_{\mathrm{x}}(t) + \delta_{\mathrm{p}}(t).
	\]
\end{definition}

\begin{theorem}
	\label{thm:multi-to-lpm}
	Under the same assumptions as in Theorem~\ref{thm:ale-lpm},
	there exists a universal constant $R$ and a constant $A = A(T,L,k,K_0,\eta)$
	such that
	\begin{align*}
		\P\left[
		\delta_\mathrm{total}(T)
		\leq A\cparam^{1/2R}
		\right]
		\to 1
	\end{align*}
	as $\cparam \to 0$.
\end{theorem}

\begin{corollary}
	\label{thm:weak-conv-multi-lpm}
	If $\mu_t = \delta_{\theta_{\rdown{t/\cparam}+1}}$ for the multinomial model
	and $\bar \mu_t = \sum_{j=1}^{k} \bar p^j_t \delta_{\bar \phi^j_t}$
	for the Laplacian path model,
	then
	\begin{align*}
		d_{\mathrm{BW}}(\mu_t \otimes m_{[0,T]}, \bar\mu_t \otimes m_{[0,T]})
		\to 0
	\end{align*}
	in probability as $\cparam \to 0$.
	Hence $\mu_t \otimes m_{[0,T]}$ converges in distribution,
	as a random element of the space of measures on $S = \T \times [0,T]$,
	to $\bar \mu_t \otimes m_{[0,T]}$.
\end{corollary}

To derive Corollary~\ref{thm:weak-conv-multi-lpm}
from Theorem~\ref{thm:multi-to-lpm},
we also use Proposition~\ref{thm:multi-fixed-crude},
which we will prove in the following section.

\begin{proof}[Proof of Corollary~\ref{thm:weak-conv-multi-lpm}]
	Let $\varphi \in \mathcal{H}$. For simplicity assume that $T$
	is an integer multiple of $\cparam$.
	Then
	\begin{align*}
		\int_0^T &\int_{\T} \varphi \,\d \mu_t\,\d t
		- \int_0^T \int_{\T} \varphi \,\d \bar\mu_t\,\d t
		=
		\sum_{j=1}^k \int_0^T
		\left(
		I^j_t \varphi(x_j(t)) - \bar p^j_t \varphi( \bar x_j(t) )
		\right)
		\d t\\
		&=
		\sum_{j=1}^k \int_0^T
		\left(
		(I^j_t - p^j_t) \varphi(x_j(t))
		+
		(p^j_t - \bar p^j_t) \varphi(x_j(t))
		+
		\bar p^j_t (\varphi(x_j(t)) - \varphi(\bar x_j(t)))
		\right)
		\d t.
	\end{align*}
	The latter two terms are small with high probability
	using Theorem~\ref{thm:multi-to-lpm},
	so we only need to bound the ``martingale'' terms $\int_0^T (I^j_t - p^j_t)\varphi(x_j(t))\,\d t$ for each $j$.
	We can write each term as
	\begin{align*}
		\sum_{n=0}^{\frac{T}{\cparam}-1}
		\left[
		(I^j_{n\cparam}-p^j_{n\cparam})
		\int_{n\cparam}^{(n+1)\cparam}
		\varphi(x_j(t))
		\,\d t
		+
		\int_{n\cparam}^{(n+1)\cparam}
		(p^j_{n\cparam} - p^j_t)\varphi(x_j(t))
		\,\d t
		\right].
	\end{align*}
	Since $|p^j_{n\cparam} - p^j_t| \leq A \cparam$
	by Proposition~\ref{thm:multi-fixed-crude},
	we have $|\sum_{n=0}^{\frac{T}{\cparam}-1} \int_{n\cparam}^{(n+1)\cparam} (p^j_{n\cparam} - p^j_t) \varphi(x_j(t))\,\d t| \leq AT\cparam$.
	For the remaining term, we would like to simply compute second moments and
	so show the term is small,
	but need some way of doing this uniformly in $\varphi$.
	Choose a large $N$,
	and for simplicity assume that $T/\cparam$ is an integer multiple of $N$.
	Then write
	\begin{align*}
		&\sum_{m=1}^{N} \sum_{n=\frac{(m-1)T}{\cparam N}}^{\frac{mT}{\cparam N}-1}
		(I^j_{n\cparam} - p^j_{n\cparam})\int_{n\cparam}^{(n+1)\cparam}\varphi(x_j(t))\,\d t\\
		&=
		\sum_{m=1}^{N} \cparam\varphi(x_j(mT/N)) \sum_{n=\frac{(m-1)T}{\cparam N}-1}^{\frac{mT}{\cparam N}}
		(I^j_{n\cparam} - p^j_{n\cparam})
		+ \delta,
	\end{align*}
	where $|\delta| \leq \frac{AT^2}{N}$ using $\| \varphi\|_{\mathrm{Lip}}\leq 1$.
	Now we will be able to find a bound independent of $\varphi$, as
	\begin{align*}
		\E \left(
		\cparam
		\sum_{n=\frac{(m-1)T}{\cparam N}}^{\frac{mT}{\cparam N}-1}
		(I^j_{n\cparam} - p^j_{n\cparam})
		\right)^2
		&\leq
		4\cparam^2 \frac{T}{\cparam N}
		= \frac{4T\cparam}{N}
	\end{align*}
	by conditional independence of the increments $I^j_{n\cparam} - p^j_{n\cparam}$.
	Define the event
	\begin{align*}
		E_m = \left\{
		\left|
		\cparam
		\sum_{n=\frac{(m-1)T}{\cparam N}-1}^{\frac{mT}{\cparam N}}
		(I^j_{n\cparam} - p^j_{n\cparam})
		\right|
		\geq
		\frac{2\sqrt{T}\cparam^{1/4}}{N}
		\right\},
	\end{align*}
	then $\P(E_m) \leq \cparam^{1/2}N$,
	and so
	\begin{align*}
		\P \left(
		\sup_{\varphi \in \mathcal{H}}
		\left|
		\sum_{m=1}^{N} \cparam\varphi(x_j(mT/N)) \sum_{n=\frac{(m-1)T}{\cparam N}-1}^{\frac{mT}{\cparam N}}
		(I^j_{n\cparam} - p^j_{n\cparam})
		\right|
		\geq 2\sqrt{T} \cparam^{1/4}
		\right)
		&\leq
		\P\left( \bigcup_{m=1}^N E_m \right)\\
		&\leq
		\cparam^{1/2}N^2.
	\end{align*}
	We chose $N$ arbitrarily,
	so if $1 \ll N \ll \cparam^{-1/4}$,
	we have
	\begin{align*}
		\sup_{\varphi \in \mathcal{H}} \left|
		\int_0^T \int_{\T} \varphi \,\d \mu_t \, \d t
		-
		\int_0^T \int_{\T} \varphi \,\d \bar\mu_t \, \d t
		\right|
		\to 0
	\end{align*}
	in probability, as required.
\end{proof}

Our proof of Theorem~\ref{thm:multi-to-lpm}
has two parts:
first we show in Proposition~\ref{thm:multi-fixed-crude}
that $t \mapsto (x_j(t), \bar x_j(t), p^j_t, \bar p^j_t)$
is Lipschitz on $[0,T]$, which will allow us to deduce global bounds
on $|x_j(t) - \bar x_j(t)|$ and $|p^j_t - \bar p^j_t|$
by showing that they are close at a sufficiently dense set of times.

Then in Proposition~\ref{thm:x-gronwall-bound}
and Proposition~\ref{thm:p-gronwall-bound},
we find bounds on $\delta_{\mathrm{p}}(t)$
and $\delta_{\mathrm{x}}(t)$ for $t \in [0,T]$
which we can turn into a bound on $\delta_{\mathrm{total}}(T)$
using a version of Gr\"{o}nwall's inequality.

\subsection{Lipschitz bound on the locations and probabilities}

\begin{proposition}
	\label{thm:multi-fixed-crude}
	There exists a constant $A = A(k, K_0, T, \eta)$
	such that,
	almost surely,
	for $0 \leq t, s \leq T$,
	\begin{align*}
		| x_j(t) - x_j(s) | \leq A |t-s|
		\quad\text{ and }\quad
		| p_t^j  -  p_s^j | \leq A |t-s|,
	\end{align*}
	for all $j = 1, \dots, k$
	in the multinomial model.
	For the Laplacian path model, for all $j$ and for $0 \leq t,s \leq T$,
	we also have
	\begin{align*}
		| \bar x_j(t) - \bar x_j(s) | \leq A |t-s|
		\quad\text{and}\quad
		| \bar p_t^j  - \bar p_s^j  | \leq A |t-s|.
	\end{align*}
\end{proposition}

\begin{proof}
	Equation (2.7) of \cite{lpm}
	gives a useful expression for the evolution of $x_j(t)$
	over time:
	\begin{align}
		\label{eq:lipschitz-preimage-deriv}
		\frac{\d}{\d t} x_j(t)
		&=
		\left\{\begin{matrix}
			0 & \text{if $\xi_t = x_j(t)$,}\\
			- x_j(t) \frac{x_j(t) + x_l(t)}{x_j(t) - x_l(t)}
			& \text{if $\xi_t = x_l(t)$ for $l \not= j$}
		\end{matrix}\right.,
	\end{align}
	and so as the denominator is bounded below by $L = L(k,K_0,T,\eta)$, we have
	\begin{align*}
		|x_j(t) - x_j(s)| \leq \frac{2}{L} | t-s |.
	\end{align*}
	
	Controlling how $\Phi_t''(x_j(t))$
	changes over time will involve a spatial Taylor expansion.
	Therefore we need to control
	the third derivative,
	so first we show that the third derivative
	at $x_j(t)$ takes a particularly convenient form
	if $\xi_{t-} = x_j(t)$.
	If $f_{s,x}(z)$ attaches a particle of capacity $s > 0$ at $x \in \T$,
	then we can calculate from \eqref{f-derivative}
	that $f_{s,1}''(1) = \frac{1 + d(s)}{2\sqrt{1-e^{-s}}}$
	and $f_{s,1}^{(3)}(1) = \frac{-3(1 + d(s))}{2\sqrt{1-e^{-s}}} = -3f_{s,1}''(1)$,
	therefore $f_{s,x}^{(3)}(x) = -\frac{3}{x}f_{s,x}''(x)$.
	Suppose $\Phi_t$ can be written as $\Phi_{t-s} \circ f_{s,x_j(t)}$
	for some $s \in (0,\cparam]$,
	then in this case we have $x_j(t-s) = x_j(t)$
	and $\Phi_{t-s}'(x_j(t)) = 0$.
	Using the chain rule to calculate $(\Phi_{t-s} \circ f_{s,x_j(t)})''(x_j(t))$
	and $(\Phi_{t-s} \circ f_{s,x_j(t)})^{(3)}(x_j(t))$,
	we also obtain $\Phi_t^{(3)}(x_j(t)) = -\frac{3 \Phi_t''(x_j(t))}{x_j(t)}$.
	
	Suppose $t,t+\delta \in [m\cparam, (m+1)\cparam]$ with $\delta > 0$.
	We will show that
	$|\Phi_{t+\delta}''(x_j(t+\delta)) - \Phi_{t}''(x_j(t))| = O(\delta)$,
	which is sufficient to show that $t \mapsto \Phi_t''(x_j(t))$ is Lipschitz on $[0,T]$.
	There are two cases: either (a) $\xi_s = x_j(s)$
	for all $s \in [m\cparam, (m+1)\cparam)$
	or (b) $\xi_s = x_l(s)$ for some $l \not= j$ over the same interval.
	
	In case (a),
	$\Phi_{t+\delta} = \Phi_t \circ f_{\delta,x_j(t)}$.
	Although we do not necessarily have 
	$\Phi_{t}^{(3)}(x_j(t)) = -\frac{3 \Phi_{t}''(x_j(t))}{x_j(t)}$,
	we \emph{do} have
	$\Phi_{t+\delta^2}^{(3)}(x_j(t)) = -\frac{3 \Phi_{t+\delta^2}''(x_j(t))}{x_j(t)}$,
	so let $t' = t + \delta^2$ and $\delta' = \delta - \delta^2$, then write
	$\Phi_{t+\delta} = \Phi_{t'} \circ f_{\delta', x_j(t)}$.
	To simplify the expressions write $x = x_j(t)$.
	Using the chain rule and the fact that $\Phi_{t'}'(x) = 0$,
	we obtain
	\begin{align}
		\nonumber
		\Phi_{t+\delta}''(x)
		&= \Phi_{t'}'(f_{\delta',x}(x))
		f_{\delta',x}''(x) + 0\\
		\label{eq:second-deriv-chain-rule}
		&= \Phi_{t'}'(f_{\delta',x}(x))
		\frac{1}{x} \frac{1 + d(\delta')}{2\sqrt{1-e^{-\delta'}}}.
	\end{align}
	Since
	$f_{\delta',x}(x) - x = x d(\delta')$,
	by Taylor's theorem,
	\begin{align*}
		\Phi_{t'}'(f_{\delta',x}(x))
		&= x d(\delta')\Phi_{t'}''(x)
		+ \frac{x^2 d(\delta')^2}{2}\Phi_{t'}^{(3)}(x)
		+ O(d(\delta')^3)\\
		&= x d(\delta') \left(1 - \frac{3}{2}d(\delta')\right)\Phi_{t'}''(x)
		+ O(d(\delta')^3).
	\end{align*}
	Then substituting this back into \eqref{eq:second-deriv-chain-rule},
	we get
	\begin{align*}
		\Phi_{t+\delta}''(x)
		&= \frac{d(\delta')(1+d(\delta'))}{2\sqrt{1-e^{-\delta'}}}
		\left(1 - \frac{3}{2}d(\delta')\right) \Phi_{t'}''(x) + O(d(\delta')^2).
	\end{align*}
	Since $d(\delta') = 2e^{\delta'}(1 + \sqrt{1 - e^{-\delta'}}) - 2
	= 2\sqrt{\delta'} + O(\delta')$,
	we can calculate
	$\frac{d(\delta')(1+d(\delta'))}{2\sqrt{1-e^{-\delta'}}}
	= 1 + 3\sqrt{\delta'} + O(\delta')$,
	and $1 - \frac{3}{2}d(\delta') = 1 - 3\sqrt{\delta'} + O(\delta')$,
	so 
	\begin{align*}
		\Phi_{t+\delta}''(x)
		&= \Phi_{t'}''(x) + O(\delta').
	\end{align*}
	Then writing $\Phi_{t'} = \Phi_t \circ f_{\delta^2,x}$
	and expanding with the chain rule and Taylor's theorem as above,
	we do not have as much cancellation and get the slightly cruder expression
	\begin{align*}
		\Phi_{t'}''(x)
		= \frac{d(\delta^2)(1+\delta^2)}{2\sqrt{1-e^{-\delta^2}}}
		\Phi_{t}''(x) + O(d(\delta^2)),
	\end{align*}
	but this is good enough. It gives us
	$\Phi_{t'}''(x_j(t)) = \Phi_t''(x_j(t)) + O(\delta)$,
	and so
	$$
	\Phi_{t+\delta}''(x_j(t)) - \Phi_t''(x_j(t)) = O(\delta).
	$$
	
	In case (b),
	by the chain rule we have, for $s \in [t,\delta+t)$,
	\begin{align*}
		\frac{\d}{\d s} \Phi_s''(x_j(s))
		&= \left.\frac{\partial}{\partial s} \Phi_s''(z) \right|_{z=x_j(s)}
		+ \Phi_s^{(3)}(x_j(s)) \frac{\d}{\d s} x_j(s).
	\end{align*}
	Differentiating \eqref{eq:loewner} twice with respect to $z$,
	the first term above is
	\begin{align*}
		\left.\frac{\partial}{\partial s} \Phi_s''(z) \right|_{z=x_j(s)}
		&=
		\Phi_s^{(3)}(x_j(s)) x_j(s) \frac{x_j(s) + \xi_s}{x_j(s) - \xi_s}
		+ \Phi_s''(x_j(s)) \left( 1 - \frac{2\xi_s^2}{(x_j(s) - \xi_s)^2} \right)\\
		&=
		-\Phi_s^{(3)}(x_j(s)) \frac{\d}{\d s} x_j(s)
		+ \Phi_s''(x_j(s)) \left( 1 - \frac{2\xi_s^2}{(x_j(s) - \xi_s)^2} \right)
	\end{align*}
	by \eqref{eq:lipschitz-preimage-deriv}.
	Therefore
	\begin{align*}
		\frac{\d}{\d s} \Phi_s''(x_j(s))
		&= \Phi_s''(x_j(s)) \left( 1 - \frac{2\xi_s^2}{(x_j(s) - \xi_s)^2} \right),
	\end{align*}
	which is bounded.
	
	We have now shown that $t \mapsto \Phi_t''(x_j(t))$ is Lipschitz
	on $[0,T]$,
	but to conclude that $t \mapsto p^j_t$ is Lipschitz
	we need $|\Phi_t''(x_j(t))|$ to be bounded below.
	This follows from the same argument used in the proof of
	Proposition~\ref{thm:aux-multi-coupling}.
\end{proof}

\subsection{Gr\"{o}nwall-type bound on the total error}
\label{sec:Gronwall}

We will prove the global bound in Theorem~\ref{thm:multi-to-lpm}
using a singular version of Gr\"{o}nwall's inequality
from \cite{singular-gronwall} applied to
$\sum_{j=1}^{k} |x_j(t)- \bar x_j(t)|$
and $\sum_{j=1}^{k} |p^j_t - \bar p^j_t|$.

\begin{proposition}
	\label{thm:x-gronwall-bound}
	Under the same conditions as Theorem~\ref{thm:multi-to-lpm},
	there exist constants $A = A(T,k,K_0,\eta) < \infty$
	such that
	\begin{align*}
		\P\left[ \delta_{\mathrm{x}}(t) \leq A\int_0^t \delta_{\mathrm{total}}(s)\,\d s + 3 \cparam^{1/4}
		\text{ for all } t \in [0,T] \right]
		\to 1
	\end{align*}
	as $\cparam \to 0$.
\end{proposition}

\begin{proposition}
	\label{thm:p-gronwall-bound}
	There exist constants $A = A(T,k,K_0,\eta) < \infty$
	and $R > 0$
	such that
	\[
	\P\left[ \delta_{\mathrm{p}}(t) \leq A\int_0^t \frac{\delta_{\mathrm{total}}(s)}{(t-s)^{51/100}}\d s + A \cparam^{1/2R}
	\text{ for all } t \in [0,T]
	\right] \to 1
	\]
	as $\cparam \to 0$.
\end{proposition}

We have stated these separately because
Proposition~\ref{thm:x-gronwall-bound}
is much simpler to prove,
while Proposition~\ref{thm:p-gronwall-bound}
will be divided into several steps.

\begin{remark}
	In the arguments below we apply Loewner's equation to points
	on the boundary of our domain,
	although in \eqref{eq:loewner-with-measure} we said that it applies
	to points in the \emph{open} exterior disc.
	We could replace the points on the boundary with extremely close
	points in $\Delta$ and then take a limit
	to resolve this formally,
	and working on the boundary directly
	can be seen as a shorthand for this procedure.
\end{remark}

\begin{proof}[Proof of Proposition~\ref{thm:x-gronwall-bound}]
	Throughout the proof, 
	$A$ represents a constant
	which may change from line to line,
	but all occurrences have a common upper bound
	depending only on $T$, $L$, $k$, $K_0$ and $\eta$.
	
	Equation (2.7) in \cite{lpm} shows that
	the movement of the preimages $\bar \phi^j_t$ in the Laplacian path model is
	determined by
	\begin{align}
		\label{eq:lpm-preimage-movement}
		\frac{\d}{\d t} \bar x_j(t)
		&=
		-\sum_{l \not= j}
		\bar p^j_t
		\bar x_j(t) \frac{\bar x_j(t) + \bar x_l(t)}{\bar x_j(t) - \bar x_l(t)}
	\end{align}
	and similarly for the multinomial model,
	\begin{align}
		\label{eq:multi-preimage-movement}
		\frac{\d}{\d t} x_j(t)
		&=
		\left\{\begin{matrix}
			0 & \text{if $\xi_t = x_j(t)$,}\\
			- x_j(t) \frac{x_j(t) + x_l(t)}{x_j(t) - x_l(t)}
			& \text{if $\xi_t = x_l(t)$ for $l \not= j$}
		\end{matrix}\right..
	\end{align}
	Write
	\begin{align*}
		\lambda^{j,l}_t = \left\{ \begin{matrix} 0 & \text{if } j=l \\
			-x_j(t) \frac{x_j(t)+x_l(t)}{x_j(t)-x_l(t)}
			& \text{otherwise}
		\end{matrix} \right.,
	\end{align*}
	and likewise
	\begin{align*}
		\bar \lambda^{j,l}_t = \left\{ \begin{matrix} 0 & \text{if } j=l \\
			- \bar x_j(t) \frac{\bar x_j(t) + \bar x_l(t)}{\bar x_j(t) - \bar x_l(t)}
			& \text{otherwise.}
		\end{matrix} \right.
	\end{align*}
	Therefore for $t \in [0,T]$ we can write,
	integrating \eqref{eq:lpm-preimage-movement}
	and \eqref{eq:multi-preimage-movement}
	with respect to time,
	\begin{align*}
		\bar x_j(t) &= \bar x_j(0) + \sum_{l \not= j} \int_0^t
		\bar \lambda^{j,l}_s \bar p^l_s
		\, \d s,\\
		x_j(t) &= x_j(0)
		+ \sum_{l\not=j} \int_{0}^{t}
		\lambda^{j,l}_s I^l_s
		\,\d s,
	\end{align*}
	where $I^l_s$ is as defined in \eqref{eq:indicator}.
	To bound the size of $\bar x_j(t) - x_j(t)$,
	we can expand the difference of the integrands
	into three terms
	\begin{align*}
		\lambda^{j,l}_s I^l_s - \bar \lambda^{j,l}_s \bar p^l_s
		&=
		(\lambda^{j,l}_s - \bar \lambda^{j,l}_s) I^l_s
		+
		\bar \lambda^{j,l}_s (I^l_s - p^l_s)
		+
		\bar \lambda^{j,l}_s (p^l_s - \bar p^l_s).
	\end{align*}
	As the denominators in the above definitions of $\lambda^{j,l}_t$
	and $\bar \lambda^{j,l}_t$
	are bounded below,
	we can linearise
	$|\lambda^{j,l}_s - \bar \lambda^{j,l}_s| \leq A\left( |x_j(s) - \bar x_j(s)| + |x_l(s) - \bar x_l(s)|\right) \leq A \delta_{\mathrm{x}}(s)$, and
	$\sum_{l\not=j}|\bar \lambda^{j,l}_s| |p_s^l - \bar p_s^l| \leq A \delta_{\mathrm{p}}(s)$,
	and so by the triangle inequality
	\begin{align}
		\label{eq:x-diff-gronwall-mg-sum}
		| x_j(t) - \bar x_j(t) |
		&\leq
		A \int_0^t \delta_{\mathrm{total}}(s)\,\d s
		+ \sum_{l\not=j} \left| \int_0^t \bar \lambda^{j,l}_s (I^l_s - p^l_s)\,\d s \right|.
	\end{align}
	To finish our proof we need to bound the martingale-type terms
	$\int_0^t \bar \lambda^{j,l}_s (I^l_s - p^l_s)\,\d s$.
	Note that for $n \in \N \cup \{0\}$ with $n \leq T/\cparam$,
	by Proposition~\ref{thm:multi-fixed-crude},
	\begin{align*}
		\int_{n\cparam}^{(n+1)\cparam \wedge t} \bar \lambda_{s}^{j,l} (I^l_s - p^l_s) \,\d s
		&=
		\int_{n\cparam}^{(n+1)\cparam \wedge t} \bar \lambda_{s}^{j,l} (I^l_{n\cparam} - p^l_{n\cparam}) \,\d s + \int_{n\cparam}^{(n+1)\cparam \wedge t} \bar \lambda_{s}^{j,l} (p^l_{n\cparam} - p^l_s)\, \d s\\
		&= (I^l_{n\cparam} - p^l_{n\cparam})\int_{n\cparam}^{(n+1)\cparam \wedge t} \bar \lambda_{s}^{j,l} \,\d s + O(\cparam^2).
	\end{align*}
	We will bound each term in the sum on the right-hand side of \eqref{eq:x-diff-gronwall-mg-sum},
	so fix some $l \not= k$.
	Let $X_n := (I^l_{n\cparam} - p^l_{n\cparam})\int_{n\cparam}^{(n+1)\cparam} \bar \lambda_{s}^{j,l} \,\d s$
	and note that $X_n$ is a bounded martingale increment,
	$\E[X_n | \theta_1, \dots, \theta_n] = 0$,
	so define the martingale $M_n := \sum_{i=0}^n X_n$.
	Note that
	\begin{align*}
		\sup_{t \leq T} \left|\int_0^t \bar \lambda^{j,l}_s (I^l_s - p^l_s)\,\d s - M_{\rdown{t/\cparam}}\right| = O(\cparam),
	\end{align*}
	so if we can bound $\max_{n \leq T/\cparam} |M_n|$ then our result follows.
	
	The increments $X_n$ are bounded by $A\cparam$ for a constant $A$
	depending only on $T, K_0$ and $\eta$,
	so $\E[ |M_{\rdown{T/\cparam}}|^2 ] \leq A^2 \cparam$.
	Therefore, by Doob's submartingale inequality,
	\begin{align*}
		\mathbb{P}\left[ \max_{n \leq T/\cparam} |M_n| \geq \frac{1}{k}\cparam^{1/4} \right]
		\leq \frac{A \cparam}{\cparam^{1/2}},
	\end{align*}
	for a constant $A$.
	So for sufficiently small $\cparam$,
	on an event of probability at least $1 - A \cparam^{1/2}$
	we have $\sum_{l\not=j} \left| \int_0^t \bar \lambda^{j,l}_s (I^l_s - p^l_s)\,\d s \right| \leq 3\cparam^{1/4}$
	and the result follows.
\end{proof}

In the proof of Proposition~\ref{thm:p-gronwall-bound}
we will use some analytic techniques from \cite[Section 4.2]{lpm}.
First we require a few further definitions,
mostly following the notation of \cite{lpm}.
\begin{definition}
	\label{def:transition-maps}
	Fix $t \in [0,T]$,
	and for $s \leq t$ define the \emph{transition maps}
	\begin{align}
		\label{eq:transition-maps}
		h_s := \Phi_s^{-1} \circ \Phi_t,
		\quad
		\bar h_s := \bar \Phi_s^{-1} \circ \bar \Phi_t.
	\end{align}
	Several transition maps for an ALE-like cluster
	are shown in \autoref{fig:transition-map}.
	Note that if $t = n \cparam$ for some $n \in \N$,
	then $h_0 = f_1 \circ \dots \circ f_n$
	and $h_{k\cparam} = f_{k+1} \circ \dots \circ f_n$
	for $k \in \N$ with $k < n$.
	
	For $s \in [0,t]$ and $j \in \{1, \dots, k\}$,
	define
	\begin{equation}
		\label{eq:wjs}
		w_j(s) := h_s( x_j(t) ),
		\quad \bar w_j(s) = \bar h_s( \bar x_j(t) ).
	\end{equation}
	Also define
	\begin{equation}
		\label{eq:kappajs}
		\kappa_j(s) := h_s''(x_j(t)),
		\quad \bar \kappa_j(s) := \bar h_s''(\bar x_j(t)).
	\end{equation}
\end{definition}

\begin{figure}[h]
	\centering
	\begin{tikzpicture}
		\draw node (circle) at (0,0) {\includegraphics[width=0.4\linewidth,trim=110 60 100 60, clip]{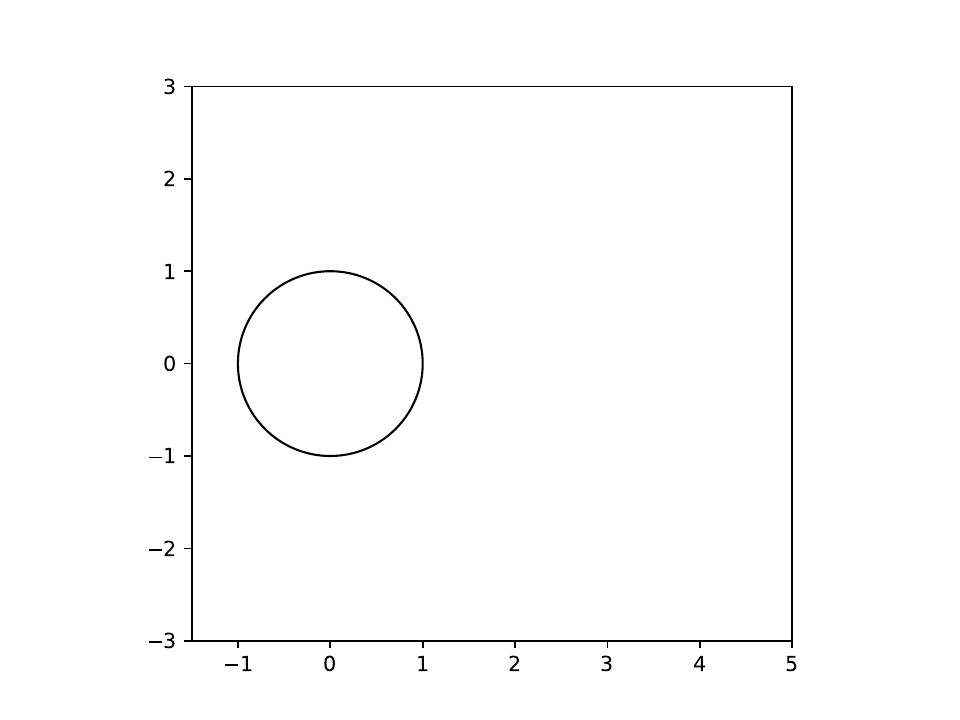}};
		\draw node (full) at (9.0,0) {\includegraphics[width=0.4\linewidth,trim=110 60 100 60, clip]{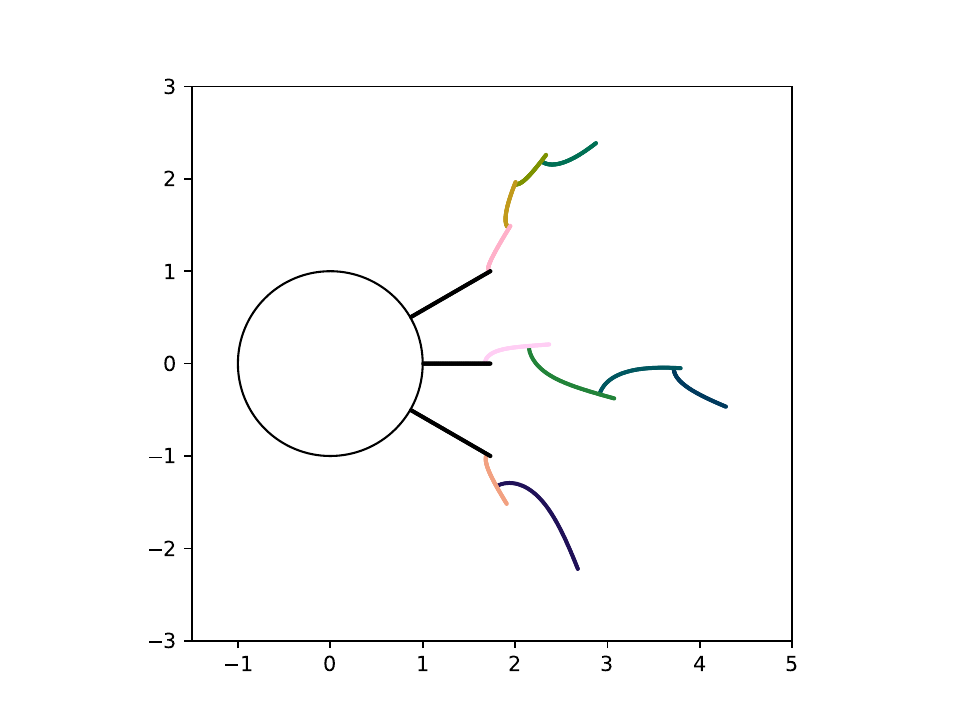}};
		\draw node (h0) at (4,3.6) {\includegraphics[width=0.4\linewidth,trim=110 60 100 60, clip]{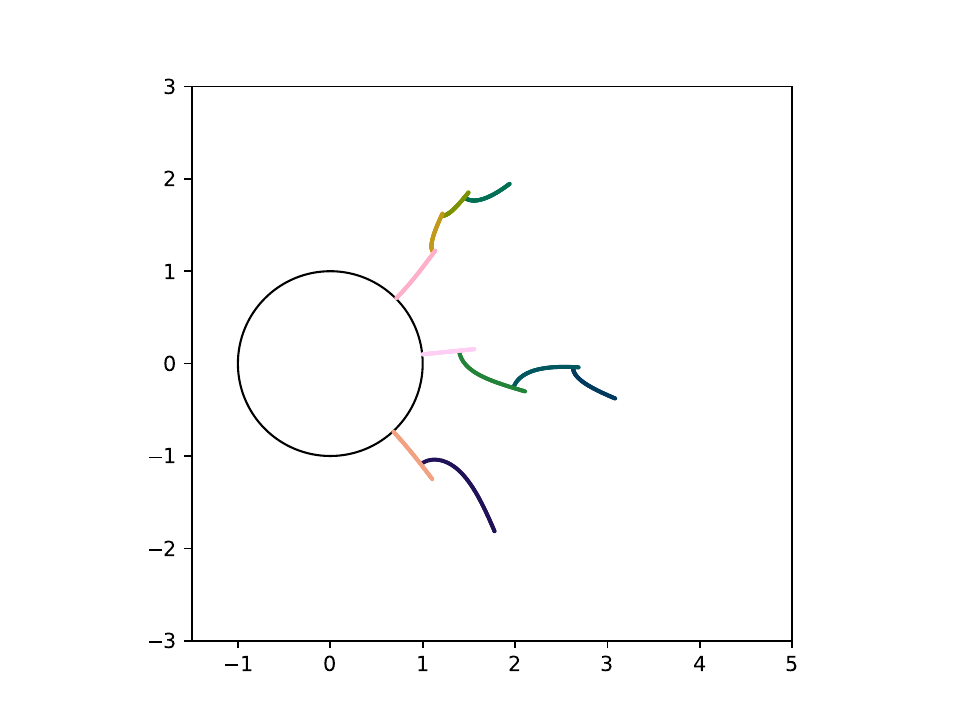}};
		\draw node (h5c) at (4,0) {\includegraphics[width=0.4\linewidth,trim=110 60 100 60, clip]{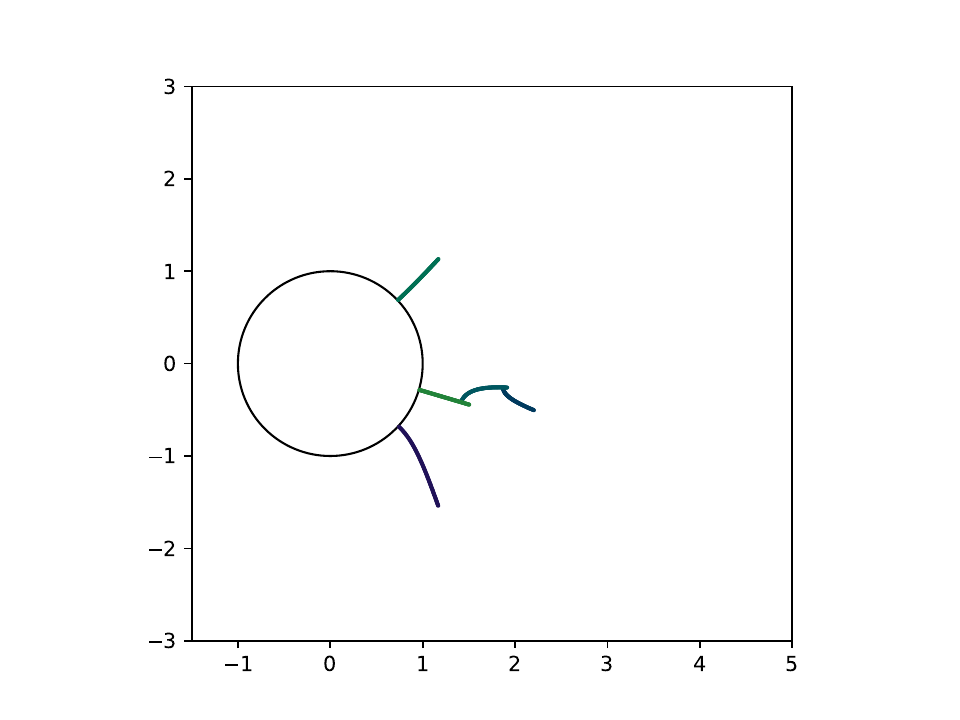}};
		\draw node (h-almost-9) at (4,-2.8) {\includegraphics[width=0.4\linewidth,trim=110 100 100 60, clip]{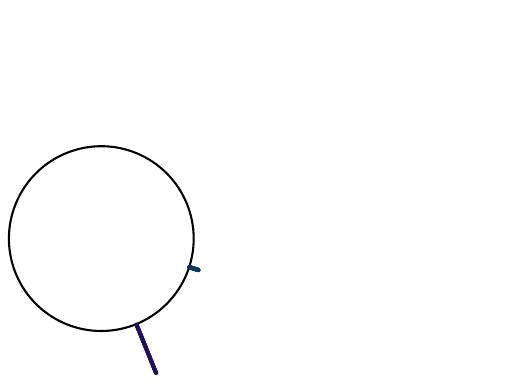}};
		\draw[->] (-0.7,0.9) -- node[anchor=south east] {$h_0$} (0.9,2.5);
		\draw[->] (-0.5,0) -- node[anchor=south] {$h_{5\cparam}$} (0.8,0);
		\draw[->] (-0.7,-0.9) -- node[anchor=south west] {$h_{8.95\cparam}$} (0.9,-2.5);
		\draw[->] (4.3,2.1) -- node[anchor=south west] {$\Phi_0$} (5.8,1.0);
		\draw[->] (4.3,0) -- node[anchor=south] {$\Phi_{5\cparam}$} (5.7,0);
		\draw[->] (3.5,-2.6) -- node[anchor=north west] {$\Phi_{8.95\cparam}$} (5.8,-1.0);
	\end{tikzpicture}
	\caption{\label{fig:transition-map}
		Several transition maps as in Definition~\ref{def:transition-maps}
		for an ALE-like cluster with ten particles
		attached to an initial three slits.
		Notice that none of the transition maps attach the initial three slits
		to the circle,
		and that the image of $\T$ under the transition map $h_t$
		is not monotone in $t$:
		although the capacity is strictly decreasing, the slits are also moving around.
	}
\end{figure}

\begin{remark}
	The transition map $h_s$
	is related to the backward equation \eqref{eq:backwards-equation}.
	If started from the trivial initial condition
	$\Phi_0 = \mathrm{Id}_\Delta$,
	then for $t=T$,
	$h_s = u_{T-s}$ for all $s \in [0,T]$,
	where $u$ is the solution to \eqref{eq:backwards-equation}.
	We have, however, used the notation $h_s$ rather than $u_{t-s}$
	to avoid confusion,
	because our $\Phi_0$ is non-trivial and the endpoint may be any $t \in (0,T]$,
	not just the final $T$.
\end{remark}

\begin{definition}
	\label{def:Hjn}
	Let $j \in \{1,\dots,k\}$ and $n \leq \rdown{T/\cparam}$.
	Let $p > 0$ be such that
	$p^j_{m\cparam} \geq p$ for all $m \leq \rdown{T/\cparam}$
	almost surely
	(we know that such a deterministic constant $p$ must exist).
	Define the event
	\begin{align*}
		H_j(n) := \left\{
		I^j_{n\cparam} = 1 \text{ and, for all } m \leq n,  
		\sum_{l=m}^{n} I^j_{l\cparam} \geq \frac{1}{2}p(n-m+1)
		\right\}.
	\end{align*}
\end{definition}

\begin{remark}
	For the rest of Section~\ref{sec:singular-methods}
	we will repeatedly condition on $H_j(n)$,
	so the reader may wish to make a note of Definition~\ref{def:Hjn}
	or mark its location.
\end{remark}

\begin{remark}
	Since $(I^j_{l\cparam})_{l\leq \rdown{T/\cparam}}$
	dominates a collection of iid Bernoulli($p$) random variables,
	$\P(H_j(n))$ is bounded below by a constant
	depending only on $T$, $\eta$ and $K_0$.
\end{remark}

We will prove Proposition~\ref{thm:p-gronwall-bound}
via several intermediate results,
all stated here before we prove them.

\begin{lemma}
	\label{thm:derivative-difference-bound}
	Let $j \in \{1, \dots, k\}$,
	and let $t = n\cparam$ for some $n \leq \rdown{T/\cparam}$.
	Let the event $H_j(n)$ be as in Definition~\ref{def:Hjn}.
	There exists a constant $A$ depending on $K_0, T$ and $\eta$
	such that, conditional on $H_j(n)$ we almost surely have
	\begin{align}
		\nonumber
		\left|
		\Phi_{t}''(x_j(t)) - \bar\Phi_{t}''(\bar x_j(t))
		\right|
		&\leq
		A |(w_j(0) - x_j(0)) - (\bar w_j(0) - \bar x_j(0))|\\
		\label{eq:2nd-deriv-bound-update-1}
		&\phantom{\leq}+ A\left|\kappa_j(0)(w_j(0) - x_j(0)) - \bar\kappa_j(0)(\bar w_j(0) - \bar x_j(0))\right|.
	\end{align}
\end{lemma}

\begin{lemma}
	\label{thm:integral-linearisation}
	Let $j$, $n$ and $t$ be as above.
	Conditional on $H_j(n)$,
	\begin{align}
		\label{eq:unsimplified-second-deriv}
		\log[\kappa_j(0)( w_j(0) - x_j(0) )]
		&=
		2t - \sum_{l=1}^{k} \int_0^t I^l_s Q_s^{j,l}
	\end{align}
	almost surely, where
	\begin{align}
		\label{eq:qjj}
		Q^{j,j}_s &:= \frac{-3x_j(s)}{w_j(s) - x_j(s)},
		\intertext{and}
		\label{eq:qjl}
		Q^{j,l}_s &:= \frac{2x_l(s)^2}{(w_j(s)-x_l(s))^2} + \frac{2x_l(s)^2}{(w_j(s)-x_l(s))(x_j(s)-x_l(s))}
	\end{align}
	for $l \not= j$.
	Similarly,
	\begin{align}
		\label{eq:unsimplified-second-deriv-lpm}
		\log[\bar\kappa_j(0)( \bar w_j(0) - \bar x_j(0) )]
		&=
		2t - \sum_{l=1}^{k}\int_0^t \bar p^l_s \bar Q^{j,l}_s \,\d s,
	\end{align}
	where the quantities $\bar Q_s^{j,l}$
	are the LPM equivalents of \eqref{eq:qjj} and \eqref{eq:qjl}.
	Then we can bound the final term in \eqref{eq:2nd-deriv-bound-update-1},
	\begin{align}
		\label{eq:linearised-kappa-w}
		\left|\kappa_j(0)(w_j(0) - x_j(0)) - \bar\kappa_j(0)(\bar w_j(0) - \bar x_j(0))\right|
		\leq
		A \sum_{l=1}^{k} \left| \int_0^t (I_s^l Q_s^{j,l} - \bar p^l_s \bar Q^{j,l}_s) \,\d s \right|.
	\end{align}
\end{lemma}

We can linearise the integrand
in the line above to bound it by something involving
$|(w_j(s)-x_j(s)) - (\bar w_j(s) - \bar x_j(s))|$:

\begin{proposition}
	\label{thm:high-probability-integral-bound}
	For some constant $R \geq 9$
	there exists an event $E_{\mathrm{p}}^t$
	with $\P[E_{\mathrm{p}}^t \,|\, H_j(n) ] \geq 1 - A \cparam^{4/R}\log(T/\cparam)$
	such that on $E_{\mathrm{p}}^t \cap H_j(n)$,
	\eqref{eq:2nd-deriv-bound-update-1}
	is almost surely bounded by a constant multiple of
	\begin{align*}
		\cparam^{\frac{1}{2} - \frac{2}{R}}
		+
		\int_0^t \frac{\delta_{\mathrm{total}}(s)}{\sqrt{t-s}}\,\d s
		+
		|(w_j(0) - x_j(0)) - (\bar w_j(0) - \bar x_j(0))|\\
		+
		\int_0^t \frac{|(w_j(s) - x_j(s)) - (\bar w_j(s) - \bar x_j(s))|}{t-s}\,\d s.
	\end{align*}
\end{proposition}

Now if we can control 
$|(w_j(s)-x_j(s)) - (\bar w_j(s) - \bar x_j(s))|$,
Proposition~\ref{thm:p-gronwall-bound} will follow.

\begin{lemma}
	\label{thm:difference-integral-bound}
	Let $j$, $n$ and $t$ be as above.
	For $s \in [0,t]$
	let $$u_s := (w_j(s) - x_j(s)) - (\bar w_j(s) - \bar x_j(s)).$$
	Then, conditional on $H_j(n)$, $u$ almost surely satisfies the integral inequality
	\begin{align*}
		|u_s| \leq X(s) + A \int_s^t \frac{X(r)}{t-r}\,\d r,
	\end{align*}
	for all $s \in [0,t]$,
	where $$X(s) := A\int_s^t \frac{\delta_{\mathrm{total}}(r)}{\sqrt{t-r}}\,\d r + \sum_{l=1}^k \left| \int_s^t C^l_r (I^l_r - p^l_r)\,\d r \right|,$$
	and the $C^l_r$ are deterministic functions (depending on $t$)
	such that $|C^l_r| \leq A(t-r)^{-1/2}$ for all $r \in [0,t)$.
\end{lemma}

\begin{corollary}
	\label{thm:u-high-probability-bound}
	For $j$, $n$ and $t$ as above, and for the $u$ defined in
	Lemma~\ref{thm:difference-integral-bound},
	conditional on $H_j(n)$,
	on an event of conditional probability at least
	$1 - A\cparam^{4/R}\log(1/\cparam)$,
	where $R$ is some positive universal constant,
	we have
	\begin{align*}
		|u_0|
		&\leq A\int_0^t \frac{\delta_{\mathrm{total}}(s)}{(t-s)^{51/100}}\d s + A\cparam^{1/R}\log(1/\cparam),
		\intertext{and}
		\int_0^t \frac{|u_s|}{t-s}\,\d s
		&\leq
		A\int_{0}^{t} \frac{\delta_{\mathrm{total}}(s)}{(t-s)^{51/100}}
		+
		A\cparam^{1/R}(\log(1/\cparam))^2.
	\end{align*}
\end{corollary}

Now we will prove the above five results in turn.

\begin{proof}[Proof of Lemma~\ref{thm:derivative-difference-bound}]
	By the chain rule we have $\Phi_t''(x_j(t)) = \kappa_j(0)\Phi_0'(w_j(0))$,
	and so
	\begin{align*}
		\Phi_{t}''(x_j(t))
		&=
		\frac{\Phi_0'(w_j(0))}{w_j(0) - x_j(0)}
		\times \kappa(0) (w_j(0) - x_j(0)).
	\end{align*}
	We will repeatedly make use of expansions of the form
	$ab - \bar a \bar b
	= (a - \bar a)b + \bar a (b - \bar b)$
	to linearise the difference of complex expressions.
	Applying this to the above decomposition of $\Phi_t''(x_j(t))$
	and its LPM equivalent, we have
	\begin{align}
		\nonumber
		\left|
		\Phi_{t}''(x_j(t)) - \bar\Phi_{t}''(\bar x_j(t))
		\right|
		&\leq
		\left|
		\frac{\Phi_0'(w_j(0))}{w_j(0)-x_j(0)} - \frac{\Phi_0'(\bar w_j(0))}{\bar w_j(0)-\bar x_j(0)}
		\right|
		|\kappa(0) (w_j(0) - x_j(0))|\\
		\label{eq:use-of-kappa}
		&\mspace{10mu}+
		\left|
		\frac{\Phi_0'(\bar w_j(0))}{\bar w_j(0)-\bar x_j(0)}
		\right|
		\left|
		\kappa_j(0)(w_j(0)-x_j(0))
		-
		\bar\kappa_j(0)(\bar w_j(0)-\bar x_j(0))
		\right|.
	\end{align}
	We claim that both $\left|\frac{\Phi_0'(\bar w_j(0))}{\bar w_j(0)-\bar x_j(0)}\right|$
	and $|\kappa_j(0)(w_j(0) - x_j(0))|$ are bounded.
	The former is bounded by $\sup_{\zeta \in [x_j(0),\bar w_j(0)]} | \Phi_0''(\zeta) |$.
	The Schwarz--Christoffel representation \cite[equation (1)]{gearlike} tells us
	$\Phi_0'(z) = \frac{\Phi_0(z)}{z} S(z)$,
	where $S(z) := \frac{\prod_{l=1}^k (z - x_j(0))}{\sqrt{\prod_{l=1}^{k} (z - b_{+}^l)(z - b_{-}^l) }}$,
	and $b_{\pm}^j$ are the two preimages of the base of the $j$th slit.
	We can differentiate this to verify that $|\Phi_0''(\zeta)|$ is bounded
	when $\zeta$ is not close to any of the $b_{\pm}^l$,
	which is the case for $\zeta \in [x_j(0), \bar w_j(0)]$.
	Next note that $\sum_{l=1}^{k} |Q_{s}^{j,l}| = O((t-s)^{-1/2})$, which is integrable,
	and so we can establish a finite upper bound
	and positive lower bound
	on $|\kappa_j(0)(w_j(0) - x_j(0))|$
	using \eqref{eq:unsimplified-second-deriv}
	(the proof of which does not rely on \eqref{eq:2nd-deriv-bound-update-1}).\\
	
	To simplify the first term
	on the right-hand side of \eqref{eq:use-of-kappa},
	since $\Phi_0'(x_j(0)) = 0$
	we can use the fundamental theorem of calculus to write
	\begin{align*}
		\Phi_0'(w_j(0)) = (w_j(0)-x_j(0))\int_0^1 \Phi_0''( \alpha w_j(0) + (1-\alpha)x_j(0) )\,\d \alpha,
	\end{align*}
	and similarly for $\Phi_0'(\bar w_j(0))$.
	Then, using the Schwarz--Christoffel representation again to bound
	$|\Phi_0^{(3)}(\zeta)|$
	and noting that $\bar x_j(0) = x_j(0)$,
	we can bound the first term on the RHS of \eqref{eq:use-of-kappa} by
	\begin{align*}
		A\int_0^1 \left| \Phi_0''(p w_j(0) + (1-p)x_j(0)) - \Phi_0''(p \bar w_j(0) + (1-p)\bar x_j(0)) \right| \d p
		&\leq
		A |w_j(0) - \bar w_j(0)|,
	\end{align*}
	which is equal to $A|(w_j(0) - x_j(0)) - (\bar w_j(0) - \bar x_j(0))|$,
	so \eqref{eq:2nd-deriv-bound-update-1} holds.
\end{proof}

\begin{proof}[Proof of Lemma~\ref{thm:integral-linearisation}]
	As $t$ is fixed, we will denote
	$\frac{\partial}{\partial s}h_s(z)$
	by $\dot h_s(z)$, etc.
	The main effect of conditioning on $H_j(n)$ is that
	$h_s'(x_j(t)) = 0$ for all $s \in [0,t)$
	(the same is always true for the LPM).
	Where it is unambiguous we will write $x_j = x_j(s)$,
	$w_j = w_j(s)$, etc.
	We will prove \eqref{eq:unsimplified-second-deriv},
	and omit the proof of \eqref{eq:unsimplified-second-deriv-lpm}
	as it is very similar.
	
	For any $z$, $h_s(z)$ is continuous in $s \in [0,t]$
	and the inverse Loewner equation describes
	the evolution of $h_s$ for $0 < s < t$:
	\begin{align}
		\label{eq:inverse-loewner-multinomial}
		\frac{\pd}{\pd s} h_s(z)
		&=
		- h_s(z) \sum_{l=1}^{k}
		I_s^l	
		\frac{h_s(z) + x_l(s)}{h_s(z) - x_l(s)}.
	\end{align}
	Using this to track $z = x_j(t)$, we have for $s \in (0,t)$ that
	\begin{align}
		\label{eq:retraction-de}
		\dot w_j(s)
		= - \sum_{l=1}^{k} I^l_s w_j\frac{w_j + x_l}{w_j - x_l}
		= - \sum_{l=1}^{k} I^l_s \left( w_j + 2x_l + \frac{2x_l^2}{w_j - x_l} \right).
	\end{align}
	Note that $w_j(0) = \Phi_0^{-1}( \Phi_t(x_j(t)) )$
	and $w_j(t) = x_j(t)$.
	Since $h_s'(x_j(t)) = 0$ for $s < t$,
	differentiating the right-hand side of \eqref{eq:inverse-loewner-multinomial}
	twice with respect to $z$ and evaluating at $x_j(t)$
	gives us
	\begin{align*}
		\frac{\dot \kappa(s)}{\kappa(s)}
		&=
		- \sum_{l=1}^{k} I^l_s \left( 1 - \frac{2x_l(s)^2}{(w_j(s) - x_l(s))^2} \right).
	\end{align*}
	As $s \uparrow t$,
	since we have conditioned on $H_j(n)$ we have
	$|w_j(s) - x_j(s)|^{-2} \approx \frac{1}{t-s}$ as $s \uparrow t$.
	Therefore the $l=j$ term above has a non-integrable singularity,
	which we would like to eliminate.
	Subtracting \eqref{eq:multi-preimage-movement} from \eqref{eq:retraction-de}
	gives
	\begin{align}
		\nonumber
		\dot w_j - \dot x_j
		&=
		-\sum_{l} I^l_s \left( w_j + 2x_l + \frac{2x_l^2}{w_j - x_l} \right)
		+\sum_{l \not= j} I^l_s x_j\frac{x_j+x_l}{x_j-x_l}\\
		\label{eq:wx-loewner}
		&=
		- I^j_s \left( w_j + 2x_j + \frac{2x_j^2}{w_j-x_j} \right)
		+ \sum_{l\not=j}
		I^l_s \left(
		\frac{2x_l^2}{(x_j-x_l)(w_j-x_l)} - 1 
		\right)
		(w_j - x_j),
	\end{align}
	and so
	\begin{align*}
		\frac{\dot w_j - \dot x_j}{w_j - x_j}
		&=
		-I^j_s \left( \frac{w_j + 2x_j}{w_j - x_j} + \frac{2x_j^2}{(w_j-x_j)^2} \right)
		+ \sum_{l\not=j} I^l_s \left( \frac{2x_l^2}{(x_j-x_l)(w_j-x_l)} - 1\right).
	\end{align*}
	Hence
	\begin{align*}
		\frac{\dot \kappa}{\kappa} + \frac{\dot w_j - \dot x_j}{w_j-x_j}
		&=
		-I^j_s \left(
		2 + \frac{3 x_j}{w_j-x_j}
		\right)
		+ \sum_{l\not=j} I^l_s \left(
		\frac{2x_l^2}{(w_j-x_l)^2} + \frac{2x_l^2}{(w_j-x_l)(x_j-x_l)} - 2
		\right)\\
		&=
		-2
		-I^j_s \frac{3 x_j}{w_j-x_j}
		+ \sum_{l\not=j} I^l_s \left(
		\frac{2x_l^2}{(w_j-x_l)^2} + \frac{2x_l^2}{(w_j-x_l)(x_j-x_l)}
		\right)\\
		&=
		-2 + \sum_{l=1}^{k} I_s^l Q_s^{j,l}.
	\end{align*}
	There is still a singularity in the $Q_{s}^{j,j}$ term
	as $s \uparrow t$,
	but it is integrable. Therefore
	\begin{align*}
		\log \frac{\lim_{s \uparrow t} \kappa_j(s)( w_j(s) - x_j(s) )}{\kappa(0)( w_j(0) - x_j(0) ) }
		&=
		-2 t + \sum_{l=1}^{k}\int_0^{t} I^l_s Q^{j,l}_s \,\d s.
	\end{align*}
	We claim that the numerator on the left-hand side is equal to 1.
	For $s \in (t-\cparam,t)$,
	$h_s = f^{\theta_n, t-s}$
	as in Definition~\ref{def:particle-map},
	and $x_j(s) = x_j(t) = e^{i\theta_n}$.
	Then $$w_j(s) - x_j(s) = f^{\theta_n,t-s}(e^{i\theta_n}) - e^{i\theta_n}
	= e^{i\theta_n} d(t-s),$$
	and using \eqref{eq:f-second-derivative} we have
	$$\kappa_j(s) = e^{-i\theta_n} (f^{t-s})''(1)
	= e^{-i\theta_n} \frac{1 + d(t-s)}{2\sqrt{1 - e^{-(t-s)}}}.$$
	Since $d(t-s) \sim 2\sqrt{t-s}$ as $s \uparrow t$,
	we get $\kappa_j(s)(w_j(s) - x_j(s)) = (1 + d(t-s))\frac{d(t-s)}{2\sqrt{1 - e^{-(t-s)}}} \to 1$ as $s \uparrow t$,
	so \eqref{eq:unsimplified-second-deriv} follows.
	
	We have already shown in the proof of
	Lemma~\ref{thm:derivative-difference-bound} that
	$A^{-1} \leq |\kappa_j(0)(w_j(0) - x_j(0))| \leq A$,
	and using \eqref{eq:unsimplified-second-deriv-lpm}
	we obtain similar upper and lower bounds on
	$|\bar\kappa_j(0)(\bar w_j(0) - \bar x_j(0))|$.
	Therefore we can use the mean value theorem to bound
	the left-hand side of \eqref{eq:linearised-kappa-w}
	by a constant multiple of
	$\left|\log[\kappa_j(0)(w_j(0) - x_j(0))]
	- \log[\bar\kappa_j(0)(\bar w_j(0) - \bar x_j(0))]\right|$
	and obtain \eqref{eq:linearised-kappa-w}.
\end{proof}

\begin{proof}[Proof of Proposition~\ref{thm:high-probability-integral-bound}]
	We only need to bound the right-hand side of \eqref{eq:linearised-kappa-w}.
	We decompose the integrands into three terms,
	\begin{align*}
		I^l_s Q^{j,l}_s - \bar p^l_s \bar Q^{j,l}_s
		&=
		(I^l_s - p^l_s) \bar Q^{j,l}_s
		+ (p^l_s - \bar p^l_s) \bar Q^{j,l}_s
		+ I^l_s (Q^{j,l}_s - \bar Q^{j,l}_s),
	\end{align*}
	so the $l$th term on the right-hand side of
	\eqref{eq:linearised-kappa-w} is bounded by a constant multiple of
	\begin{align}
		\label{eq:rhs-decomposition}
		\left|
		\int_0^{t} (I^l_s - p^l_s) \bar Q^{j,l}_s \,\d s
		\right|
		+
		\int_0^{t}
		|p^l_s - \bar p^l_s| |\bar Q^{j,l}_s|
		\,\d s
		+
		\int_0^{t}
		|Q^{j,l}_s - \bar Q^{j,l}_s|
		\,\d s.
	\end{align}
	As in the proof of Proposition~\ref{thm:x-gronwall-bound},
	for $s \in [m\cparam,(m+1)\cparam)$
	we can write $I^l_s - p^l_s = I^l_{m\cparam} - p^l_{m\cparam} + O(\cparam)$,
	and so the first term in \eqref{eq:rhs-decomposition} can be written
	\[
	\left|
	\sum_{m=0}^{n-1}
	(I^l_{m\cparam} - p^l_{m\cparam})
	\int_{m\cparam}^{(m+1)\cparam} \bar Q^{j,l}_s\,\d s
	\right|
	+ O(\cparam).
	\]
	
	Since $\bar Q^{j,l}_s$ is deterministic
	and $(I_{m\cparam}^l - p_{m\cparam}^l)_{m \geq 0}$
	are martingale increments
	with respect to $(\theta_m)_{m\geq 0}$,
	this sum resembles a martingale,
	so we could try to control it
	by bounding its second moment.
	However, by conditioning on $H_j(n)$
	we can no longer use adaptedness.
	Fortunately,
	for all $n \leq \rdown{T/\cparam}$,
	$\mathbb{P}(H_j(n)) \geq A^{-1}$
	and so for each $l \in \{1,\dots,k\}$,
	\begin{align*}
		\E\left[ \left|
		\sum_{m=0}^{n-1}
		(I^l_{m\cparam} - p^l_{m\cparam})
		\int_{m\cparam}^{(m+1)\cparam} \bar Q^{j,l}_s\,\d s
		\right|^2 \,\Bigg|\, H_j(n) \right]
		&\leq
		A \E\left[ \left|
		\sum_{m=0}^{n-1}
		(I^l_{m\cparam} - p^l_{m\cparam})
		\int_{m\cparam}^{(m+1)\cparam} \bar Q^{j,l}_s\,\d s
		\right|^2\right]\\
		&\leq
		2A\E\left[
		\sum_{m=0}^{n-1} \left(
		\int_{m\cparam}^{(m+1)\cparam}
		|\bar Q^{j,l}_s|
		\,\d s
		\right)^2
		\right],
	\end{align*}
	where the final inequality comes from the martingale property
	and the fact that $|I^l_{m\cparam}-p^l_{m\cparam}| \leq 1$
	(the factor of 2 is because we bound both the real and imaginary
	parts of $\bar Q^{j,l}_s$ by $|\bar Q^{j,l}_s|$).
	Then $|\bar Q^{j,l}_s| \leq A/\sqrt{t-s}$,
	and so
	$$\sum_{m=0}^{n-1} \E \left(\int_{m\cparam}^{(m+1)\cparam} |\bar Q^{j,l}_s| \,\d s\right)^2
	\leq A \cparam \sum_{m=0}^{n-1}( \sqrt{m+1} - \sqrt{m})^2
	\leq A\cparam \sum_{m=0}^{n - 1} \frac{A}{m+1}
	\leq A \cparam \log\frac{t}{\cparam}.$$
	Define the event
	\begin{align}
		\label{eq:ept}
		E_{\mathrm{p}}^t := \left\{ \sum_{l=1}^{k}\left|
		\int_0^t (I_s^l - p_s^l) \bar Q^{j,l}_s \,\d s
		\right|
		\leq \cparam^{\frac{1}{2} - \frac{2}{R}}
		\right\},
	\end{align}
	then by Markov's inequality,
	$\P(E_{\mathrm{p}}^t) \geq 1 - A \cparam^{4/R}\log(T/\cparam)$
	for sufficiently small $\cparam$.
	
	For the second term in \eqref{eq:rhs-decomposition},
	we use the bound $|\bar Q^{j,l}_s| \leq A /\sqrt{t-s}$
	to get
	\begin{align}
		\label{eq:Q-linearisation-p-term}
		\sum_{l=1}^{k} \int_0^t |p_s^l - \bar p_s^l| |\bar Q_s^{j,l}|
		\leq A \int_0^t \frac{\delta_{\mathrm{p}}(s)}{\sqrt{t-s}}.
	\end{align}
	
	For the final term with $l\not= j$,
	the denominators in \eqref{eq:qjl} stay away from 0
	so we can linearise
	\begin{align*}
		|Q^{j,l}_s - \bar Q^{j,l}_s| &\leq A\left( |(w_j(s) - x_j(s)) - (\bar w_j(s) - \bar x_j(s))| + |x_j(s) - \bar x_j(s)| + |x_l(s) - \bar x_l(s)| \right)\\
		&\leq A |(w_j(s) - x_j(s)) - (\bar w_j(s) - \bar x_j(s))|
		+ A \delta_{\mathrm{x}}(s).
	\end{align*}
	The $l=j$ term is linearised similarly,
	\begin{align*}
		|Q^{j,j}_s - \bar Q^{j,j}_s|
		&\leq
		\frac{3|x_j(s)-\bar x_j(s)|}{|\bar w_j(s) - \bar x_j(s)|}
		+
		\frac{|(w_j(s) - x_j(s)) - (\bar w_j(s) - \bar x_j(s))|}{|w_j(s) - x_j(s)| |\bar w_j(s) - \bar x_j(s)|}\\
		&\leq
		A\frac{\delta_{\mathrm{x}}(s)}{\sqrt{t-s}}
		+
		A \frac{|(w_j(s) - x_j(s)) - (\bar w_j(s) - \bar x_j(s))|}{t - s},
	\end{align*}
	which absorbs the $l \not= j$ terms.
	Note that we have used $|w_j(s) - x_j(s)| \geq A^{-1}\sqrt{t-s}$
	for all $s \leq t$,
	which follows from the displayed part of Definition~\ref{def:Hjn}.
	Combined with \eqref{eq:Q-linearisation-p-term},
	on the event $E_{\mathrm{p}}^t \cap H_j(n)$
	this gives our claimed bound.
\end{proof}

\begin{proof}[Proof of Lemma~\ref{thm:difference-integral-bound}]
	Since the $s$-derivative of $u_s$
	is singular as $s \uparrow t$,
	we will carefully analyse and control the singularity.
	To simplify the notation we will write $w_j = w_j(s)$, etc.
	Subtract \eqref{eq:wx-loewner} from its LPM equivalent to obtain
	\begin{align}
		\nonumber
		\dot u_s
		=
		&-((w_j - x_j) - (\bar w_j - \bar x_j))\\
		\nonumber
		&+ \sum_{l\not=j}\left(
		-I^l_s \frac{2x_l^2}{(x_j-x_l)(w_j-x_l)} (w_j - x_j)
		+\bar p^l_s \frac{2 \bar x_l^2}{(\bar x_j-\bar x_l)(\bar w_j-\bar x_l)}(\bar w_j - \bar x_j)
		\right)\\
		\label{eq:wx-diff}
		&-I^j_s \left( 3x_j + \frac{2x_j^2}{w_j - x_j} \right)
		+\bar p^j_s \left( 3\bar x_j + \frac{2\bar x_j^2}{\bar w_j -  \bar x_j}\right).
	\end{align}
	We will linearise all these terms to write \eqref{eq:wx-diff}
	in the form $\dot u_s = \gamma(s)u_s + \alpha(s)$
	for suitable functions $\gamma$ and $\alpha$.
	The first line on the right-hand side is clearly $-u_s$.
	The $l$th summand on the second line is a difference of the form
	$\bar a \bar b \bar c - abc$,
	so can be linearised as
	$(\bar a - a)\bar b \bar c +a(\bar b - b)\bar c + ab(\bar c - c)$,
	giving
	\begin{align*}
		(\bar p^l_s - I^l_s) \frac{2\bar x_l^2(\bar w_j - \bar x_j)}{(\bar x_j-\bar x_l)(\bar w_j-\bar x_l)}
		&+
		I^l_s \left( \frac{2x_l^2}{(x_j-x_l)(w_j-x_l)} - \frac{2 \bar x_l^2}{(\bar x_j-\bar x_l)(\bar w_j-\bar x_l)} \right)(\bar w_j - \bar x_j)\\
		&+
		I^l_s \frac{2 \bar x_l^2}{(\bar x_j-\bar x_l)(\bar w_j-\bar x_l)}
		((\bar w_j - \bar x_j) - (w_j - x_j)).
	\end{align*}
	If we write $\bar p^l_s = p^l_s + (\bar p^l_s - p^l_s)$,
	then the first term becomes the sum of
	$-(I^l_s - p^l_s) \frac{2\bar x_l^2(\bar w_j - \bar x_j)}{(\bar x_j-\bar x_l)(\bar w_j-\bar x_l)}$
	and something bounded by $A \delta_{\mathrm{total}}(s)$.
	The final term is also bounded by a multiple of $u_s$,
	and so for two bounded functions $A^l_1$ and $A^l_2$
	we can write the $l$th summand as
	\begin{align}
		\label{eq:u-de-lth-summand}
		-(I^l_s - p^l_s)\frac{2\bar x_l^2(\bar w_j - \bar x_j)}{(\bar x_j-\bar x_l)(\bar w_j-\bar x_l)}
		+ A_1^l(s) \delta_{\mathrm{total}}(s)
		+ A_2^l(s) u_s.
	\end{align}
	Similarly we can write the last line of \eqref{eq:wx-diff}
	as
	\begin{align*}
		-(I^j_s - p^j_s)\left(3 \bar x_j + \frac{2 \bar x_j^2}{\bar w_j - \bar x_j} \right)
		+
		(\bar p^j_s - p^j_s)\left(3 \bar x_j + \frac{2 \bar x_j^2}{\bar w_j - \bar x_j} \right)\\
		+
		3 I^j_s \left( \bar x_j - x_j \right)
		+
		I^j_s \left(
		\frac{2\bar x_j^2}{\bar w_j -  \bar x_j}
		-
		\frac{2x_j^2}{w_j-x_j}
		\right).
	\end{align*}
	Linearising the final term on the line above
	by writing $\bar x_j^2 = \bar x_j x_j + \bar x_j(\bar x_j - x_j)$,
	and $x_j^2 = x_j \bar x_j + x_j(x_j - \bar x_j)$,
	we find that it is equal to
	\[
	\frac{2 I^j_s x_j \bar x_j}{(w_j-x_j)(\bar w_j-\bar x_j)} u_s
	+
	I^j_s\left(
	\frac{2\bar x_j}{\bar w_j - \bar x_j}
	+ \frac{2 x_j}{w_j - x_j}
	\right)
	(\bar x_j - x_j).
	\]
	On $H_j(n)$ we have $|w_j(s) - x_j(s)| \geq A^{-1} \sqrt{t-s}$ for all $s < t$
	and $|\bar w_j(s) - \bar x_j(s)| \geq A^{-1}\sqrt{t-s}$,
	so we can bound all the terms involving
	$|\bar p^j_s - p^j_s|$ and $|\bar x_j - x_j|$
	by a constant multiple
	of $\frac{\delta_{\mathrm{total}}(s)}{\sqrt{t-s}}$.
	Therefore the last line of \eqref{eq:wx-diff}
	can be rewritten
	\begin{align}
		\label{eq:u-de-jth-summand}
		-(I^j_s - p^j_s)\left(3 \bar x_j + \frac{2 \bar x_j^2}{\bar w_j - \bar x_j} \right)
		+
		\frac{2 I^j_s x_j \bar x_j}{(w_j-x_j)(\bar w_j-\bar x_j)} u_s
		+
		A^j(s) \frac{\delta_{\mathrm{total}}(s)}{\sqrt{t-s}}
	\end{align}
	for a bounded function $A^j$.
	Substituting \eqref{eq:u-de-lth-summand}
	and \eqref{eq:u-de-jth-summand} back into \eqref{eq:wx-diff},
	and noting that the $A^l_1$ terms can be absorbed
	by the $A^j$ term, we obtain
	\begin{align}
		\nonumber
		\dot u_s
		&=
		\left( \frac{2 I^j_s x_j \bar x_j}{(w_j-x_j)(\bar w_j-\bar x_j)}
		+ \sum_{l\not=j}A^l_2(s) - 1 \right) u_s
		+ A^j(s) \frac{\delta_{\mathrm{total}}(s)}{\sqrt{t-s}}\\
		\label{eq:wx-diff-simplified}
		&\phantom{=}
		-\sum_{l\not=j} \left(
		(I^l_s - p^l_s) \frac{2\bar x_l^2(\bar w_j - \bar x_j)}{(\bar x_j-\bar x_l)(\bar w_j-\bar x_l)}
		\right)
		-
		(I^j_s - p^j_s) \left(3 \bar x_j + \frac{2 \bar x_j^2}{\bar w_j - \bar x_j}\right).
	\end{align}
	Let $$\gamma(s) := \frac{2 I^j_s x_j \bar x_j}{(w_j-x_j)(\bar w_j-\bar x_j)}
	+ \sum_{l\not=j}A^l_2(s) - 1$$
	and let $\alpha(s)$ denote the remaining terms in
	\eqref{eq:wx-diff-simplified},
	$$\alpha(s) := A^j(s)\frac{\delta_{\mathrm{total}}(s)}{\sqrt{t-s}}
	+ \sum_{l=1}^k C^l_s (I^l_s - p^l_s),$$
	where $C^j_s := - (3\bar x_j + \frac{2\bar x_j^2}{\bar w_j - \bar x_j})$
	and $C^l_s := -\frac{2\bar x_l^2 (\bar w_j - \bar x_j)}{(\bar x_j - \bar x_l)(\bar w_j - \bar x_l)}$ for $l \not= j$.
	
	The ODE becomes
	$\dot u_s = \gamma(s) u_s + \alpha(s)$.
	Let $$Y(s) := \int_s^t \alpha(r)\,\d r,$$
	then we can write $u$ using an integrating factor as
	\begin{equation*}
		\label{eq:integral-form-of-u}
		u_s = -Y(s) + \int_s^t Y(r)\gamma(r)\exp\left(-\int_s^r \gamma(v)\,\d v\right)\!\d r,
	\end{equation*}
	and so by the triangle inequality
	\begin{align}
		\label{eq:integral-bound-on-u-unsimplified}
		|u_s| \leq |Y(s)| + \int_s^t |Y(r)||\gamma(r)|\exp\left(-\int_s^r \mathrm{Re}\gamma(v)\,\d v\right)\!\d r.
	\end{align}
	Note that
	since $\arg (w_j(s) - x_j(s)) - \arg x_j(s) \to 0$
	and
	$\arg( \bar w_j(s) - \bar x_j(s) ) - \arg \bar x_j(s) \to 0$
	as $s \uparrow t$,
	the real part of $\gamma(s)$ is bounded below,
	and so \eqref{eq:integral-bound-on-u-unsimplified}
	gives us a simpler bound
	\begin{align*}
		|u_s| \leq |Y(s)| + A\int_s^t |Y(r)||\gamma(r)|\d r.
	\end{align*}
	Conditional on $H_j(n)$ we almost surely have $|\gamma(s)| \leq A/(t-s)$
	for all $s \in [0,t)$,
	so $|Y(s)| \leq X(s)$ for the $X$ defined in
	Lemma~\ref{thm:difference-integral-bound},
	and the result follows.
\end{proof}

\begin{proof}[Proof of Corollary~\ref{thm:u-high-probability-bound}]
	Let $M_s^l := \int_s^t C_r^l (I^l_r - p^l_s)\,\d s$.
	All $k$ terms $|M_s^l|$ can be bounded in exactly the same way,
	so write $M_s = M^j_s$ and assume that
	$X(s) = A\int_s^t \frac{\delta_{\mathrm{total}}(r)}{\sqrt{t-r}}\,\d r + |M(s)|$.
	Let $G(s) := \frac{\delta_{\mathrm{total}}(s)}{\sqrt{t-s}}$.
	We have the bounds
	\begin{align}
		\nonumber
		|u_0| &\leq |X(0)| + A\int_0^t \frac{|X(s)|}{t-s}\d s\\
		\label{eq:u0-decomposed}
		&\leq
		A\underbrace{\int_0^t |G(s)|\,\d s}_{(A1)} + \underbrace{|M_0|}_{(A2)} + A\underbrace{\int_0^t \frac{1}{t-s}\left(\int_s^t |G(r)|\,\d r\right)\d s}_{(A3)}
		+ A\underbrace{\int_0^t \frac{\left| M_s \right|}{t-s}\d s}_{(A4)},
	\end{align}
	and
	\begin{align}
		\nonumber
		\int_0^t \frac{|u_s|}{t-s}\,\d s
		&\leq
		\int_0^t \frac{|X(s)|}{t-s}\,\d s
		+ A \int_0^t \frac{1}{t-s}\int_s^t \frac{|X(r)|}{t-r}\,\d r \d s\\
		\nonumber
		&\leq
		\int_0^t \frac{|X(s)|}{t-s}\,\d s\\
		\label{eq:uintegral-decomposed}
		&\phantom{= }+
		A \underbrace{\int_0^t \frac{1}{t-s} \int_{s}^t \frac{1}{t-r} \int_r^t
			|G(v)|\,\d v \, \d r \, \d s}_{(B1)}
		+
		A \underbrace{\int_0^t \frac{1}{t-s} \int_s^t \frac{|M_r|}{t-r}\,\d r \,\d s}_{(B2)}.
	\end{align}
	We will first bound the terms involving $G$:
	(A1), (A3) and (B1).
	Almost by definition,
	(A1) is bounded by $A \int_0^t \frac{\delta_{\mathrm{total}}(s)}{\sqrt{t-s}}\,\d s
	\leq A \int_0^t \frac{\delta_{\mathrm{total}}(s)}{(t-s)^{51/100}}\,\d s$.
	
	By Tonelli's theorem, (A3) in \eqref{eq:u0-decomposed}
	can be written
	\begin{align*}
		\int_0^t \int_s^t \frac{|G(r)|}{t-s}\,\d r \d s
		&=
		\int_0^t \int_0^r \frac{|G(r)|}{t-s}\,\d s \d r\\
		&=
		\int_0^t |G(r)| \log \frac{t}{t-r}\,\d v\\
		&\leq
		A \int_0^t \delta_{\mathrm{total}}(r) \frac{\log\frac{t}{t-r}}{\sqrt{t-r}}\,\d r\\
		&\leq
		A \int_0^t \frac{\delta_{\mathrm{total}}(r)}{(t-r)^{51/100}}\d r,
	\end{align*}
	where the last inequality holds,
	modifying the constant $A$,
	because
	$\log\frac{t}{t-r} \leq 100 \left(\frac{t}{t-r}\right)^{1/100}$
	for all $r \in (0,t)$.
	
	We can also bound (B1) in \eqref{eq:uintegral-decomposed}
	by repeated applications of Tonelli's theorem:
	\begin{align*}
		\int_0^t \frac{1}{t-s} \int_{s}^{t} \frac{1}{t-r} \int_{r}^t |G(v)|\,\d v\, \d r\, \d s
		&\leq
		A\int_{s=0}^t \frac{1}{t-s} \int_{r=s}^{t} \frac{1}{t-r} \int_{v=r}^t \frac{\delta_{\mathrm{total}}(v)}{\sqrt{t-v}}\,\d v\, \d r\, \d s\\
		&=
		A \int_{v=0}^{t} \int_{r=0}^{v} \frac{\delta_{\mathrm{total}}(v)}{\sqrt{t-v}}\frac{1}{t-r} \int_{s=0}^{r} \frac{1}{t-s}\,\d s \, \d r \, \d v\\
		&=
		A \int_{v=0}^{t} \int_{r=0}^{v} \frac{\delta_{\mathrm{total}}(v)}{\sqrt{t-v}}\frac{1}{t-r}\log\frac{t}{t-r}\,\d r \,\d v\\
		&\leq
		A \int_{v=0}^t \int_{r=0}^v
		\frac{\delta_{\mathrm{total}}(v)}{\sqrt{t-v}}
		\frac{t^{1/100}}{(t-r)^{1 + \frac{1}{100}}} 
		\,\d r \,\d v\\
		&\leq
		A \int_0^t \frac{\delta_{\mathrm{total}}(v)}{(t-v)^{51/100}}\,\d v.
	\end{align*}
	Now we will bound the terms
	in \eqref{eq:u0-decomposed} and \eqref{eq:uintegral-decomposed}
	involving $M_s$:
	(A2), (A4) and (B2).
	First write $I^j_r - p^j_r = I^j_{\cparam\rdown{r/\cparam}} - p^j_r
	= (I^j_{\cparam\rdown{r/\cparam}} - p^j_{\cparam\rdown{r/\cparam}})
	+ (p^j_{\cparam\rdown{r/\cparam}} - p^j_r)$.
	Let $$M_s' := \int_s^t (I^j_{\cparam\rdown{r/\cparam}} - p^j_{\cparam\rdown{r/\cparam}}) C_r\,\d r,$$
	then by Proposition~\ref{thm:multi-fixed-crude}
	we have $|p^j_{\cparam\rdown{r/\cparam}} - p^j_r| \leq A\cparam$
	for all $r \in [0,T]$, so
	\begin{align*}
		|M_s - M_s'| = \left|
		\int_s^t (p^j_{\cparam\rdown{r/\cparam}} - p^j_r) C_r \,\d r
		\right|
		&\leq
		\int_s^t |p^j_{\cparam\rdown{r/\cparam}} - p^j_r| |C_r| \,\d r\\
		&\leq
		A \int_s^t \cparam (t-r)^{-1/2}\,\d r
		\leq A \cparam \sqrt{t-s}.
	\end{align*}
	Therefore in all of (A2), (A4) and (B2)
	we can replace $M$ with $M'$
	by adding an extra error term bounded by $A\sqrt{t}\,\cparam$.
	Recall that $t = n\cparam$.
	If $t-s < 10\cparam$ then we have the deterministic bound
	$|M'_s| \leq A\int_{0}^{t-s}r^{-1/2}\,\d r \leq A \sqrt{t-s}$.
	If $t-s \geq 10\cparam$,
	it is easy to bound the second moment of $|M'_s|$
	since $(I^j_{m\cparam} - p^j_{m\cparam})_{m \geq 1}$
	are martingale increments:
	\begin{align*}
		\E( |M_s'|^2 \,|\, H_j(n) ) \leq A\E( |M_s'|^2 )
		&=
		A\sum_{m=\rdown{s/\cparam}}^{n-1}
		\E( (I^j_{m\cparam} - p^j_{m\cparam})^2 )
		\left(\int_{m\cparam \vee s}^{(m+1)\cparam} |C_r| \,\d r\right)^2\\
		&\leq
		A \sum_{m=\rdown{s/\cparam}}^{n-1}
		\left( \int_{m\cparam \vee s}^{(m+1)\cparam} (n\cparam-r)^{-1/2} \,\d r \right)^2\\
		&=
		A \sum_{m=0}^{n - \rdown{s/\cparam}-1} \left( \int_{m\cparam}^{(m+1)\cparam \wedge (t-s)} r^{-1/2}\,\d r \right)^2\\
		&\leq
		A\cparam \log\left(\frac{t-s}{\cparam}\right) \leq A\cparam \log(1/\cparam).
	\end{align*}
	For $s \in [0,t-10\cparam]$, define the event
	\begin{align*}
		E_s := \{ |M_s'| \leq \cparam^{\frac{1}{2} - \frac{4}{R}} \},
	\end{align*}
	then by Markov's inequality
	\[
	1 - \P(E_s | H_j(n) ) \leq \P( |M_s'|^2 > \cparam^{1 - \frac{8}{R}} | H_j(n) ) \leq \frac{A\cparam\log(1/\cparam)}{\cparam^{1 - \frac{8}{R}}} \leq A \cparam^{\frac{8}{R}} \log(1/\cparam).
	\]
	We will choose a particular sequence $(s_m)_{1 \leq m \leq N_t}$
	in $[0,t-10\cparam]$ such that $\bigcap_{m=1}^{N_t} E_{s_m}$
	holds with high probability,
	and on this event $|M_s'|$ is bounded uniformly on $[0,t-10\cparam]$.
	Let $r_1 = 10\cparam$ and $s_1 = t - r_1$.
	for $m \geq 1$ set $r_{m+1} = r_m + \cparam^{\frac{1}{R}}\sqrt{r_m}$ and $s_m = \max(t-r_m,0)$.
	Then $r_{m+1} - r_m \geq \cparam^{\frac{1}{R}} \sqrt{r_m - r_{m-1}}$,
	and so $r_{m+1} - r_m \geq \cparam^{\frac{1}{2^{m}} + \frac{1}{R}(2 - \frac{1}{2^{m}})} \geq \cparam^{\frac{4}{R}}$ for $m$ greater than some constant $m_{R}$.
	Let $N_t := \min \{m : r_m \geq t \}$
	and it follows from the above that $N_t \leq A\cparam^{-\frac{4}{R}}$.
	Let $s \in (s_{m+1},s_m)$ for some $m \geq 1$,
	then
	\begin{align*}
		|M'_s - M'_{s_m}| \leq \int_{s}^{s_m} |C_r|\,\d r
		&\leq (s_m - s_{m+1}) \sup_{r \in [s_{m+1},s_m]} |C_r|\\
		&\leq A(r_{m+1} - r_m) r_m^{-1/2}\\
		&= A \cparam^{1/R}.
	\end{align*}
	Therefore, on $\bigcap_{m=1}^{N_t} E_{s_m}$,
	$|M'_s| \leq A\cparam^{1/R}$ for all $s \in [0,t-10\cparam]$,
	and by the union bound $\P(\bigcap_{m=1}^{N_t} E_{s_m} | H_j(n))
	\geq 1 - N_t A \cparam^{8/R}\log(1/\cparam) \geq 1 - A \cparam^{4/R} \log(1/\cparam)$.
	
	On this event we can bound (A4) of \eqref{eq:u0-decomposed},
	\begin{align*}
		\int_0^t \frac{|M_s'|}{t-s}\,\d s
		&\leq \int_0^{t-10\cparam} \frac{\cparam^{1/R}}{t-s}\,\d s
		+ \int_{t-10\cparam}^t \frac{A\sqrt{t-s}}{t-s}\,\d s\\
		&\leq A\cparam^{1/R}\log(1/\cparam) + A \cparam^{1/2},
	\end{align*}
	and similarly (B2) of \eqref{eq:uintegral-decomposed} has a bound
	\begin{align*}
		\int_0^t \frac{1}{t-s} \int_{s}^{t} \frac{|M_r'|}{t-r}\,\d r \,\d s
		\leq A \cparam^{1/R}( \log(1/\cparam))^2 + A \cparam^{1/2}.
	\end{align*}
	We can then use these bounds in \eqref{eq:u0-decomposed}
	and \eqref{eq:uintegral-decomposed} to complete our proof.
\end{proof}

\begin{proof}[Proof of Proposition~\ref{thm:p-gronwall-bound}]
	The five results from Lemma~\ref{thm:derivative-difference-bound}
	to Corollary~\ref{thm:u-high-probability-bound}
	show that for a given $n \leq \rdown{T/\cparam}$,
	the bound $\delta_\mathrm{p}(t)
	\leq A \int_0^t \frac{\delta_{\mathrm{total}}(s)}{(t-s)^{51/100}}\d s
	+ A \cparam^{1/R}(\log(1/\cparam))^2$
	holds with high (conditional) probability if $H_j(n)$ occurs.
	Let $B(t)$ be the event that this bound holds at time $t$.
	We will derive the claimed global bound
	from the fact that $\P(B(n\cparam) | H_j(n))$ is close to 1
	for fixed $n$.
	
	For simplicity assume that $\cparam^{-1 + \frac{1}{2R}}$
	and $T/\cparam^{\frac{1}{2R}}$ are integers.
	For $i \in \{0, \dots, T\cparam^{-\frac{1}{2R}}-1\}$ the
	probability that $H_j(n)$ does not hold for any
	$n \in \{ i{\cparam^{-1+ \frac{1}{2R}}} + 1, i{\cparam^{-1+ \frac{1}{2R}}} + 2, \dots, (i+1){\cparam^{-1+ \frac{1}{2R}}} \}$
	is bounded by $(1-p)^{\cparam^{-1 + \frac{1}{2R}}}$
	for a positive constant $p$,
	and so the expected number of such $i$ is less than $T\cparam^{-\frac{1}{2R}}(1-p)^{\cparam^{-1+\frac{1}{2R}}} \ll 1$.
	Therefore by Markov's inequality there exists,
	with very high probability,
	a sequence of times $\tau_0 < \tau_1 < \dots < \tau_{T\cparam^{-\frac{1}{2R}}-1}$
	such that for each $i$,
	$\tau_i \in (i\cparam^{\frac{1}{2R}},(i+1)\cparam^{\frac{1}{2R}}]$,
	$\tau_i/\cparam$ is an integer,
	$H_j(\tau_i/\cparam)$ holds,
	and $\tau_{i+1} - \tau_i \leq 2 \cparam^{\frac{1}{2R}}$.
	Since $\P(B(n\cparam) | H_j(n)) \geq 1 - A \cparam^{4/R}\log(1/\cparam)$
	for a fixed $n$,
	the conditional probability of
	$\bigcap_{i=0}^{T\cparam^{-\frac{1}{2R}}-1}B(\tau_i)$
	given $(\tau_i)_{0 \leq i \leq T\cparam^{-\frac{1}{2R}}-1}$
	is at least $1 - A \cparam^{\frac{7}{2R}}\log(1/\cparam)$.
	Therefore with high probability we have a sequence of times
	$(\tau_i)_{0 \leq i \leq T\cparam^{-\frac{1}{2R}}-1}$
	as above such that $B(\tau_i)$ holds for each $i$.
	
	Let $t \in (\tau_i, \tau_{i+1})$ for some $i \leq T\cparam^{-\frac{1}{2R}}-1$.
	Then by Proposition~\ref{thm:multi-fixed-crude},
	\begin{align*}
		\delta_{\mathrm{p}}(t)
		&= \delta_{\mathrm{p}}(\tau_i)
		+ (\delta_{\mathrm{p}}(t) - \delta_{\mathrm{p}}(\tau_i))\\
		&\leq \delta_{\mathrm{p}}(\tau_i) + A ( \tau_{i+1} - \tau_i )\\
		&\leq \delta_{\mathrm{p}}(\tau_i) + A \cparam^{\frac{1}{2R}}\\
		&\leq A \int_0^{\tau_i}\frac{\delta_{\mathrm{total}}(s)}{(\tau_i-s)^{51/100}}\d s
		+ A \cparam^{\frac{1}{2R}},
	\end{align*}
	where the extra $A \cparam^{1/R}(\log(1/\cparam))^2$ term
	is absorbed by the $A\cparam^{1/2R}$.
	Then since $t > \tau_i$,
	we have $(\tau_i - s)^{-51/100} < (t - s)^{-51/100}$
	for all $s \in (0,\tau_i)$,
	and so
	\[
	\delta_{\mathrm{p}}(t) \leq A \int_0^t \frac{\delta_{\mathrm{total}}(s)}{(t-s)^{51/100}}\d s
	+ A \cparam^{\frac{1}{2R}}
	\]
	for all $t \in [0,T]$, as claimed.
\end{proof}

\begin{proof}[Proof of Theorem~\ref{thm:multi-to-lpm}]
	Propositions~\ref{thm:x-gronwall-bound}
	and \ref{thm:p-gronwall-bound}
	give us $\delta_{\mathrm{total}}(t)
	\leq A \int_0^t \frac{\delta_{\mathrm{total}}(s)}{(t-s)^{51/100}}\,\d s
	+ A\cparam^{1/2R}$
	for all $t \in [0,T]$.
	Using a singular version of Gr\"{o}nwall's inequality
	(see \cite[Corollary 2]{singular-gronwall}),
	this implies
	$\delta_{\mathrm{total}}(T) \leq A \cparam^{1/2R}$ as required.
\end{proof}

\section{Proof of main theorem}

We did most of the hard work in Section~\ref{sec:singular-methods},
and will now bring together the intermediate convergence results
to show convergence of the ALE to the LPM.

\begin{proof}[Proof of Theorem~\ref{thm:ale-lpm}]
	Let $\xi_t$ be the driving function of the ALE,
	and $\bar\mu_t$ the driving measure of the LPM.
	To show that the two converge in distribution,
	we will show that $d_{\mathrm{BW}}(\delta_{\xi_t} \otimes m_{[0,T]}, \bar\mu_t \otimes m_{[0,T]}) \to 0$
	in probability as $\cparam \to 0$.
	
	Note that another way of writing \eqref{eq:aux-rotation}
	for the auxiliary process is
	\begin{align*}
		\Phi_n^* = R_{\delta_1 + \dots + \delta_n} \circ \left( \Phi_0 \circ f_{\theta_1^*-\delta_1} \circ f_{\theta^*_2 - (\delta_1 + \delta_2)}
		\circ \dots \circ f_{\theta^*_n - (\delta_1 + \dots + \delta_n)} \right)
		\circ R_{-(\delta_1 + \dots +\delta_n)}
	\end{align*}
	So let $\xi_t^*$ be the driving measure for the angle sequence
	$(\theta^*_n - (\delta_1 + \dots + \delta_n) )_{n \leq \rdown{T/\cparam}}$.
	By Theorem~\ref{thm:ale-aux-global} we have a coupling between
	$\xi^*$ and $\xi$,
	such that if we define the event
	$E_1 = \{ \tau_D \wedge \couplingfailure > \rdown{T/\cparam} \}$,
	then $\P(E_1) \geq 1 - A\cparam^{1/2}$.
	On $E_1$ note that
	\begin{align*}
		\sup_{t \in [0,T]} |\xi_t - \xi^*_t| \leq \left(\frac{T}{\cparam} + 2\right)D.
	\end{align*}
	
	Next, to pass from the auxiliary model to the multinomial model,
	let $\xi^{\mathrm{multi}}_t$
	be the driving measure of the multinomial model,
	define the event
	$E_2 = \{ \tau_{\not=} > \rdown{T/\cparam} \}$,
	and note that on $E_2$ the models coincide:
	$\xi^{\mathrm{multi}}_t = \xi^*_t$ for all $t \in [0,T]$.
	By Corollary~\ref{thm:aux-multi},
	$\P(E_2) \geq 1 - AT\cparam^{-2}D$.
	
	Finally, by Corollary~\ref{thm:weak-conv-multi-lpm},
	\begin{align*}
		d_{\mathrm{BW}}(\delta_{\xi^{\mathrm{multi}}_t} \otimes m_{[0,T]}, \bar \mu_t \otimes m_{[0,T]}) \to 0
	\end{align*}
	in probability as $\cparam \to 0$.
	
	Then by the triangle inequality,
	\begin{align*}
		d_{\mathrm{BW}}(\delta_{\xi_t} \otimes m_{[0,T]}, \bar\mu_t \otimes m_{[0,T]})
		&\leq
		d_{\mathrm{BW}}(\delta_{\xi_t} \otimes m_{[0,T]}, \delta_{\xi^*_t} \otimes m_{[0,T]})\\
		&\phantom{\leq}+ d_{\mathrm{BW}}(\delta_{\xi^*_t} \otimes m_{[0,T]}, \bar\mu_t \otimes m_{[0,T]}),
	\end{align*}
	and on $E_1 \cap E_2$ this is bounded by
	\begin{align*}
		T \left( \frac{T}{\cparam} + 2 \right)\! D
		+
		d_{\mathrm{BW}}(\delta_{\xi^{\mathrm{multi}}_t} \otimes m_{[0,T]}, \bar \mu_t \otimes m_{[0,T]}).
	\end{align*}
	Since $\cparam^{-1}D = o(1)$,
	this upper bound tends to zero in probability as $\cparam \to 0$,
	and $\P( E_1 \cap E_2 ) \to 1$,
	giving us
	$d_{\mathrm{BW}}(\delta_{\xi_t} \otimes m_{[0,T]}, \bar\mu_t \otimes m_{[0,T]}) \overset{p}{\to} 0$ as required.
\end{proof}

\section*{Acknowledgements}

The initial version of this paper was completed
while studying for a PhD at Lancaster University,
and many thanks are due to my then-supervisor Amanda Turner
for her guidance and patience.
Thanks also to both examiners Dmitry Belyaev and Dmitry Korshunov
who provided useful feedback.
I benefited significantly from discussions with Alan Sola
and Fredrik Viklund.
An anonymous reviewer gave many suggestions which improved the presentation.

\section*{Declarations}

\textbf{Funding:} This work was begun under the Engineering and Physical Sciences Research Council studentship 2118765. Part of it was performed while the author was visiting the Mathematical Sciences Research Institute supported by the National Science Foundation (grant number DMS-1928930), and later while working under the EPSRC grant EP/T028653/1 and Royal Society grant RF{\textbackslash}ERE{\textbackslash}231149.


\end{document}